\pdfoutput=1
\documentclass[11pt]{amsart}
\usepackage{amssymb,latexsym,amsmath,amscd,amsthm,amsfonts, enumerate}
\usepackage{multirow}
\usepackage{color}
\usepackage[all]{xy}
\usepackage{etex}
\usepackage{caption}
\usepackage{soul}
\usepackage{float}
\usepackage{titletoc}
\usepackage[normalem]{ulem}
\usepackage{graphicx}

\usepackage{tikz,tikz-cd}

\usetikzlibrary{snakes}
\usetikzlibrary{shapes.geometric,positioning}
\usetikzlibrary{arrows,decorations.pathmorphing,decorations.pathreplacing}
\tikzset{->-/.style={decoration={markings,mark=at position .5 with {\arrow{>}}},postaction={decorate}}}
\usetikzlibrary{positioning,shapes,shadows,arrows,snakes}

\usepackage[colorlinks=true,pagebackref,hyperindex]{hyperref}
\newcommand\myshade{85}
\colorlet{mylinkcolor}{violet}
\colorlet{mycitecolor}{red}
\colorlet{myurlcolor}{cyan}

\hypersetup{
  linkcolor  = mylinkcolor!\myshade!black,
  citecolor  = mycitecolor!\myshade!black,
  urlcolor   = myurlcolor!\myshade!black,
  colorlinks = true,
}

\numberwithin{equation}{section}

\newtheorem{theorem}{Theorem}[section]
\newtheorem{proposition}[theorem]{Proposition}
\newtheorem{proposition-definition}[theorem]{Proposition-Definition}

\newtheorem{corollary}[theorem]{Corollary}
\newtheorem{lemma}[theorem]{Lemma}

\theoremstyle{definition}
\newtheorem{remark}[theorem]{Remark}

\newtheorem{definition}[theorem]{Definition}

\newtheorem{manualtheoreminner}{Theorem}
\newenvironment{manualtheorem}[1]{%
\renewcommand\themanualtheoreminner{#1}%
  \manualtheoreminner
}{\endmanualtheoreminner}

\newcommand{\thick}{\mathsf{thick}}
\newcommand{\Hom}{\mathrm{Hom}}
\newcommand{\Int}{\mathrm{Int}}

\newcommand{\add}{\mathsf{add}}

\newcommand{\za}{\alpha}
\newcommand{\zb}{\beta}

\newcommand{\zD}{\Delta}

\newcommand{\zg}{\gamma}

\newcommand{\proj}{\operatorname{{\rm proj }}}

\newcommand{\calt}{\mathcal{T}}

\newcommand{\cald}{\mathcal{D}}

\newcommand{\calp}{\mathcal{P}}

\newcommand{\calz}{\mathcal{Z}}

\newcommand{\calk}{\mathcal{K}}
\renewcommand{\P}{\mathrm{P}}

\newcommand{\lp}{{}^{\perp}\calp[>\hspace{-3pt}0]}
\newcommand{\rp}{\calp[<\hspace{-3pt}0]^{\perp}}

\def\s{\stackrel}

\definecolor{dark-green}{RGB}{14,150,2}
\definecolor{red}{RGB}{250,0,0}
\newcommand{\gpoint}{\color{dark-green}{\circ}}
\newcommand{\rpoint}{\color{red}{\bullet}}

\newcommand{\sib}{\color{blue}}
\UseRawInputEncoding
\begin{document}

\title[A geometric realization of silting theory for gentle algebras]{A geometric realization of silting theory for gentle algebras}

\author{Wen Chang}
\address{School of Mathematics and Statistics, Shaanxi Normal University, Xi'an 710062, China}
\email{changwen161@163.com}

\author{Sibylle Schroll}
\address{Department of Mathematics, Universit\"at zu K\"oln, Weyertal 86-90, 50931 K\"oln, Germany and
Institutt for matematiske fag, NTNU, N-7491 Trondheim, Norway}
\email{schroll@math.uni-koeln.de}

\thanks{The first author is supported by Shaanxi Province, Shaanxi Normal University and the NSF of China (Grant 11601295).
The second author was supported by the EPSRC Early Career Fellowship EP/P016294/1.}

\date{\today}


\begin{abstract}
A gentle algebra gives rise to a  dissection of an oriented marked surface with boundary into polygons and  the bounded derived category of the gentle algebra has a geometric interpretation in terms of this surface. In this paper we study  silting theory  in  the bounded derived category of a gentle algebra in terms of its underlying surface. In particular, we show how silting mutation corresponds to the changing of graded arcs and that in some cases silting mutation results in the interpretation of the octahedral axioms in terms of  the flipping of diagonals in a quadrilateral as in the work of
 Dyckerhoff-Kapranov \cite{DK18} in the context of triangulated surfaces. We also show that silting reduction corresponds to the cutting of the underlying surface as is the case for the Calabi-Yau reduction of surface cluster categories as shown by Marsh-Palu \cite{MP14} and Qiu-Zhou \cite{QZ17}.
\end{abstract}

\maketitle
\setcounter{tocdepth}{1}

\tableofcontents

\section*{Introduction}\label{Introductions}

Gentle algebras are classical objects in the representation theory of associative algebras. They were introduced  in the 1980s  as a  generalization of iterated tilted algebras of type $A_n$ \cite{AH81}, and affine type $\widetilde A_n$ \cite{AS87}. Remarkably, they play a role in many other areas of mathematics such as in the context of dimer models \cite{B11, B12}, Lie algebras \cite{HK01}, cluster theory \cite{ABCJP10, BZ11}, discrete derived categories \cite{V03, BPP17} and in connection with homological mirror symmetry \cite{B16, HKK17, BD18, LP20}. In \cite{OPS18, LP20} it is shown that any gentle algebra gives rise to a dissection of an oriented marked surface with boundary, and the objects and morphisms in the derived category of this gentle algebra can been interpreted on the surface model. This geometric interpretation led to the construction of a complete derived invariant \cite{APS19, O19}, building on a geometric interpretation of a derived invariant by Avella-Alaminos and Geiss in \cite{OPS18, LP20}.

Since their inception, gentle algebras  have been a constant object of study and much about their representation theory is known. For example, the indecomposable objects and the Auslander-Reiten sequences in their module category are described in \cite{WW85,AS87,BR87}.
 A geometric model of the  module category of a gentle algebra is given in \cite{BC19}.
The indecomposable objects in their derived category are classified in terms of so-called homotopy string  and homotopy band objects in \cite{BM03} (the terminology homotopy strings and bands is due to \cite{Bob11}). A basis of the morphism space between indecomposable  objects is given  in   \cite{ALP16} in terms of string and band combinatorics. Gentle algebras also have  nice homological properties, for example, they are  Gorenstein \cite{GR05}, the class of gentle algebras is closed under  tilting-cotilting equivalences \cite{S99} and also under derived equivalences \cite{SZ03}.


Silting objects in  triangulated categories were introduced  in \cite{KV88}.
 They play an important role for triangle equivalences, as shown, for example,  in terms of  generalised Morita equivalences in \cite{K98}. Recently, there has been a renewed interest in silting objects  in connection with cluster theory and, in particular, with  $\tau$-tilting theory, see for example the surveys \cite{BY13,H19}.
An important development of silting theory, in particular, in connection or analogy with mutation in cluster theory is made  in \cite{AI12}. In this paper  the notion of silting mutation at silting subcategories of  triangulated categories is introduced, thereby giving a systematic way of constructing new silting subcategories from a given silting subcategory. Furthermore, in analogy with the notion of Calabi-Yau reduction in cluster categories, the notion of a silting reduction of a triangulated category by a pre-silting subcategory is introduced. The connection between Calabi-Yau reduction and silting reduction is made precise  in \cite{IY18, IY20}.

In this paper, we study the silting theory of the derived category of  gentle algebras in terms of their geometric models. More precisely, we show that in analogy with mutation of surface cluster categories,  the silting mutation at a silting subcategory generated by an indecomposable object, is given by changing exactly one graded curve in a dissection of the  surface into polygons. 
In terms of silting reduction we show that in analogy of with Calabi-Yau reduction of surface cluster categories,  the silting reduction at a pre-silting subcategory generated by an indecomposable object associated to a curve on the surface corresponds to cutting the surface along this curve. This connection of mutation in terms of exchanging curves  and reduction in terms of cutting surfaces establishes a geometric analogy  with cluster theory and triangulations of surfaces more generally: Mutation of (surface) cluster algebras corresponds to flipping arcs in an ideal triangulation of a surface \cite{FST08}, the octahedral axiom  in a triangulated category associated with a triangulated surface has been interpreted as flipping the diagonal in a quadrilateral on a triangulated surface \cite{DK18}, and Calabi-Yau reduction of cluster categories has an interpretation in terms of  cutting  surfaces \cite{MP14,QZ17}. Furthermore, flips on decorated marked disks are used in \cite{K08} to construct groupoid structures on pure braid groups and in \cite{Q16,QZ19,KQ20} flips on general decorated marked surfaces appear in the context of 3-Calabi-Yau triangulated categories. Finally, we note that a geometric surface model for discrete derived algebras is given in \cite{B18}, and  mutation is studied in terms of this geometric model.

We now summarise the main results of this paper. For this we briefly set up some notations, which are defined in more detail in Section \ref{Preliminaries}. For a gentle algebra $A$, let $(S,M, \Delta_A)$ be the associated surface of the bounded derived category $D^b(A)$ as constructed in \cite{OPS18} and where $S$ is a compact oriented surface, $M$ a set of marked points both in the boundary and in the interior of $S$ and $\Delta_A$ the admissible dissection of $(S,M)$ induced by $A$.
Then the set of isomorphism classes of indecomposable objects in the bounded homotopy category $K^b(\proj A)$ of complexes of finitely generated projective modules is in bijection with graded curves (up to homotopy) connecting marked points in the boundary, which we refer to as graded arcs, and homotopy classes of certain graded closed curves together with indecomposable representations over  $k[x,x^{-1}]$.
Denote by $\P_\gamma$ the object in $K^b(\proj A)$ corresponding to a graded arc $\gamma$ under this bijection. Recall from \cite{APS19} that any basic silting object in $K^b(\proj A)$ corresponds to an admissible dissection $\Delta$ of graded arcs, where the arcs satisfy certain grading conditions. We denote this silting object by $\P_\Delta$.  In Section \ref{Section:Silting mutations}, given a graded arc $\gamma$ in $\Delta$,  we define the left and right mutation $\mu^\pm_\gamma(\Delta)$ of $\Delta$, which is obtained from $\Delta$ by removing the graded arc $\gamma$ and replacing it with another uniquely defined graded arc in $(S,M)$. Then we show that $\mu^\pm_\gamma(\Delta)$ is again an admissible dissection of $(S,M)$, and that $\P_{\mu^\pm_\gamma(\Delta)}$ is  a silting object in $K^b(\proj A)$. More precisely, we show the following.

\begin{manualtheorem}{1}[Proposition \ref{proposition:compatibility}, Theorem \ref{theorem: compatibilitymutations}]\label{theorem: prelimimary1}
{\it Let $A$ be a gentle algebra with associated surface model $(S,M, \Delta_A)$ and let  $\P_{\zD}$ be a silting object in $ K^b(\proj A)$  with associated graded admissible dissection $\zD$.
 Then for each graded arc $\gamma$ in $\zD$, the object $ \P_{\mu^\pm_\gamma(\Delta)}$   is a silting object in $K^b(\proj A)$.

 Furthermore, $\P_{\mu^\pm_\gamma(\Delta)}$ is the silting mutation of $ \P_\zD$ at the indecomposable direct summand $\P_\gamma$, that is the following  diagram is commutative
    \begin{equation*}
\xymatrix{ \Delta \ar[d]_{}\ar[rr]^{\mu^\pm_{\gamma}}&&
\mu^\pm_{\gamma}\Delta \ar[d]_{}~\\
\P_{\Delta}\ar[u]_{}\ar[rr]^{\mu^\pm_{\P_\gamma}}&&
\P_{\mu^\pm_\gamma(\Delta)}\ar[u]_{}}
\end{equation*}
}
\end{manualtheorem}


As a consequence, we determine when the silting mutation of a tilting object is again a tilting object. We do this by using the combinatorics of the surface in Proposition \ref{proposition:tilting}, and then in Corollary \ref{corollary:tilting} we  restate the result in purely in algebraic terms without reference to the surface.

In order to describe the silting reduction at a pre-silting subcategory, let $\gamma$ be a graded arc in $(S,M,\zD_A)$ such that $\P_\gamma$ is  a pre-silting object in $K^b(\proj A)$. 
Given such an arc $\gamma$, we  cut the surface $(S,M)$ along $\gamma$ and denote by $(S_\gamma, M_\gamma)$ the corresponding cut surface (see Definition \ref{definition:cut surface}).
Now set $\calp = \add(\P_\gamma)$. Following \cite{AI12}, the silting reduction of the category $K^b(\proj A)$ with respect to $\calp$ is the Verdier quotient $\calz_\calp = K^b(\proj A)/\thick\calp$ which is again a triangulated category. Denote by $\overline{\calz_\calp}$  its orbit category with respect to the shift functor. 
In analogy with Calabi-Yau reduction of surface cluster algebras \cite{MP14, QZ17}, the following theorem shows that  $(S_\zg,M_\zg)$ is a geometric realization of  $\overline{\calz_\calp}$.

\begin{manualtheorem}{3}[Theorem \ref{theorem:cutsurface}]\label{thm:prelimilary}
{\it Let $\gamma$ be a graded arc in $(S,M,\zD_A)$ such that $\P_\gamma$ is  a pre-silting object in $K^b(\proj A)$ and set $\calp= \add(P_\gamma)$. 
Then the cut surface $(S_\gamma, M_\gamma)$ is a geometric model of the  orbit category $ \overline{\calz_\calp}$ of the silting reduction of  $K^b(\proj A)$ with respect to $\calp$.   More precisely,
\begin{enumerate}
\item the isomorphism classes of indecomposable objects in $\overline{\calz_\calp}$  are in bijective correspondence with the set of arcs and pairs consisting of a closed curve of winding number zero in $(S_\zg,M_\zg)$ and  a non-zero element in the base field $k$.
\vspace{.2cm}
\item
the dimensions of the morphism spaces in $\overline{\calz_\calp}$ are equal to the number of oriented intersections of the corresponding arcs and  closed curves in $(S_\zg,M_\zg)$.
\end{enumerate}}
\end{manualtheorem}

%

The paper is organised as follows. We begin by recalling some background on marked surfaces, derived categories of gentle algebras, and  silting theory in  Section \ref{Preliminaries}. In Section \ref{section:Distinguished triangles on the surface}, we introduce the notion of a distinguished triangle on the surface, and we prove two key lemmas, which will be frequently used throughout the paper.
In sections \ref{Section:Silting mutations} and   \ref{section:Silting reductions} we respectively study silting mutations and silting reductions in the derived category of gentle algebras as well as their geometric interpretations.

\section*{Acknowledgments}
Wen Chang would like to express his thanks to Yu Qiu and Yu Zhou for useful discussions on marked surfaces and to Dong Yang, Pu Zhang and Bin Zhu for discussions on triangulated categories. Sibylle Schroll would like to thank Claire Amiot and Pierre-Guy Plamandon for initiating the discussion of silting mutations on the surface and Aaron Chan and Hipolito Treffinger for interesting discussions on silting reduction.


\section{Preliminaries}\label{Preliminaries}

 In this paper, all algebras will be assumed to be over a base field $k$ which is algebraically closed. The algebraic closure of $k$ is not a strictly necessary condition, however, it ensures that an  indecomposable $k[x,x^{-1}]$-module  is uniquely determined by its dimension up to isomorphism. This simplifies the correspondence of the oriented intersection numbers and the dimension of the Hom-spaces as stated in subsection \ref{subsection: marked surfaces}. In particular, it allows us to view the `higher' band objects (that is those not sitting at the mouth of a homogeneous tube) as non-primitive graded closed curves.
Arrows in a quiver are composed from left to right as follows: for arrows $a$ and $b$ we write $ab$ for the path from the source of $a$ to the target of $b$.
We adopt the convention that maps are also composed  from left to right, that is if $f: X \to Y$ and $g: Y \to Z$ then $fg : X \to Z$. In general, we consider left modules.
We denote by $\mathbb{Z}$ the set of integer numbers, and by $\mathbb{Z}^*$ the set of non-zero integer numbers.

\subsection{Marked surfaces}\label{subsection: marked surfaces}
In this subsection we recall some concepts about marked surfaces from \cite{OPS18} and \cite{APS19}.
\begin{definition}\cite[Definition 1.7]{APS19}\label{definition:marked surface}
A pair~$(S,M)$ is called a \emph{marked surface} if
\begin{itemize}
 \item $S$ is an oriented surface with non-empty boundary ~$\partial S$;
 \item $M = M^{\gpoint} \cup M^{\rpoint} \cup P^{\rpoint}$ is a finite set of marked points,  where the elements in $M^{\gpoint} \cup M^{\rpoint}$ are in $\partial S$, and the elements in $P^{\rpoint}$, which will be called \emph{punctures}, are in the interior of $S$. Each connected component of~$\partial S$ is required to contain at least one marked point in $M^{\gpoint} \cup M^{\rpoint}$, and the points~$\gpoint$ and~$\rpoint$
       are alternating on  such a component.
       The elements of~$M^{\gpoint}$ and~$M^{\rpoint} \cup P^{\rpoint}$ will be respectively represented by symbols~$\gpoint$ and~$\rpoint$.
\end{itemize}
\end{definition}

On the surface, all curves are considered up to homotopy, and all  intersections of curves are required to be transversal.
 We will call the orientation of the surface the  clockwise orientation, and we call anti-clockwise orientation the opposite orientation. When drawing surfaces in the plane, we will do so that locally,
the orientation of the surface becomes the clockwise orientation of the plane.

\begin{definition}\cite[Definition 1.19]{OPS18}\label{definition:arcs}
Let $(S,M)$ be a marked surface.
\begin{itemize}
 \item An \emph{$\gpoint$-arc} is a non-contractible curve, with endpoints in~$M^{\gpoint}$.
 \item A \emph{loop} is an $\gpoint$-arc which starts and ends at the same $\gpoint$-point.
 \item An \emph{$\rpoint$-arc} is a non-contractible curve, with endpoints in $M^{\rpoint}\cup P^{\rpoint}$.
 \item  An \emph{infinite-arc} is a non-contractible infinite curve which either starts or ends at a $\gpoint$-marked point and wraps around a single puncture in $P^{\rpoint}$ infinitely many times in the anti-clockwise direction or is a curve wrapping such that on either end it wraps infinitely many times around a puncture in $P^{\rpoint}$ in the anti-clockwise direction.
 \item An \emph{admissible arc} is an $\gpoint$-arc or an infinite-arc.
 \item A \emph{closed curve} is a non-contractible curve in the interior of $S$ whose starting point and ending point are equal. A closed curve is called primitive if it is not a non-trivial power of a closed curve in the fundamental group of $S$.
     For a closed curve $\zg$ and a positive integer $n$, denote by $\zg^n$ the closed curve on $S$, which is the $n$-th power of $\zg$ in the fundamental group of $S$.
 \end{itemize}
\end{definition}

For admissible arcs and closed curves in a marked surface $(S,M)$, we now introduce the notion of oriented intersections and  oriented intersection numbers. We will use these to recall in subsection \ref{subsection:derived categories of gentle algebras} some results from \cite{OPS18} on the morphisms in the derived category of a gentle algebra. By abuse of notation we will say that an infinite arc wrapping infinitely many times around a $\rpoint$-puncture is an arc connecting to the corresponding $\rpoint$-point.

\begin{definition}\label{definition:oriented intersections}
For any two admissible arcs or closed curves $\za$ and $\zb$, we denote by $\Int_S(\za,\zb)$ the multiset of \emph{oriented intersections} from $\za$ to $\zb$ in $S$, consisting of

(1) intersections of $\za$ and $\zb$ in the interior of $S$;

(2)  $\rpoint$-punctures with both $\za$ and $\zb$ starting or ending at this puncture;

(3) $\gpoint$-marked points on the boundary, provided both $\za$ and $\zb$ start or end at this point and $\zb$ follows $\za$ in the anti-clockwise orientation at the $\gpoint$-marked points.

Note that in case (3) depending on whether $\za$ and $\zb$ are loops or not and depending on the  anticlockwise order at the $\gpoint$-marked point of the corresponding half-edges, the intersection of $\za$ and $\zb$ at this $\gpoint$-point contributes anything between 0  and  4  elements to $\Int_S(\za,\zb)$.

For example, if $\za$ is a loop while $\zb$ is not, then there are the following three possibilities  for configurations at a boundary intersection $\gpoint$-point $p$, where $_{\bullet}\za$ and $\za_{\bullet}$ represent the two half-edges of $\za$. In the first case (reading left to right), $p$ does not contribute to $\Int_S(\za,\zb)$. In the second case, $p$ contributes once to $\Int_S(\za,\zb)$. In the third case, $p$  contributes twice to $\Int_S(\za,\zb)$.

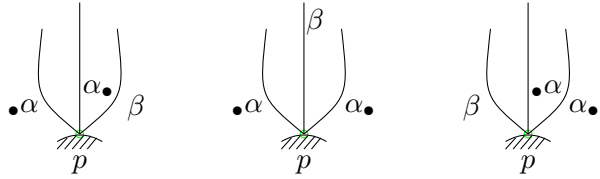
\begin{center}
{\begin{tikzpicture}[scale=0.5]
	\draw[green] (0,0) circle [radius=0.1];

	\draw[bend left](-.6,-.2)to(.6,-.2);
	\draw[-]  (-.33,-0.1)to(-.6,-.4);
	\draw[-]  (-.1,-.05)to(-.4,-.4);
	\draw[-]  (.1,-.05)to(-.2,-.4);
	\draw[-]  (.3,-.05)to(0,-.4);
	\draw[-]  (.45,-.1)to(0.2,-.4);

    \draw (0,-.8) node {$p$};

    \draw (-1.5,.7) node {$_{\bullet}\za$};
	\draw[](0,0)..controls (-1.2,1) .. (-1,2.8);

    \draw (1.5,.7) node {$\zb$};
	\draw[](0,0)..controls (1.2,1) .. (1,2.8);

	\draw[-]  (0,0)to(0,3.5);
    \draw (.5,1.2) node {$\za_{\bullet}$};
	\end{tikzpicture}}
\quad\quad
{\begin{tikzpicture}[scale=0.5]
	\draw[green] (0,0) circle [radius=0.1];

	\draw[bend left](-.6,-.2)to(.6,-.2);
	\draw[-]  (-.33,-0.1)to(-.6,-.4);
	\draw[-]  (-.1,-.05)to(-.4,-.4);
	\draw[-]  (.1,-.05)to(-.2,-.4);
	\draw[-]  (.3,-.05)to(0,-.4);
	\draw[-]  (.45,-.1)to(0.2,-.4);

    \draw (0,-.8) node {$p$};

    \draw (-1.5,.7) node {$_{\bullet}\za$};
	\draw[](0,0)..controls (-1.2,1) .. (-1,2.8);

    \draw (1.5,.7) node {$\za_{\bullet}$};
	\draw[](0,0)..controls (1.2,1) .. (1,2.8);

	\draw[-]  (0,0)to(0,3.5);
    \draw (.3,3) node {$\zb$};
	\end{tikzpicture}}
\quad\quad
{\begin{tikzpicture}[scale=0.5]
	\draw[green] (0,0) circle [radius=0.1];

	\draw[bend left](-.6,-.2)to(.6,-.2);
	\draw[-]  (-.33,-0.1)to(-.6,-.4);
	\draw[-]  (-.1,-.05)to(-.4,-.4);
	\draw[-]  (.1,-.05)to(-.2,-.4);
	\draw[-]  (.3,-.05)to(0,-.4);
	\draw[-]  (.45,-.1)to(0.2,-.4);

    \draw (0,-.8) node {$p$};

    \draw (-1.5,.7) node {$\zb$};
	\draw[](0,0)..controls (-1.2,1) .. (-1,2.8);

    \draw (1.5,.7) node {$\za_{\bullet}$};
	\draw[](0,0)..controls (1.2,1) .. (1,2.8);

	\draw[-]  (0,0)to(0,3.5);
    \draw (.5,1.2) node {$_{\bullet}\za$};
	\end{tikzpicture}}
\end{center}

Let $\sharp\Int_S(\za,\zb)$ be the number of elements in $\Int_S(\za,\zb)$.
We define the \emph{oriented intersection number} $|\Int_S(\za,\zb)|$ from $\za$ to $\zb$ in $S$ as follows.
$$|\Int_S(\za,\zb)|= \left\{\begin{array}{ll}
				\sharp\Int(\za,\zb) &~\textrm{if}~ \nexists~ {\rpoint} \in \Int_S(\za,\zb)~\textrm{and}~ \za\neq \zb~ \textrm{are not powers}\\
~&~ \textrm{of the same primitive closed curve};\\
                mn\cdot[2 \cdot\sharp\Int(\zg,\zg)+1] &~\textrm{if}~
\za=\zg^m ~\textrm{and}~ \zb=\zg^n~ \textrm{for a primitive closed}\\
~&~ \textrm{curve}~ \zg;\\
                2\cdot\sharp\Int(\za,\za)+1 &~\textrm{if}~ \za=\zb~ \textrm{is a $\gpoint$-arc which is not a loop};\\
                2\cdot\sharp\Int(\za,\za) &~\textrm{if}~ \za=\zb~ \textrm{is a loop};\\
				\infty &~\textrm{if}~ \exists~
{\rpoint} \in \Int_S(\za,\zb).
			\end{array}\right.$$
\end{definition}
Note that the oriented intersection number is infinite if and only if $\za$ and $\zb$ are both infinite arcs wrapping around a common  $\rpoint$-puncture.

\begin{definition}\cite[Definition 1.9]{APS19}\label{definition:addmissable dissections}
A collection of $\gpoint$-arcs~$\{\gamma_1, \ldots, \gamma_r\}$ is \emph{admissible} if the only possible intersections of these arcs are at the endpoints, and each subsurface enclosed by the arcs contains at least one $\rpoint$-point from $M^{\rpoint}\cup P^{\rpoint}$. A maximal admissible collection of~$\gpoint$-arcs is called an \emph{admissible $\gpoint$-dissection}.  The notion of~\emph{admissible $\rpoint$-dissection} is defined in a similar way.
\end{definition}

Denote by $g$ the genus of $S$ and by $b$ the number of connected components of $\partial S$. It is proved in \cite{APS19}, see also \cite{PPP18}, that an admissible collection of~$\gpoint$-arcs is an admissible $\gpoint$-dissection if and only if it contains exactly
 \( |M^{\gpoint}|+|P^{\rpoint}|+b+2g-2 \) arcs if and only if each subsurface enclosed by the arcs is a polygon which contains exactly one $\rpoint$-point from $M^{\rpoint}\cup P^{\rpoint}$.

Given an  admissible $\gpoint$-dissection $\Delta$, it is shown in \cite{APS19}, see also \cite{OPS18, PPP18}, that there exists a unique admissible $\rpoint$-dissection $\Delta^*$ (up to homotopy) such that
 each arc of $\Delta^*$ intersects exactly one arc of $\Delta$.
We call $\Delta^*$ the \emph{dual $\rpoint$-dissection} of $\Delta$.
For a fixed admissible $\gpoint$-dissection $\Delta$ and its dual $\rpoint$-dissection $\Delta^*$, unless otherwise stated, all curves on the surface are always assumed to be in minimal position with respect to both these dissections.

\begin{definition}\cite[Definition 2.10]{OPS18}\cite[Definition 2.4]{APS19}\label{definition:gradings}
A \emph{graded curve} $(\gamma,f)$ is an admissible arc or a closed curve $\gamma$, together with a function
   \[
    f: \gamma \cap \Delta^* \longrightarrow \mathbb{Z},
   \]
  where~$\gamma \cap \Delta^*$ is the totally ordered set of intersection points of~$\gamma$ with~$\Delta^*$, where the order is induced by the direction of $\gamma$.
  The function $f$ is defined as follows:
  If~$p$ and~$q$ are in~$\gamma \cap \Delta^*$ and~$q$ is the direct successor of~$p$,
  then~$\gamma$ enters a polygon enclosed by~$\rpoint$-arcs of~$\Delta^*$ via~$p$ and leaves it via~$q$.
  If the point~$\gpoint$ in this polygon is to the left of~$\gamma$, then~$f(q) = f(p)+1$; otherwise,~$f(q) = f(p)-1$. Furthermore, if $\gamma$ is a closed curve and if $\gamma \cap \Delta^* =
  \{p_0, \ldots, p_n\}$, where the index is considered module $n$, then $f$ needs to satisfy
  $f(p_0) = f(p_{n+1})$.
\end{definition}

Note that if $\zg$ is an admissible arc, then there always exists a grading, while if $\zg$ is a closed curve, then there may not exist a grading  of $\zg$ since there might not be a function $f$ satisfying $f(p_0) = f(p_{n+1})$.  In the geometric language of  winding numbers (see \cite[Section 1]{APS19} or \cite{HKK17, LP20}), a closed curve is gradable if and only if its winding number  with respect to the admissible dissection is zero. We also note that any grading on $\zg$ is completely determined by its value on a single element in $\gamma \cap \Delta^*$. If a grading $f$ on $\zg$ exists, then the map $f[n]: l \rightarrow f(l)-n$ with $n\in \mathbb{Z}$ is also a grading on $\zg$, and all gradings on $\zg$ are of this form.
We call  $[1]$ the \emph{shift} of the grading $f$.
As we will see in  Theorem \ref{theorem:object and map in derived category2} (1), the grading shift on curves corresponds to the shift functor in the derived category.

\begin{lemma}\label{lemma:pre}
Let $\zg$ be a primitive closed curve, and let $n$ be any positive integer.
Then $\zg$ is gradable if and only if its power $\zg^n$ is gradable.
Moreover, there is a one-to-one correspondence between the gradings $f$ on $\zg$ and the gradings $f^n$ on $\zg^n$ with $f^n[1]=f[1]^n$.
\end{lemma}
\begin{proof}
Denote by $\gamma \cap \Delta^* = \{p_0, \ldots, p_m\}$, where $p_i=\gamma \cap \gamma_i^*$ for some initial $\rpoint$-arc $\gamma_i^*$ in $\Delta^*, 0 \leq i \leq m$. Then we may assume $\gamma^n \cap \Delta^* = \{p_0^{1}, \ldots, p_m^{1}, \ldots, p_0^{n}, \ldots, p_m^{n}\}$, where $p_i^j=\gamma^n \cap \gamma_i^*$ for any $0\leq i\leq m, 1\leq j\leq n$.
Suppose $\zg$ is gradable and let $f: \gamma \cap \Delta^* \longrightarrow \mathbb{Z}$ be a grading on it. Let
   \[
    f^n: \gamma^n \cap \Delta^* \longrightarrow \mathbb{Z}
   \]
be the function such that $f^n(p_i^j)=f(p_i)$ for any $0\leq i\leq m, 1\leq j\leq n$. Then note that $f^n$ is a well-defined grading on $\zg^n$, and this gives a one-to-one correspondence between the gradings on $\zg$ and the gradings on $\zg^n$. Furthermore, it directly follows from the definitions that  $f^n[1]=f[1]^n$.
\end{proof}

\subsection{Derived categories}\label{subsection:derived categories of gentle algebras}
We recall in this section the derived category of a gentle algebra, and its geometric realization given in \cite{OPS18} using a graded marked surface.

For a finite dimensional algebra $A$, it is well known that the (bounded) derived category $\cald^b(A)$ is triangle equivalent to the homotopy category $K^{-,b}(\proj A)$ of complexes of projective $A$-modules bounded on the right and bounded in homology. In the following, we will not distinguish these two categories, and view the perfect derived category $K^b(\proj A)$ as a full subcategory of $\cald^b(A)$. The derived category of a gentle algebra is well-studied. In particular, the authors in \cite{BM03} classified the indecomposable objects in the category in terms of  \emph{(homotopy) string objects} and \emph{(homotopy) band objects}.
The morphisms between these objects are explicitly described in \cite{ALP16}.

More recently, a one-to-one correspondence between graded marked surfaces \sloppy $(S,M,\zD_A)$ and gentle algebras $A$ has been established  in \cite{OPS18,PPP18}, where $\zD_A$ is the admissible dissection uniquely determined by $A$.
A complete description of the indecomposable objects in $\cald^b(A)$ was given in terms of graded curves in \cite[Theorem 2.12]{OPS18}, and the morphisms are described as graded oriented intersections in \cite[Theorem 3.3]{OPS18}.  We adopt the notation from \cite{OPS18}, where we denote the indecomposable object corresponding to a graded curve $(\zg, f)$ by $\P_{(\zg,f)}$. The precise correspondence is recalled below. Note, however,  that in \cite{OPS18} all closed curves considered were assumed to be primitive. In contrast, in this paper, we allow closed curves which are not necessarily primitive. This has the advantage that we can directly relate intersections of curves corresponding to `higher' band objects with  the dimensions of their morphism spaces. With this in mind we recall the following results.

\begin{theorem}\label{theorem:object and map in derived category}\cite[Theorem 2.12]{OPS18}
Let $A$ be a gentle algebra associated with a graded marked surface
$(S,M,\Delta_A)$. Then
\begin{itemize}
 \item[(1)] the isomorphism classes of  indecomposable string objects $\P_{(\zg,f)}$ in $\cald^b(A)$ are in bijection with graded admissible arcs $(\zg,f)$ on $S$;
 \item[(2)] the isomorphism classes of indecomposable band objects $\P_{(\zg,f,\lambda)}$ in $\cald^b(A)$ are in bijection with triples $(\zg,f,\lambda)$, where $(\zg,f)$ is a graded closed curve on $S$ and $\lambda \in k^*$;
  \item[(3)] under the correspondences in (1) and (2), the indecomposable objects in the perfect category $K^b(\proj A)$ correspond to the graded $\gpoint$-arcs and graded closed curves.
   \end{itemize}
\end{theorem}
\begin{proof}
Statement (1) correspond to \cite[Theorem 2.12]{OPS18} (1).

For the statement (2), by \cite[Theorem 2.12]{OPS18} (2), the isomorphism classes of the indecomposable band objects $\P_{(\zg,f,M)}$ in $\cald^b(A)$ are in bijection with  triples $(\zg,f,M)$, where $(\zg,f)$ is a graded closed curve on $S$ with $\zg$ primitive and $M$ is an isomorphism class of indecomposable $k[x]$-module.  Note that since $k$ is algebraically closed, for any positive integer $n$ and any non-zero $\lambda \in k$, there exists, up to isomorphism, only one indecomposable $n$-dimensional $k[x,x^{-1}]$-module with eigenvalue $\lambda$. On the other hand, Lemma \ref{lemma:pre} shows that any grading $f$ over a primitive closed curve $\zg$ naturally induces a grading $f^n$ over $\zg^n$, and this gives a one-to-one correspondence  between the gradings on $\zg$ and the gradings on $\zg^n$.
Therefore any indecomposable band object $\P_{(\zg,f,M)}$ can be parameterized by a triple $(\zg^n, f^n, \lambda)$, where $(\zg^n, f^n)$ is a graded closed curve, $n$ is the dimension of the indecomposable $k[x,x^{-1}]$-module $M$ whose eigenvalue is $\lambda$.  So we have the stated bijection.

Statement (3) is true because an object in $\cald^b(A)$ is of infinite global dimension if and only if it is of the form $\P_{(\zg,f)}$ with $\zg$ an infinite arc.
\end{proof}

\begin{remark}
We adapt the convention that  if  a result or proof does not depend on the scalar $\lambda$, we will omit it  and we will write $\P_{(\zg,f)}$ instead of $\P_{(\zg,f,\lambda)}$.
\end{remark}

Now we recall how the oriented intersection number of graded curves relates  to  the dimension of the morphism space between the associated indecomposable objects.
Since the results in \cite{OPS18} were given for primitive closed curves, we give a short proof to show that they can be extended to the set-up in this paper, where we also consider non-primitive closed curves.

Before giving the precise statement of the correspondence of morphisms and intersections in the theorem below, let us recall the following from \cite{OPS18}. Each intersection of two curves on the boundary (i.e. at a $\gpoint$-marked point) gives rise to exactly one morphism between the corresponding indecomposable objects, while each intersection in the interior (not at a $\rpoint$-puncture) of $S$ gives rise to two morphisms.
More precisely, let $\P_{(\zg_1,f_1)}$ and $\P_{(\zg_2,f_2)}$ be two indecomposable objects in $\cald^b(A)$ corresponding to admissible arcs or primitive closed curves $\zg_1$ and $\zg_2$, where $f_1$ and $f_2$ are their respective gradings.

Assume first that $\zg_1$ and $\zg_2$ intersect at a $\gpoint$-point $p$ on the boundary as depicted in Figure~\ref{boundary intersection}, where for $i\in \{1,2\}$, $q_i$ is the intersection in $\gamma_i \cap \Delta_A^*$ which is nearest to $p$.
If $f_1(q_1) = f_2(q_2)$, then there is a morphism from $\P_{(\zg_1,f_1)}$ to $\P_{(\zg_2,f_2)}$. Note that there is  no morphism from $\P_{(\zg_2,f_2)}$ to $\P_{(\zg_1,f_1)}[j]$ for any $j\in \mathbb{Z}$ arising from this intersection at $p$.

	\begin{center}
	\begin{figure}[H]
	{\begin{tikzpicture}[scale=0.5]
	\draw[green] (0,0) circle [radius=0.1];
	\draw[](0.1,0)to(4.9,0);
	\draw[red](4.3,-0.5)to(4.6,0.5);
	\draw (2.5,.5) node {$\zg_2$};	
	\draw[](0.1,-0.05)to(4.93,-2.93);
	\draw[red](4.3,-0.5-2.6)to(4.8,0.5-2.9);
	\draw (2.5,-2.2) node {$\zg_1$};	
	\draw (4.5,-2.2) node {$q_1$};
	\draw (5,.5) node {$q_2$};	
	\draw (-1,0) node {$p$};			
	\draw[bend right](-.2,-.6)to(-.2,.6);
	\draw[-]  (-0.1,-.33)to(-.4,-.6);
	\draw[-]  (-.05,-.1)to(-.4,-.4);
	\draw[-]  (-.05,.1)to(-.4,-.2);
	\draw[-]  (-.05,.3)to(-.4,0);
	\draw[-]  (-.1,.45)to(-.4,0.2);

	\end{tikzpicture}}
	\caption{}~\label{boundary intersection}
	\end{figure}
	\end{center}

Assume now that $\zg_1$ and $\zg_2$ intersect at some point $p$ in the interior of the surface, which is not a $\rpoint$-puncture.  Suppose that we have the following local configuration in $S$ depicted in  Figure~\ref{interior intersection} below.
If $f_1(q_1) = f_2(q_2)$, then there
is one morphism from $\P_{(\zg_1,f_1)}$ to $\P_{(\zg_2,f_2)}$, and one morphism from $\P_{(\zg_2,f_2)}$ to $\P_{(\zg_1,f_1)}[1]$ (or to $\P_{(\zg_1,f_1)}[-1]$ depending on the position of the corresponding $\gpoint$-marked point).
	
	\begin{center}
	\begin{figure}[H]
		{\begin{tikzpicture}[scale=0.5]
			\draw[-] (0.15,0.15) -- (4.85,4.85);
			\draw[-] (0.15,4.85) -- (4.85,0.15);

			\draw[-,red] (3.5,0.8) -- (4.8,1);
			\draw[-,red] (3.9,4.8) -- (5.1,4.5);
			
			\draw (-.2,5.2) node {$\zg_1$};
			\draw (-.2,-.2) node {$\zg_2$};
			
			\draw (4.6,5.2) node {$q_2$};
			\draw (4.1,0.4) node {$q_1$};
			\draw (2.5,2) node {$p$};
			\end{tikzpicture}}
			\caption{}\label{interior intersection}
			\end{figure}
	\end{center}

\begin{theorem}\cite[Theorem B]{ALP16}\cite[Theorem 3.3]{OPS18}\label{theorem:object and map in derived category2}
Let $A$ be a gentle algebra associated with a graded marked surface
$(S,M,\Delta_A)$. Then
\begin{itemize}
  \item[(1)] the shift of any indecomposable object corresponds to the shift of the corresponding graded curve, that is
     $$\P_{(\zg,f)}[1]=\P_{(\zg,f[1])},$$
     where $(\zg,f)$ is a graded admissible arc or a graded closed curve.
  \item[(2)] for any two indecomposable objects $\P_{(\zg_1,f_1)}$ and $\P_{(\zg_2,f_2)}$, if we denote $$\underset{i=-\infty}{\overset{\infty}{\bigoplus}}
     \Hom_{\cald^b(A)}(\P_{(\zg_1,f_1)},\P_{(\zg_2,f_2)}[i])$$
  by
  $$\Hom_{\cald^b(A)}^{\bullet}(\P_{(\zg_1,f_1)},\P_{(\zg_2,f_2)}),$$ then
  $$\dim \Hom_{\cald^b(A)}^{\bullet}(\P_{(\zg_1,f_1)},\P_{(\zg_2,f_2)}) =|\Int_S(\zg_1,\zg_2)| .$$
   \end{itemize}
\end{theorem}
\begin{proof}
Statement (1) corresponds \cite[Theorem 3.3]{OPS18} if $(\zg,f)$ is a graded admissible arc or a primitive graded closed curve.
Now let $(\zg,f)=({\zg'}^n,{f'}^n)$ for some primitive graded closed curve $(\zg',f')$. On the one hand, note that $\P_{(\zg,f)}$ and $\P_{(\zg',f')}$ are in the same homogenous tube of $\cald^b(A)$, thus $\P_{(\zg,f)}[1]$ and $\P_{(\zg',f')}[1]=\P_{(\zg',f'[1])}$ are in the same homogenous tube. On the other hand, $\P_{(\zg',f'[1])}$ and $\P_{({\zg'}^n,{f'[1]}^n)}$ are in the same homogenous tube. Furthermore, $f[1]=f'^n[1]={f'[1]}^n$ by Lemma \ref{lemma:pre}, so $\P_{({\zg'}^n,{f'}^n[1])}=\P_{(\zg,f[1])}$.
Therefore $\P_{({\zg},{f})}[1]$ and $\P_{(\zg,f[1])}$ are in the same homogenous tube with the same height. So they coincide, that is $\P_{({\zg},{f})}[1]=\P_{(\zg,f[1])}$.

For  statement (2), we have the following two cases.

Case 1. If both $\zg_1$ and $\zg_2$ are admissible arcs or primitive closed curves, then the statement corresponds to \cite[Theorem 3.3]{OPS18} (1), noticing that when $\zg=\zg_1=\zg_2$, an interior self-intersection  gives rise to two maps from $\P_{(\zg,f_\zg)}$ to $\P_{(\zg,f_\zg)}[\mathbb{Z}]$ for any grading $f_\zg$, and the identity map on $\P_{(\zg,f_\zg)}$ does not correspond an intersection.

Case 2. Suppose that $\zg_1$ and $\zg_2$ are admissible arcs or gradable closed curves, where at least one is a gradable non-primitive closed curve. Without loss of generality assume that $\zg_1$ is such that  $\zg_1={\zg'_1}^m$, for ${\zg'_1}$  a primitive closed curve and $m >1$ and let  $\zg_2={\zg'_2}^n$, where  ${\zg'_2}$ is an admissible arc or a primitive closed curve and $n \geq 1$.
Note that here for convenience, when $\zg_2$ is an admissible arc, we also write it as ${\zg'_2}^n$, where $\zg'_2=\zg_2$ and $n=1$.
Then we have the following two subcases.

If ${\zg'}_1\neq {\zg'}_2$, then on the one hand, $$|\Int_S(\zg_1,\zg_2)|=mn\cdot |\Int_S(\zg'_1,\zg'_2)|,$$
by the definition of the oriented intersection number and the fact that ${\zg'}_1\neq {\zg'}_2$. On the other hand, \cite[Theorem B]{ALP16} says that we have
$$\dim \Hom_{\cald^b(A)}(\P_{(\zg_1,f_1)},\P_{(\zg_2,f_2)})=
mn \cdot \dim \Hom_{\cald^b(A)}(\P_{(\zg'_1,f'_1)},\P_{(\zg'_2,f'_2)}),$$
where $f'_1$ and $f'_2$ are the gradings on $\zg'_1$ and $\zg'_2$ respectively such that $f_1={f'_1}^m$ and $f_2={f'_2}^m$. Therefore we have the wanted equality.

Let ${\zg'}_1={\zg'}_2$, then the statement corresponds to\cite[Theorem 3.3]{OPS18} (2) if $m=n=1$. Otherwise, note that in Definition \ref{definition:oriented intersections}, we define $|\Int_S(\zg_1,\zg_2)|$ as $mn\cdot |\Int_S(\zg'_1,\zg'_2)+1|$, which equals to
$$mn \cdot \dim \Hom_{\cald^b(A)}(\P_{(\zg'_1,f'_1)},\P_{(\zg'_2,f'_2)}).$$
Then by combining the equality
$$\dim \Hom_{\cald^b(A)}(\P_{(\zg_1,f_1)},\P_{(\zg_2,f_2)})=
mn \cdot \dim \Hom_{\cald^b(A)}(\P_{(\zg'_1,f'_1)},\P_{(\zg'_2,f'_2)}),$$
from \cite[Theorem B]{ALP16}, we have the wanted equality.

To sum up, for all the cases, we have proved
  $$\dim \Hom_{\cald^b(A)}^{\bullet}(\P_{(\zg_1,f_1)},\P_{(\zg_2,f_2)}) =|\Int_S(\zg_1,\zg_2)| .$$
\end{proof}

\begin{remark}\label{remark:object and map in derived category2}
Given two graded curves, then adding the identity map if it exists (that is in the case of endomorphisms), all the maps arising from the graded oriented intersections, give a  $k$-basis of the morphism space between the corresponding indecomposable objects, if the curves are not powers of the same primitive closed curve. In contrast, if two curves are powers of the same primitive closed curve, then this does not hold anymore. For example, if $(\zg, f)$ is graded closed curve with $\zg$ primitive and there is no self-intersection on $\zg$, then there exists no intersections between $\zg$ and any power $\zg^n$. However, by \cite[Theorem B]{ALP16}, the dimensions of the morphism spaces between $\P_{(\zg,f)}$ and $\P_{(\zg^n,f^n)}$ are both $n$.
\end{remark}

\subsection{Silting objects}\label{subsection: silting objects}
In the following sections, we recall some background on the silting theory of a general triangulated category and for the bounded derived categories of gentle algebras, in particular.

Let $\calt$ be a triangulated category.
We call a full subcategory $\calp$ in $\calt$ \emph{pre-silting} if $\Hom_{\calt}(\calp,\calp[i])=0$ for all $i>0$. It is \emph{silting} if in addition $\calt=\thick\calp$. An object $P$ of $\calt$ is said to be \emph{pre-silting} if $\add P$ is a pre-silting subcategory and \emph{silting} if $\add P$ is a silting subcategory. We always assume that pre-silting objects as well as silting objects are basic.

Note that for any gentle algebra $A$, the category $K^b(\proj A)$ always has silting objects, while the derived category $\cald^b(A)$ contains silting objects if and only if $gl.dim A < \infty$, that is, $\cald^b(A)$ contains silting objects exactly if it is  triangle equivalent to $K^b(\proj A)$. The silting objects in $K^b(\proj A)$ are completely described as follows.
\begin{theorem}\cite[Theorem 3.2]{APS19}\label{theo::siltingObjects}
 Let $(S,M,\Delta_A)$ be a graded marked surface.
 Let $X$ be a basic silting object in $K^b(\proj A)$.
 Then $X$ is isomorphic to a direct sum $\bigoplus_{i=1}^n \P_{(\gamma_i, f_i)}$,
 where $\zD=\{\gamma_1, \ldots, \gamma_n\}$ is an admissible~$\gpoint$-dissection of~$(S,M)$.
\end{theorem}
Note that for an arbitrary admissible~$\gpoint$-dissection $\zD$,  there currently is no characterisation of when there exists a grading $f$ such that $(\zD,f)$ gives rise to a silting object. So we introduce the following.
\begin{definition}\label{definition:silting dissection}
Let $(S,M,\Delta_A)$ be a graded marked surface. Let $\zD=\{\zg_1,\cdot\cdot\cdot,\zg_n\}$ be an admissible $\gpoint$-dissection and let $f=\{f_{\zg_1},\cdot\cdot\cdot,f_{\zg_n}\}$ be a set of gradings of the arcs in $\zD$. We call
$(\zD,f)$ a \emph{silting dissection} if $$\P_{(\zD,f)}=\bigoplus\limits_{(\zg,f_\zg)\in (\zD,f)}\P_{(\zg,f_\zg)}$$ is a silting object in $K^b(\proj A)$.
\end{definition}

\subsection{Silting mutations}\label{subsection:Mutations of silting objects}
It is shown in \cite{AI12} that it is always possible to mutate silting objects. In the following we recall the definition of (irreducible) silting mutations.
\begin{proposition-definition}\cite{AI12}\label{Definition:mutation of silting objects}
Let $M$ be a (basic) silting object in a triangulated category $\calt$. Let $X$ be an indecomposable summand of $M$: $M=M^\prime\oplus X$.
We define $N_X$ as the mapping cone in the distinguished triangle
\begin{equation}\label{exchange triangle}
X\s{a}\longrightarrow M_X\longrightarrow {N_X}\s{}\longrightarrow {M[1]},
\end{equation}
where $a$ is a minimal left $add(M^\prime)$-approximation. We call $\mu^+_X(M)=M^\prime\oplus N_X$ the \emph{left silting mutation} of $M$ at $X$. The triangle \eqref{exchange triangle} is called the \emph{exchange triangle} of the left silting mutation. The \emph{right silting mutation} $\mu^-_X(M)$ and its exchange triangle are defined dually. Then $\mu^+_X(M)$ and $\mu^-_X(M)$ are both silting objects in $\calt$.
\end{proposition-definition}

\subsection{Silting reductions}\label{subsection:}
Let $\calt$ be a triangulated category with a pre-silting subcategory $\calp$.
We call the Verdier quotient $\calt/\thick\calp$ the \emph{silting reduction} of $\calt$ with respect to $\calp$, where $\thick\calp$ is the smallest subcategory of $\calt$ containing $\calp$, which is closed under direct summands.
The following result shows that under some mild conditions we can realize this Verdier quotient as a subfactor category of $\calt$ introduced in \cite{IY08}. For this, in the notation above, we define the following subcategories of $\calt$
 $$\lp=\{M\in \calt: \Hom_\calt(M,P[i])=0 ~\textrm{for}~\forall~ i > 0, ~\textrm{and}~\forall~ P \in \calp\},$$ $$\rp=\{M\in \calt: \Hom_\calt(P[i],M)=0 ~\textrm{for}~\forall~ i < 0, ~\textrm{and}~\forall~ P \in \calp\}.$$

\begin{theorem}\cite[Theorem 4.2]{IY08} \cite[Lemma 3.5, Theorem 3.6 ]{IY18} \label{theorem:equivalence}
Let $\calt$ be a triangulated category and let $\calp$ be  a pre-silting subcategory of $\calt$. Set
$$\mbox{$\calz=\lp  \cap\rp$}.$$
Then the following hold.

(1) Denote by $\calz_{\calp}$ the factor category of $\calz$ with respect to the subcategory $\calp$.Then as a subfactor category of $\calt$, $\calz_{\calp}$ has the structure of a triangulated category
with respect to the following shift functor and triangles:
\begin{itemize}
\item For $X\in\calz$, we take a triangle
\[\xymatrix{X\ar[r]^{\iota_X}&P_X\ar[r]& X\langle1\rangle\ar[r]& X[1]}\]
with a (fixed) left $\calp$-approximation $\iota_X$.
Then $\langle1\rangle$ gives a well-defined auto-equivalence of ${\calz_\calp}$,
which is the \emph{shift functor} of ${\calz_\calp}$.
\item For a triangle $X\xrightarrow{f} Y\xrightarrow{g} Z\xrightarrow{h}X[1]$
with $X,Y,Z\in\calz$, take the following commutative diagram of triangles:
\begin{equation*}\label{define a triangle}
\xymatrix{
X\ar[r]^f\ar@{=}[d]&Y\ar[r]^g\ar[d]&Z\ar[r]^h\ar[d]^a&X[1]\ar@{=}[d]\\
X\ar[r]^{\iota_X}&P_X\ar[r]&X\langle1\rangle\ar[r]&X[1]
}\end{equation*}
Then we have a complex $X\xrightarrow{\overline{f}}Y\xrightarrow{\overline{g}}Z
\xrightarrow{\overline{a}}X\langle1\rangle$. We define \emph{triangles} in
${\calz_\calp}$ as the complexes which are isomorphic to complexes obtained in this way.
\end{itemize}

(2) We assume the following conditions are satisfied:

(P1) $\calp$ is a functorially finite subcategory of $\calt$;

(P2) $\Hom_\calt(X,\calp[i])=0=\Hom_\calt(\calp,X[i]) ~\textrm{for}~\forall~X\in \calt~\textrm{and}~ i \gg 0$.

Let $\rho\colon\calt\to\calt/\thick\calp$ be the canonical projection functor.
Then the composition $\calz\subset\calt\xrightarrow{\rho}\calt/\thick\calp$ of natural functors induces an equivalence of triangulated categories:
\[\bar{\rho}\colon{\calz_\calp}\stackrel{\simeq}{\longrightarrow}\calt/\thick\calp.\]
\end{theorem}

Note that the two conditions (P1) and (P2) in the above theorem are satisfied for any pre-silting subcategory $\calp$ (arising from a pre-silting object) in the perfect derived category $\calt=K^b(\proj A)$ of a gentle algebra $A$.

For a triangulated category $\calt$, we denote by $\overline{\calt}$ its \emph{orbit category}
 with respect to the shift functor. Denote by $\overline{X}$ an object in $\overline{\calt}$  with  representative $X$. Then for any two objects $\overline{X}$ and $\overline{Y}$,
we have
$$\Hom_{\overline{\calt}}(\overline{X},\overline{Y})=\Hom_{\calt}^{\bullet}(X,Y):= \underset{i=-\infty}{\overset{\infty}{\bigoplus}}\Hom_{\calt}(X,Y[i]),$$
where $X$ and $Y$ are any fixed representatives of $\overline{X}$ and $\overline{Y}$ respectively. Note that $\overline{\calt}$ is an additive category whose morphism space might be infinite dimensional. However, if $\calt = K^b(\proj A)$ then $\Hom_{\overline{\calt}}(\overline{X},\overline{Y})$ is finite dimensional for any $\overline{X},\overline{Y} \in \overline{\calt}$.

\section{Distinguished triangles on the surface}\label{section:Distinguished triangles on the surface}

Let $A$ be a gentle algebra associated with a graded marked surface $(S,M,\zD_A)$.
We will denote by $\cald^b(A)$ its bounded derived category.
Throughout this paper, we will keep these notation.
We begin by  introducing the notion of a distinguished triangle on the surface, which will play a central role throughout the remainder of the paper.

\begin{proposition-definition}[distinguished triangle on the surface]\label{prop-def-distinguished triangle on the surface}
Let $\za$ and $\zb$ be two $\gpoint$-arcs with a common endpoint $q_1$. Let $\zg$ be the smoothing of $\za$ and $\zb$ at $q_1$, where $\zb$ and $\zg$ share a common endpoint $q_2$, and $\zg$ and $\za$ share a common endpoint $q_3$, see Figure \ref{Figure:distinguished triangle on the surface}. Then $\za$ is the smoothing of $\zb$ and $\zg$ at $q_2$, and $\zb$ is the smoothing of $\zg$ and $\za$ at $q_3$.

We call the (generalized) triangle formed by $\za$, $\zb$ and $\zg$ a \emph{distinguished triangle} on the marked surface and we denote it by $\triangle (\za,\zb,\zg)$.

\begin{figure}[H]
\begin{center}
{\begin{tikzpicture}[scale=0.7]
\draw[green] (0,0) circle [radius=0.2];
\draw[green] (0,5) circle [radius=0.2];
\draw[green] (5,0) circle [radius=0.2];

\draw[-] (0,0.2) -- (0,4.8);
\draw[-] (0.2,0) -- (4.8,0);
\draw[-] (0.15,4.85) -- (4.85,0.15);

\draw[-,red] (4.2,-0.2) -- (3.8,.3);
\draw[-,red] (3.8,0.8) -- (4.8,1);

\draw[-,red] (.4,-0.3) -- (.8,.3);
\draw[-,red] (-.4,0.8) -- (.4,.8);

\draw[-,red] (.2,4.3) -- (.9,4.7);
\draw[-,red] (-.4,4.2) -- (.4,4);

\draw (2,3.5) node {$\za$};
\draw (2.5,-.5) node {$\zb$};
\draw (-.5,2.5) node {$\zg$};

\draw (-.5,5.5) node {$q_3$};
\draw (-.5,-.5) node {$q_2$};
\draw (5.5,-.5) node {$q_1$};

\draw (4.3,1.4) node {$s^\za_{q_1}$};
\draw (3.7,-.6) node {$s^\zb_{q_1}$};
\draw (0.8,-.6) node {$s^\zb_{q_2}$};
\draw (-.4,1.3) node {$s^\zg_{q_2}$};
\draw (-.4,3.6) node {$s^\zg_{q_3}$};
\draw (1.25,4.3) node {$s^\za_{q_3}$};
\end{tikzpicture}}

\end{center}
\begin{center}
\caption{distinguished triangle on the surface}\label{Figure:distinguished triangle on the surface}
\end{center}
\end{figure}
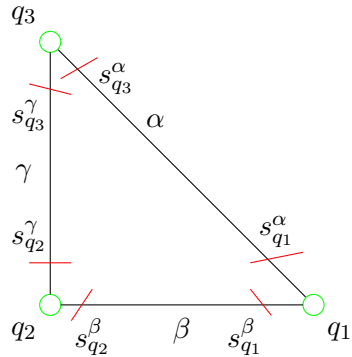

\end{proposition-definition}
\begin{proof}
 Since $\zg$ is the smoothing of $\za$ and $\zb$, there exists neither a  $\rpoint$-puncture nor a boundary component in the interior of $\triangle (\za,\zb,\zg)$. The arc obtained by smoothing $\zb$ and $\zg$ at $q_2$ is homotopic to $\za$. Similarly $\zb$ is homotopic to the smoothing of $\zg$ and $\za$ at  $q_3$.
\end{proof}

\begin{remark}\label{remark:distinguished triangle}
(1) We will refer to a usual triangle on the surface as a real triangle, that is, a triangle such that the sides are all simple and intersect only in marked $\gpoint$-points, and the curve obtained from smoothing the crossings of the sides is contractible. Then a real triangle is a distinguished triangle. However, note that the sides of a distinguished triangle on the surface might self-intersect or intersect with each other in other points than the $q_i, 1\leq i \leq 3$, and that the $q_i$ may coincide with each other. So a distinguished triangle is not necessarily a real triangle.

(2) The name distinguished triangle on the surface is motivated by the following observation: We will see in Lemmas \ref{lemma:equalities} and  \ref{lemma:triangles as triangles}  that for any distinguished triangle on the surface there are gradings on the arcs corresponding to the sides of the distinguished triangle such that the distinguished triangle on the surface gives rise  to a distinguished triangle in the derived category.
\end{remark}

For the remainder of  the paper, we fix the following notation.
Let $\zg$ be an $\gpoint$-arc with endpoints $q_1$ and $q_2$, denote by $s^\zg_{q_i}$ the intersection in $\gamma \cap \Delta_A^*$ which is nearest to $q_i$, see the examples in Figure \ref{Figure:distinguished triangle on the surface}. Note,  in order for  the notation to be well-defined in the case of a loop $\zg$, we treat the unique endpoint of a loop as two distinct endpoints.

\begin{lemma}\label{lemma:equalities}
Assume  that we have a distinguished triangle on the surface as in Figure \ref{Figure:distinguished triangle on the surface}. Let $f_\za$, $f_\zb$ and $f_\zg$ be any gradings on $\alpha, \beta $ and $\gamma$, respectively. Then
\[
f_\za(s^\za_{q_3})-f_\za(s^\za_{q_1})+f_\zb(s^\zb_{q_1})
-f_\zb(s^\zb_{q_2})+f_\zg(s^\zg_{q_2})-f_\zg(s^\zg_{q_3})=1.
\]
\end{lemma}
\begin{proof}
Assume that $\delta,\delta'\in \Delta_A^*$ are $\rpoint$-arcs such that $s^\za_{q_1}=\za\cap\delta$ and $s^\zb_{q_1}=\zb\cap\delta'$. Then there are two cases.

Case I. $\delta\neq \delta'$. Since $\zg$ is the smoothing of $\za$ and $\zb$ at $q_1$, $\zg$  intersects  $\delta$ and $\delta'$ successively. Denote the respective intersection points  by $p$ and $q$.
Since $\zg$ is the smoothing of $\za$ and $\zb$, we have the following equalities

\begin{align*}\label{eq:disttriang}
f_\za(s^\za_{q_3})-f_\za(s^\za_{q_1}) & =  f_\zg(s^\zg_{q_3})-f_\zg(p); \\
f_\zb(s^\zb_{q_2})-f_\zb(s^\zb_{q_1}) & =
f_\zg(s^\zg_{q_2})-f_\zg(q).
\end{align*}

On the other hand, going from $q$ to $p$ along $\zg$, we note that $q_1$ is at the right side of $\zg$. So we have
\[
f_\zg(p)=
f_\zg(q)-1.
\]

Combining these equalities, we get the desired equation.

Case II. $\delta=\delta'$. In this case, $\zg$ does not intersect with $\delta$ close to $q_1$, since by assumption $\zg$ is in minimal position. Denote by $p=\za\cap \sigma$ (resp. $q=\zb\cap \varsigma$) the first intersection  with respect to $q_1$  in $\za\cap \Delta_A^*$ (resp. $\zb\cap \Delta_A^*$) such that the corresponding  arc  in $\Delta_A^*$ also intersects $\gamma$, see Figure \ref{figure:lemma}.
Then $\zg$ successively intersects with $\sigma$ and $\varsigma$, and we denote the intersections by $p'$ and $q'$ respectively.
Similar to the first case, we have equalities
\[
f_\za(s^\za_{q_3})-f_\za(p)=f_\zg(s^\zg_{q_3})-f_\zg(p');
\]
\[
f_\zb(s^\zb_{q_2})-f_\zb(q)=
f_\zg(s^\zg_{q_2})-f_\zg(q').
\]

 Now we denote by $p^0, p^1,\cdots, p^{m+1}$ the ordered intersections in $\za\cap \zD_A^*$ between $s^\za_{q_1}$ and $p$, where $p^0=s^\za_{q_1}, p^{m+1}=p$, and $p^i=\za\cap \sigma^i, 0\leq i\leq m+1$, for some $\rpoint$-arcs $\sigma^i\in \zD_A^*$. Similarly, we denote by $q^0, q^1,\cdots, q^{n+1}$ the ordered intersections in $\zb\cap \zD_A^*$ between $s^\zb_{q_1}$ and $q$, where $q^0=s^\zb_{q_1}, q^{n+1}=q$, and $q^i=\zb\cap \varsigma^i, 0\leq i\leq n+1$. Then by the construction of $\sigma$  and $\varsigma$, we have $m=n$ and $\sigma^i=\varsigma^i$ for any $0\leq i\leq m$. We denote these arcs as $\delta_i$, where in particular $\delta=\Delta_A$.

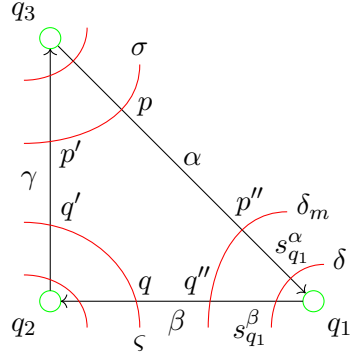
\begin{figure}[H]
\begin{center}
{\begin{tikzpicture}[scale=0.7]
\draw[green] (0,0) circle [radius=0.2];
\draw[green] (0,5) circle [radius=0.2];
\draw[green] (5,0) circle [radius=0.2];

\draw[->] (0,0.2) -- (0,4.8);
\draw[<-] (0.2,0) -- (4.8,0);
\draw[->] (0.15,4.85) -- (4.85,0.15);

\draw (2.7,2.7) node {$\za$};
\draw (2.4,-.4) node {$\zb$};
\draw (-.4,2.3) node {$\zg$};

\draw (1.7,4.8) node {$\sigma$};
\draw (1.7,-.8) node {$\varsigma$};
\draw (5.5,.8) node {$\delta$};
\draw (5,1.8) node {$\delta_m$};

\draw (2.8,.4) node {$q''$};
\draw (1.8,.3) node {$q$};
\draw (3.8,1.9) node {$p''$};
\draw (4.6,1) node {$s_{q_1}^\za$};
\draw (3.8,-.5) node {$s_{q_1}^\zb$};
\draw (1.8,3.7) node {$p$};
\draw (.4,1.8) node {$q'$};
\draw (.4,2.8) node {$p'$};

\draw (-.5,5.5) node {$q_3$};
\draw (-.5,-.5) node {$q_2$};
\draw (5.5,-.5) node {$q_1$};

\draw[red](.7,-.5) to [out=90,in=0] (-.5,.5);
\draw[red](1.7,-.5) to [out=90,in=0] (-.5,1.5);

\draw[red](4.2,-.5) to [out=90,in=180] (5.2,.7);
\draw[red](3,-.5) to [out=90,in=180] (4.5,1.7);

\draw[red](-.5,4.2) to [out=0,in=270] (.7,5.2);
\draw[red](-.5,3) to [out=0,in=270] (1.7,4.5);
\end{tikzpicture}}

\end{center}
\begin{center}
\caption{The second case in Lemma \ref{lemma:equalities} }\label{figure:lemma}
\end{center}
\end{figure}

 In the following, when we talk about the (left or right) position of a $\gpoint$-point with respect to the arcs, we always go from $q_1$ to $q_2$ to $q_3$ to $q_1$.   Note that any two arcs $\delta_i, \delta_{i+1}, 0 \leq i \leq m-1$ are sides of a polygon in $\Delta_A^*$, and if a $\gpoint$-point is on the left (resp. right) side of $\za$, then it is on the right (resp. left) side of $\zb$, since there exist no $\gpoint$-points in the interior of a distinguished triangle on the surface. So we have
\[
f_\za(s^\za_{q_1})-f_\za(p'')=f_\zb(s^\zb_{q_1})-f_\zb(q''),
\]
where $p''=\za\cap \delta_{m}$ and $q''=\zb\cap \delta_{m}$.

On the other hand, $\delta_{m}$, $\sigma$ and $\varsigma$ are sides of a polygon in $\Delta_A^*$.
Now we have three subcases with regards to the position of the $\gpoint$-point $t$ in this polygon.

 Subcase 1. $t$ is on the left side of $\za$ and on the right side of $\zb$ and $\zg$.  Then we have the following equalities:
\[
f_\za(p'')=
f_\za(p)+1;
f_\zb(q)=
f_\zb(q'')-1;
f_\zg(p')=
f_\zg(q')-1.
\]

Subcase 2. $t$ is on the left side of $\zb$ and on  the right side of $\za$ and $\zg$.  Then we have:
\[
f_\za(p'')=
f_\za(p)-1;
f_\zb(q)=
f_\zb(q'')+1;
f_\zg(p')=
f_\zg(q')-1.
\]

Subcase 3. $t$ is on the left side of $\zg$, and on  the right side of $\za$ and $\zb$.  Then we have :
\[
f_\za(p'')=
f_\za(p)-1;
f_\zb(q)=
f_\zb(q'')-1;
f_\zg(p')=
f_\zg(q')+1.
\]

    Finally,  the result follows from  combining the above equalities in each subcase with the three  equalities above.
\end{proof}

The following lemma is a re-statement of Theorem 4.1 in \cite{OPS18} in the context of distinguished triangles on the surface.
The key point of Lemma~\ref{lemma:triangles as triangles} is that the compatibility of the gradings in a distinguished triangle on the surface proved in Lemma~\ref{lemma:equalities} gives rise to three maps in the derived category which form a distinguished triangle in the derived category. This is a key idea underlying many of the proofs in this paper.

\begin{lemma}[distinguished triangles on the surface as distinguished triangles]\label{lemma:triangles as triangles}
Assume we have a distinguished triangle on the surface as in Figure \ref{Figure:distinguished triangle on the surface}. Let $f_\za$, $f_\zb$ and $f_\zg$ be gradings of $\alpha, \beta$ and $\gamma$ respectively such that $f_\za(s^\za_{q_1})=f_\zb(s^\zb_{q_1})$ and $f_\zb(s^\zb_{q_2})=f_\zg(s^\zg_{q_2})$. Then

(1) there are three maps

$$\P_{(\za,f_\za)}\s{a}\longrightarrow \P_{(\zb,f_\zb)},~ \P_{(\zb,f_\zb)}\s{b}\longrightarrow \P_{(\zg,f_\zg)}~ \mbox{ and }~ \P_{(\zg,f_\zg)}\s{c}\longrightarrow {\P_{(\za,f_\za)}[1]} $$
in $\cald^b(A)$ arising from the intersections of the arcs $\za, \zb $ and $\zg$ at $q_1, q_2$ and $q_3$ respectively;

(2) the maps in (1) give rise to the following distinguished triangle in $\cald^b(A)$
\begin{equation}\label{triangles as triangles-exchange triangle}
\P_{(\za,f_\za)}\s{a}\longrightarrow \P_{(\zb,f_\zb)}\s{b}\longrightarrow \P_{(\zg,f_\zg)}\s{c}\longrightarrow {\P_{(\za,f_\za)}[1]}.
\end{equation}
\end{lemma}
\begin{proof}
The first two maps in (1) follow from the hypotheses on the grading, that is from the fact that  $f_\za(s^\za_{q_1})=f_\zb(s^\zb_{q_1})$ and $f_\zb(s^\zb_{q_2})=f_\zg(s^\zg_{q_2})$.
By Lemma \ref{lemma:equalities} the equality $f_\zg(s^\zg_{q_3})=f_\za(s^\za_{q_3})-1$ holds and thus the intersection of $\zg$ and $\za$ at $q_3$ induces the third map. It then follows from \cite [Theorem 4.1]{OPS18} that the triangle \eqref{triangles as triangles-exchange triangle} is distinguished.
\end{proof}

Although  distinguished triangles on the surfaces are in general not   real triangles on the surface, the relations between the diagonals in a quadrilateral formed by two distinguished triangles on the surfaces are like the `flip' of diagonals in (real)  ideal triangulations, see  Figure {\sib \ref{figure:flip}}, where $\zg$ and $\zg_+$ are the diagonals and all triangles are distinguished. The following proposition shows that this kind of flip on the marked surface corresponds to  an instance of  the octahedral axiom (and also to a homotopy push-out/pull-back) in the derived category. A similar correspondence is also established in  \cite{DK18} by Dyckerhoff and Kapranov  for triangulated categories with  2-periodic dg-enhancements associated with a surface model.

\begin{proposition}[flips as  homotopy push-outs/pull-backs]\label{proposition:flip and push-out}
Let $\zg$, $\zg_i$ and $\za_i$, $i=1,2$, be five $\gpoint$-arcs. Assume that they form two distinguished triangles $\triangle(\zg_i,\za_i,\zg)$, $i=1,2$, on the surfaces as in Figure \ref{figure:flip}.

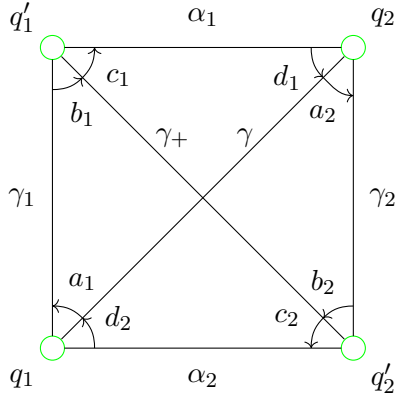
\begin{figure}[H]
\begin{center}
{\begin{tikzpicture}[scale=0.8]
\draw[green] (0,0) circle [radius=0.2];
\draw[green] (0,5) circle [radius=0.2];
\draw[green] (5,5) circle [radius=0.2];
\draw[green] (5,0) circle [radius=0.2];

\draw[-] (0,0.2) -- (0,4.8);
\draw[-] (5,0.2) -- (5,4.8);

\draw[-] (0.2,0) -- (4.8,0);
\draw[-] (0.2,5) -- (4.8,5);

\draw[-] (0.15,0.15) -- (4.85,4.85);
\draw[-] (0.15,4.85) -- (4.85,0.15);

\draw (3.2,3.5) node {$\zg$};
\draw (2,3.5) node {$\zg_+$};

\draw (2.5,5.5) node {$\za_1$};
\draw (2.5,-.5) node {$\za_2$};

\draw (5.5,2.5) node {$\zg_2$};
\draw (-.5,2.5) node {$\zg_1$};

\draw (-.5,5.5) node {$q^\prime_1$};
\draw (-.5,-.5) node {$q_1$};
\draw (5.5,-.5) node {$q^\prime_2$};
\draw (5.5,5.5) node {$q_2$};

\draw (0.5,1.1) node {$a_1$};
\draw (4.5,3.9) node {$a_2$};
\draw (0.5,3.9) node {$b_1$};
\draw (4.5,1.1) node {$b_2$};

\draw (1.1,0.5) node {$d_2$};
\draw (3.9,4.5) node {$d_1$};
\draw (3.9,0.5) node {$c_2$};
\draw (1.1,4.5) node {$c_1$};

\draw[->] (0.5,0.5) to [out=135,in=0] (0,0.7);
\draw[->] (4.5,4.5) to [out=315,in=180] (5,4.2);

\draw[<-] (0.5,4.5) to [out=225,in=0] (0,4.3);
\draw[->] (5,0.7) to [out=180,in=45] (4.5,.5);

\draw[->](4.5,.5) to [out=225,in=90] (4.3,0);
\draw[->](.5,4.5) to [out=45,in=270] (0.7,5);

\draw[->](.7,0) to [out=90,in=-45] (.5,.5);
\draw[->](4.3,5) to [out=270,in=135] (4.5,4.5);

\end{tikzpicture}}
\end{center}
\begin{center}
\caption{Flipping the diagonals in a quadrilateral formed by two distinguished triangles $\triangle(\zg_1,\za_1,\zg)$ and $\triangle(\zg_2,\za_2,\zg)$ on the surface corresponds to a homotopy push-out/pull-back in the derived category.} \label{figure:flip}
\end{center}
\end{figure}

Let $f_\zg$ be any grading of $\zg$. Then the following hold.

(1) The arc obtained by smoothing the crossing of $\zg_1$ and $\za_2$ at $q_1$ is homotopic to the arc obtained by smoothing the crossing of $\za_1$ and $\zg_2$ at $q_2$. Then, denoting this arc by $\zg_+$, the triangles $\triangle(\za_1,\zg_2,\zg_+)$ and $\triangle(\za_2,\zg_1,\zg_+)$ are both distinguished triangles on the surfaces.

(2) For $i=1,2$, there exist unique gradings $f_{\zg_+}$, $f_{\zg_i}$ and $f_{\za_i}$ of the arcs $\zg_+$, $\zg_i$ and $\za_i$ respectively,  such that the respective intersections at $q_i$ and $ q^\prime_i$ yield the following maps in the category $\cald^b(A)$
$$\left\{\begin{array}{rcll}
				a_i:& \P_{(\zg,f_\zg)} & \longrightarrow & \P_{(\zg_i,f_{\zg_i})} \\
				b_i:& \P_{(\zg_i,f_{\zg_i})} & \longrightarrow & \P_{(\zg_+,f_{\zg_+})}  \\
				c_i:&\P_{(\zg_+,f_{\zg_+})} & \longrightarrow & \P_{(\za_i,f_{\za_i})} \\
				d_i:& \P_{(\za_i,f_{\za_i})} & \longrightarrow & \P_{(\zg,f_\zg)}[1],
			\end{array}\right.$$
which are parts of the four distinguished triangles in Figure \ref{figure:push-out}, corresponding to the four distinguished triangles on the surfaces in Figure \ref{figure:flip}.

(3)  Every square in Figure \ref{figure:push-out} is commutative, in particular, $a_1\circ b_1=a_2\circ b_2$ and $(c\circ d :=) c_1\circ d_1=c_2\circ d_2$, and the following triangle is a distinguished triangle in $\cald^b(A)$,
\begin{equation}\label{eq:distinguished triangle}
\P_{(\zg,f_\zg)}\s{(a_1,a_2)}\longrightarrow \P_{(\zg_1,f_{\zg_1})}\oplus \P_{(\zg_2,f_{\zg_2})}\s{\left(
                       \begin{smallmatrix}
                         b_1 \\
                         -b_2\\
                       \end{smallmatrix}
                     \right)}\longrightarrow \P_{(\zg_+,f_{\zg_+})}\s{c\circ d}\longrightarrow {\P_{(\zg,f_\zg)}[1]}.
\end{equation}

(4)  The object $\P_{(\zg_+,f_{\zg_+})}$ is the homotopy push-out of the maps
$$a_i:\P_{(\zg,f_{\zg})} \longrightarrow \P_{(\zg_i,f_{\zg_i})}, i=1,2.$$
The object $\P_{(\zg,f_{\zg})}$ is the homotopy pull-back of the maps
$$b_i: \P_{(\zg_i,f_{\zg_i})} \longrightarrow \P_{(\zg_+,f_{\zg_+})}, i=1,2.$$

\begin{figure}[H]
\begin{center}
{\begin{tikzpicture}[scale=1.6]
\draw (1,0) node {$\P_{(\za_2,f_{\za_2})}$};
\draw (3,0) node {$\P_{(\za_2,f_{\za_2})}$};

\draw (1,1) node {$\P_{(\zg_2,f_{\zg_2})}$};
\draw (3,1) node {$\P_{(\zg_+,f_{\zg_+})}$};
\draw (5,1) node {$\P_{(\za_1,f_{\za_1})}$};
\draw (5,0) node {$\P_{(\zg,f_{\zg})}[1]$};
\draw (-1,1) node {$\P_{(\za_1,f_{\za_1})}[-1]$};

\draw (1,2) node {$\P_{(\zg,f_{\zg})}$};
\draw (3,2) node {$\P_{(\zg_1,f_{\zg_1})}$};
\draw (5,2) node {$\P_{(\za_1,f_{\za_1})}$};
\draw (-1,2) node {$\P_{(\za_1,f_{\za_1})}[-1]$};

\draw (1,3) node {$\P_{(\za_2,f_{\za_2})}[-1]$};
\draw (3,3) node {$\P_{(\za_2,f_{\za_2})}[-1]$};

\draw (2,2.12) node {$a_1$};
\draw (2,1.12) node {$b_2$};

\draw (4,2.12) node {$b_1c_1$};
\draw (4,1.12) node {$c_1$};
\draw (4,0.12) node {$d_2$};

\draw (1.22,.5) node {$b_2c_2$};
\draw (1.12,1.5) node {$a_2$};
\draw (3.12,1.5) node {$b_1$};
\draw (3.12,.5) node {$c_2$};
\draw (5.12,.5) node {$d_1$};
\draw (1.35,2.5) node {\tiny$-d_2[-1]$};
\draw (0.1,2.12) node {\tiny$-d_1[-1]$};
\draw (3.6,2.5) node {\tiny$-d_2[-1]\circ a_1$};
\draw (0.1,0.8) node {\tiny$-d_1[-1]\circ a_2$};
\draw[-] (-1.05,1.2) -- (-1.05,1.8);
\draw[-] (-1,1.2) -- (-1,1.8);

\draw[-] (5.05,1.2) -- (5.05,1.8);
\draw[-] (5,1.2) -- (5,1.8);

\draw[-] (1.5,-0.02) -- (2.4,-0.02);
\draw[-] (1.5,0.03) -- (2.4,0.03);

\draw[-] (1.7,3.02) -- (2.3,3.02);
\draw[-] (1.7,2.97) -- (2.3,2.97);

\draw[->] (1.4,2) -- (2.5,2);
\draw[->] (1.4,1) -- (2.5,1);

\draw[->] (-0.3,2) -- (.5,2);
\draw[->] (-0.3,1) -- (.5,1);

\draw[->] (3.5,2) -- (4.5,2);
\draw[->] (3.5,1) -- (4.5,1);
\draw[->] (3.5,0) -- (4.5,0);

\draw[->] (1,0.7) -- (1,0.2);
\draw[->] (1,1.7) -- (1,1.2);
\draw[->] (1,2.7) -- (1,2.2);

\draw[->] (3,0.7) -- (3,0.2);
\draw[->] (3,1.7) -- (3,1.2);
\draw[->] (3,2.7) -- (3,2.2);
\draw[->] (5,0.7) -- (5,0.2);
\end{tikzpicture}}
\end{center}
\begin{center}
\caption{Flipping the diagonals in a quadrilateral formed by two distinguished triangles $\triangle(\zg_1,\za_1,\zg)$ and $\triangle(\zg_2,\za_2,\zg)$ on the surface corresponds to a  homotopy push-out/pull-back in the derived category.}\label{figure:push-out}
\end{center}
\end{figure}

\end{proposition}
\begin{proof}
(1)  Since $\triangle(\zg_1,\za_1,\zg)$ is a distinguished triangle on the surface, $\zg_1$ is the smoothing of $\zg$ and $\za_1$ at $q_2$.  So the smoothing of $\zg_1$ and $\za_2$ at $q_1$ can be obtained in two steps.  First by smoothing $\za_1$ and $\zg$ at $q_2$, and then smoothing the resulting arc with $\za_2$ at $q_1$. Similarly, the smoothing of $\za_1$ and $\zg_2$ at $q_2$ can be obtained in two steps. Namely,  by first smoothing $\zg$ and $\za_2$ at $q_1$, and then smoothing the resulting arc with $\za_1$ at $q_2$.
That is, they are both the smoothing of the arcs $\za_1$, $\zg$ and $\za_2$ at the two points $q_2$ and $q_1$, up to the order of making the smoothing, which does not change the final result.
So both successive smoothings of crossings give rise to the arc $\zg_+$.
Furthermore, by  construction and by Proposition-Definition \ref{prop-def-distinguished triangle on the surface}, the triangles $\triangle(\za_1,\zg_2,\zg_+)$ and $\triangle(\za_2,\zg_1,\zg_+)$ are both distinguished triangles on the surface.

(2) For $i=1,2$, let $f_{\zg_i}$ and $f_{\za_i}$ be the gradings of the arcs $\zg_i$ and $\za_i$ respectively such that $f_{\zg_i}(s^{\zg_i}_{q_i})=f_{\zg}(s^{\zg}_{q_i})$ and $f_{\za_i}(s^{\za_i}_{q^\prime_i})=f_{\zg_i}(s^{\zg_i}_{q^\prime_i})$.
Since $\triangle (\zg,\zg_i,\za_i), i=1,2$ are distinguished triangles on the surface, by lemma \ref{lemma:triangles as triangles}, we have $f_{\za_1}(s^{\za_1}_{q_2})=f_{\zg}(s^{\zg}_{q_2})-1$ and $f_{\za_2}(s^{\za_2}_{q_1})=f_{\zg}(s^{\zg}_{q_1})-1$.
Let $f_{\zg_+}$ be the grading of $\zg_+$ such that $f_{\zg_+}(s^{\zg_+}_{q^\prime_1})=f_{\zg_1}(s^{\zg_1}_{q^\prime_1})
(=f_{\za_1}(s^{\za_1}_{q^\prime_1}))$.
To sum up, we have the following equalities
\begin{gather}
f_{\zg_+}(s^{\zg_+}_{q'_1})=f_{\za_1}(s^{\za_1}_{q'_1}),
f_{\za_1}(s^{\za_1}_{q_2})=f_{\zg_2}(s^{\zg_2}_{q_2})-1;\label{eq:r1}\\
f_{\zg_1}(s^{\zg_1}_{q'_1})=f_{\zg_+}(s^{\zg_+}_{q'_1}),
f_{\za_2}(s^{\za_2}_{q_1})=f_{\zg_1}(s^{\zg_1}_{q_1})-1.\label{eq:r2}
\end{gather}

Since $\triangle (\zg_+,\za_1,\zg_2)$ and $\triangle (\zg_1,\zg_+,\za_1)$ are distinguished triangles on the surfaces, then by the above equalities and Lemma \ref{lemma:equalities}, we have
\begin{equation}
f_{\zg_+}(s^{\zg_+}_{q^\prime_2})=f_{\zg_2}(s^{\zg_2}_{q^\prime_2}).
\end{equation}
We then have the following maps in $\cald^b(A)$ arising from the intersections $q_i,q^\prime_i, i=1,2$,
$$\left\{\begin{array}{rcll}
				a_i:& \P_{(\zg,f_\zg)} & \longrightarrow & \P_{(\zg_i,f_{\zg_i})} \\
				b_i:& \P_{(\zg_i,f_{\zg_i})} & \longrightarrow & \P_{(\zg_+,f_{\zg_+})}  \\
				c_i:&\P_{(\zg_+,f_{\zg_+})} & \longrightarrow & \P_{(\za_i,f_{\za_i})} \\
				d_i:& \P_{(\za_i,f_{\za_i})} & \longrightarrow & \P_{(\zg,f_\zg)}[1].
			\end{array}\right.$$
Clearly, these gradings $f_{\zg_+}$, $f_{\zg_i}$ and $f_{\za_i}$ are unique with respect to the given grading $f_\zg$.

 Then by Lemma~\ref{lemma:triangles as triangles}, the four distinguished triangles $\triangle(\zg_i,\za_i,\zg)$, $i=1,2$, $\triangle(\za_1,\zg_2,\zg_+)$ and $\triangle(\za_2,\zg_1,\zg_+)$ on the surface give rise to the four distinguished triangles in $\cald^b(A)$ in the diagram in Figure \ref{figure:push-out}.

(3) The fact that the triangle is distinguished follows from \cite[Theorem 4.1]{OPS18} (see also \cite{CPS19}). In particular, we have equalities $a_1\circ b_1=a_2\circ b_2$ and $c_1\circ d_1=c_2\circ d_2$, and every square in Figure \ref{figure:push-out} are commutative.

(4)  From above, the following square is commutative, and it gives rise to the distinguished triangle \eqref{eq:distinguished triangle}, so it is a homotopy cartesian square.
\begin{center}
\begin{tikzpicture}[scale=1.6]
\draw (1,1) node {$\P_{(\zg_2,f_{\zg_2})}$};
\draw (3,1) node {$\P_{(\zg_+,f_{\zg_+})}$};
\draw (1,2) node {$\P_{(\zg,f_{\zg})}$};
\draw (3,2) node {$\P_{(\zg_1,f_{\zg_1})}$};

\draw (2,2.12) node {$a_1$};
\draw (2,1.12) node {$b_2$};
\draw (1.12,1.5) node {$a_2$};
\draw (3.12,1.5) node {$b_1$};

\draw[->] (1.4,2) -- (2.5,2);
\draw[->] (1.4,1) -- (2.5,1);

\draw[->] (1,1.7) -- (1,1.2);
\draw[->] (3,1.7) -- (3,1.2);
\end{tikzpicture}
\end{center}
 Therefore $\P_{(\zg_+,f_{\zg_+})}$ is the homotopy push-out of $a_1$ and $a_2$, and $\P_{(\zg,f_{\zg})}$ is the homotopy pull-back of $b_1$ and $b_2$.
\end{proof}

\section{Silting mutation}\label{Section:Silting mutations}
In this section, we define the mutation of a silting dissection over a graded marked surface $(S,M,\zD_A)$, and we show that it gives rise to  a geometric interpretation of the mutation of silting objects in $K^{b}(\proj A)$.

\subsection{Mutation of silting dissections}\label{Mutation of silting dissections}

Let $(\zg,f_\zg)$ be a graded arc in a silting dissection $(\zD,f)$ (c.f. Definition \ref{definition:silting dissection}) over the graded marked surface. Let $q_1$ and $q_2$ be the endpoints of $\zg$, which may  coincide. For $i=1, 2$, denote by $\zg_i\neq \zg$ the $\gpoint$-arc in $\zD$ with endpoint $q_i$, which is the first anti-clockwise $\gpoint$-arc in $\zD$  following $\zg$ based at $q_i$. Note that $\zg_i$ may not exist. Then we have three possible local configurations, see Figure \ref{Figure:Possible positions of a pre-silting arc}, where we only draw the arcs $\gamma_i$ such that $f_{\zg_i}(s_{q_{i}}^{\zg_i})=f_{\zg}(s_{q_{i}}^{\zg})$.
More precisely, picture (I) means that for both $i=1, 2$, $\zg_i\neq \zg$ and $f_{\zg_i}(s_{q_{i}}^{\zg_i})=f_{\zg}(s_{q_{i}}^{\zg})$. Picture (II) means that there exists exactly one $i$ such that $\zg_i\neq \zg$ and $f_{\zg_i}(s_{q_{i}}^{\zg_i})=f_{\zg}(s_{q_{i}}^{\zg})$. Picture (III) means that there exists no $i$ such that $\zg_i\neq \zg$ and $f_{\zg_i}(s_{q_{i}}^{\zg_i})=f_{\zg}(s_{q_{i}}^{\zg})$.
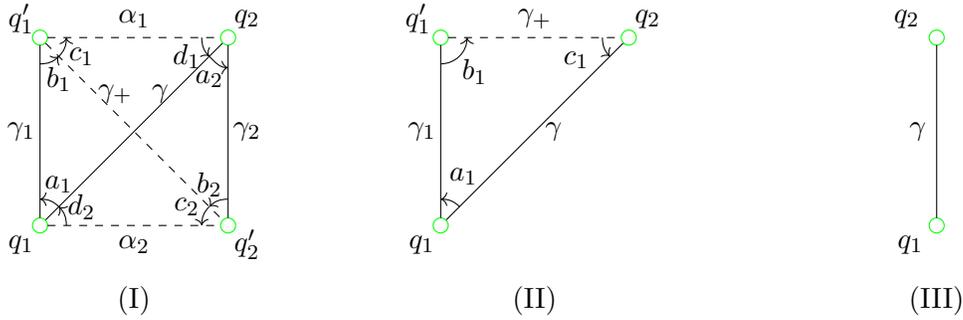
\begin{figure}[H]
\begin{center}
{\begin{tikzpicture}[scale=0.5]
\draw[green] (0,0) circle [radius=0.2];
\draw[green] (0,5) circle [radius=0.2];
\draw[green] (5,5) circle [radius=0.2];
\draw[green] (5,0) circle [radius=0.2];

\draw[-] (0,0.2) -- (0,4.8);
\draw[-] (5,0.2) -- (5,4.8);

\draw[dashed,-] (0.2,0) -- (4.8,0);
\draw[dashed,-] (0.2,5) -- (4.8,5);

\draw[-] (0.15,0.15) -- (4.85,4.85);
\draw[dashed,-] (0.15,4.85) -- (4.85,0.15);

\draw (3.2,3.5) node {$\zg$};
\draw (2,3.5) node {$\zg_+$};

\draw (2.5,5.5) node {$\za_1$};
\draw (2.5,-.5) node {$\za_2$};

\draw (5.5,2.5) node {$\zg_2$};
\draw (-.5,2.5) node {$\zg_1$};

\draw (-.5,5.5) node {$q^\prime_1$};
\draw (-.5,-.5) node {$q_1$};
\draw (5.5,-.5) node {$q^\prime_2$};
\draw (5.5,5.5) node {$q_2$};

\draw (0.5,1.1) node {$a_1$};
\draw (4.5,3.9) node {$a_2$};
\draw (0.5,3.9) node {$b_1$};
\draw (4.5,1.1) node {$b_2$};

\draw (1.1,0.5) node {$d_2$};
\draw (3.9,4.5) node {$d_1$};
\draw (3.9,0.5) node {$c_2$};
\draw (1.1,4.5) node {$c_1$};

\draw[->] (0.5,0.5) to [out=135,in=0] (0,0.7);
\draw[->] (4.5,4.5) to [out=315,in=180] (5,4.2);

\draw[<-] (0.5,4.5) to [out=225,in=0] (0,4.3);
\draw[->] (5,0.7) to [out=180,in=45] (4.5,.5);

\draw[->](4.5,.5) to [out=225,in=90] (4.3,0);
\draw[->](.5,4.5) to [out=45,in=270] (0.7,5);

\draw[->](.7,0) to [out=90,in=-45] (.5,.5);
\draw[->](4.3,5) to [out=270,in=135] (4.5,4.5);
\draw (2.5,-2) node {(I)};
\end{tikzpicture}}
\qquad\qquad
{\begin{tikzpicture}[scale=0.5]
\draw[green] (0,0) circle [radius=0.2];
\draw[green] (0,5) circle [radius=0.2];
\draw[green] (5,5) circle [radius=0.2];

\draw[-] (0,0.2) -- (0,4.8);

\draw[-,dashed] (0.2,5) -- (4.8,5);

\draw[-] (0.15,0.15) -- (4.85,4.85);

\draw (3,2.5) node {$\zg$};
\draw (-.5,2.5) node {$\zg_1$};
\draw (2.5,5.5) node {$\zg_+$};

\draw (-.6,5.5) node {$q^\prime_1$};
\draw (-.5,-.5) node {$q_1$};
\draw (5.5,5.5) node {$q_2$};

\draw (0.6,1.3) node {$a_1$};
\draw[->] (0.5,0.5) to [out=135,in=0] (0,0.7);

\draw[->] (0,4.3) to [out=0,in=270] (0.7,5);
\draw (0.9,4.1) node {$b_1$};

\draw[->](4.3,5) to [out=270,in=135] (4.5,4.5);
\draw (3.6,4.4) node {$c_1$};

\draw (2.5,-2) node {(II)};
\end{tikzpicture}}
\qquad\qquad
{\begin{tikzpicture}[scale=0.5]
\draw[green] (2.5,0) circle [radius=0.2];
\draw[green] (2.5,5) circle [radius=0.2];

\draw[-] (2.5,0.2) -- (2.5,4.8);

\draw (2,2.5) node {$\zg$};
\draw (6,6) node {};
\draw (-1,-1) node {};
\draw (1.7,5.5) node {$q_2$};
\draw (1.8,-.5) node {$q_1$};

\draw (2.5,-2) node {(III)};
\end{tikzpicture}}
\end{center}
\begin{center}
\caption{Possible positions of a pre-silting arc $(\zg,f_\zg)$ illustrating the left silting mutation. }\label{Figure:Possible positions of a pre-silting arc}
\end{center}
\end{figure}

We now describe the  other arcs appearing in Figure \ref{Figure:Possible positions of a pre-silting arc}.
In picture (I), $\za_i$ is the smoothing of $\zg_i$ and $\zg$ at the intersection $q_i$, and $\zg_+$ is the smoothing of $\zg_1$ and $\za_2$ at $q_1$, or equivalently, the smoothing of $\zg_2$ and $\za_1$ at $q_2$ (see Proposition \ref{proposition:flip and push-out}(1)). In picture (II), $\zg_+$ is the smoothing of $\zg_1$ and $\zg$ at $q_1$. So in both pictures, all the triangles are distinguished triangles on the surface.

Dually, we may consider the $\gpoint$-arcs $\zg_i$ in $\zD$ based at $q_i$ which directly follow $\zg$ in the clockwise direction. In this case there are also three possible local configurations which we describe in Figure \ref{Figure:Possible positions of a pre-silting arc-left}.

\begin{figure}[H]
\begin{center}
{\begin{tikzpicture}[scale=0.5]
\draw[green] (0,0) circle [radius=0.2];
\draw[green] (0,5) circle [radius=0.2];
\draw[green] (5,5) circle [radius=0.2];
\draw[green] (5,0) circle [radius=0.2];

\draw[-] (0,0.2) -- (0,4.8);
\draw[-] (5,0.2) -- (5,4.8);

\draw[dashed,-] (0.2,0) -- (4.8,0);
\draw[dashed,-] (0.2,5) -- (4.8,5);

\draw[dashed,-] (0.15,0.15) -- (4.85,4.85);
\draw[-] (0.15,4.85) -- (4.85,0.15);

\draw (3,3.5) node {$\zg_-$};
\draw (2,3.5) node {$\zg$};

\draw (2.5,5.5) node {$\za_2$};
\draw (2.5,-.5) node {$\za_1$};

\draw (5.5,2.5) node {$\zg_2$};
\draw (-.5,2.5) node {$\zg_1$};

\draw (-.5,5.5) node {$q_2$};
\draw (-.5,-.5) node {$q^\prime_1$};
\draw (5.5,-.5) node {$q_1$};
\draw (5.5,5.5) node {$q^\prime_2$};






\draw (2.5,-2) node {(I$^\prime$)};
\end{tikzpicture}}
\qquad\qquad
{\begin{tikzpicture}[scale=0.5]
\draw[green] (0,0) circle [radius=0.2];
\draw[green] (0,5) circle [radius=0.2];
\draw[green] (5,0) circle [radius=0.2];

\draw[-] (0,0.2) -- (0,4.8);

\draw[-,dashed] (0.2,0) -- (4.8,0);

\draw[-] (0.15,4.85) -- (4.85,0.15);

\draw (2,3.5) node {$\zg$};

\draw (2.5,-.5) node {$\zg_-$};

\draw (-.5,2.5) node {$\zg_1$};

\draw (-.6,5.5) node {$q_2$};
\draw (-.5,-.5) node {$q^\prime_1$};
\draw (5.5,-.5) node {$q_1$};

\draw (2.5,-2) node {($\text{II}^\prime$)};
\end{tikzpicture}}
\qquad\qquad
{\begin{tikzpicture}[scale=0.4]
\draw[green] (2.5,0) circle [radius=0.2];
\draw[green] (2.5,5) circle [radius=0.2];

\draw[-] (2.5,0.2) -- (2.5,4.8);

\draw (2,2.5) node {$\zg$};
\draw (6,6) node {};
\draw (-1,-1) node {};
\draw (1.7,5.5) node {$q_2$};
\draw (1.8,-.5) node {$q_1$};

\draw (2.5,-2) node {(III$^\prime$)};
\end{tikzpicture}}
\end{center}
\begin{center}
\caption{Possible positions of a pre-silting arc $(\zg,f_\zg)$ illustrating the right silting mutation.}\label{Figure:Possible positions of a pre-silting arc-left}
\end{center}
\end{figure}
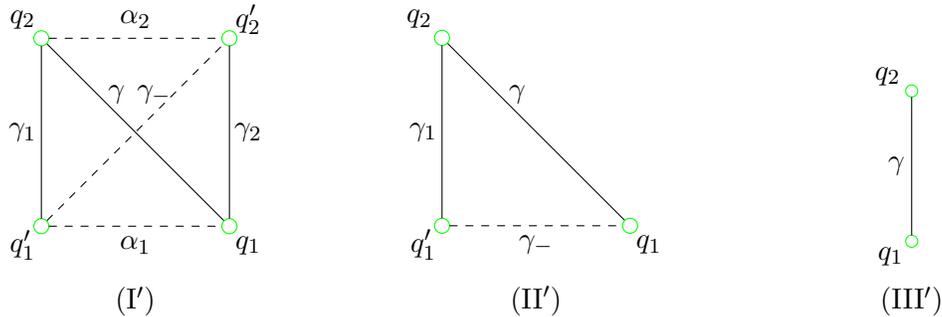

\begin{definition}[Mutation of silting dissections]\label{definition:mutation of silting dissections}
For the three cases in Figure \ref{Figure:Possible positions of a pre-silting arc}, we define the left mutation $\mu^+_{(\zg,f_\zg)}((\Delta,f))$  of the silting dissection $(\zD,f)$ at the graded arc $(\zg,f_\zg)$ as follows:  $$\mu^+_{(\zg,f_\zg)}((\Delta,f))= (\zD^+_\zg,f^+_\zg):=\{\mu^+_{(\zg,f_\zg)}((\za,f_\za)), (\za,f_\za)\in (\Delta,f)\},$$
where
$$\mu^+_{(\zg,f_\zg)}((\za,f_\za))=\left\{\begin{array}{ll}
						(\zg_+,f_{\zg_+}) & \textrm{if}~ (\za,f_\za)=(\zg,f_\zg)~; \\
                        (\za,f_\za) & \textrm{otherwise}~. \\
					\end{array}\right.$$
with the following conventions:

(1) For $(\zg,f_\zg)$ in (I) and (II), the arc $\zg_+$ is as shown in Figure \ref{Figure:Possible positions of a pre-silting arc} (I) and (II) respectively, and the grading $f_{\zg_+}$ is such that $f_{\zg_+}(s_{q'_{1}}^{\zg_+})=f_{\zg_1}(s_{q'_{1}}^{\zg_1})$.

(2) For $(\zg,f_\zg)$ in (III), let $\zg_+=\zg$, and set the grading $f_{\zg_+}=f_{\zg}[1]$.

Dually, we define the right mutation $\mu^-_{(\zg,f_\zg)}((\Delta,f))$.
\end{definition}
\begin{remark}
Note that the setting in case (I) in the above definition coincides with the setting in Proposition \ref{proposition:flip and push-out}. So we have maps $a_i, b_i, c_i,$ and $d_i$ as in Proposition \ref{proposition:flip and push-out} (2). And we also have a distinguished triangle \eqref{eq:distinguished triangle} formed by these maps as shown in the third part of the proposition. On the other hand, the setting in case (II) coincides with the setting in Lemma \ref{lemma:triangles as triangles}. So we have maps $a_i, b_i$ and $c_i$ and a distinguished triangle formed by them as stated in the lemma.
\end{remark}

\begin{proposition}\label{proposition:compatibility}
The mutations $\mu^+_{(\zg,f_\zg)}((\Delta,f))$ and $\mu^-_{(\zg,f_\zg)}((\Delta,f))$ of a silting dissection $(\zD,f)$ are again silting dissections.
\end{proposition}
\begin{proof}
We only prove the case for the left mutation $\mu^+_{(\zg,f_\zg)}((\Delta,f))$, the proof for the right mutation $\mu^-_{(\zg,f_\zg)}((\Delta,f))$ is analogous.

By the construction of $\zD^+_\zg$, in all three cases (I), (II), and (III), the only possible intersections between the arcs in $\zD^+_\zg$ are at the endpoints.  The quadrilateral obtained from smoothing the endpoint crossings of the arcs $\alpha_1, \gamma_2, \alpha_2, $ and $ \gamma_1 $ in case (I) is such that the collection $\zD^+_\zg$ does not cut out any subsurface not containing a $\rpoint$-marked points in their interior. Thus changing $\gamma$ to $\gamma_+$ in $\zD$ implies that  $\zD^+_\zg$ is again an admissible collection. Similarly, in case (II)  the collection $\zD^+_\zg$ is admissible and in case (III) we have $\zD^+_\zg = \zD$.   Furthermore, since  $\zD^+_\zg$ and $\zD$  have the same number of arcs, $\zD^+_\zg$ is maximal, and therefore it is an admissible $\gpoint$-dissection.

Now to prove that $\mu^+_{(\zg,f_\zg)}((\Delta,f))$ is a silting dissection, we only need to show that $\Hom_{\cald^b(A)}(\P_{(\za,f_\za)},\P_{(\zb,f_\zb)}[m])=0$, for any positive integer $m$ and any two graded arcs $(\za,f_\za)$, $(\zb,f_\zb)$ in   $\mu^+_{(\zg,f_\zg)}((\Delta,f))$. This is clearly true when neither $(\za,f_\za)$ nor $(\zb,f_\zb)$ coincide with $(\zg_+,f_{\zg_+})$, since then $(\za,f_\za)$ and $(\zb,f_\zb)$  both are in the silting dissection $(\zD,f)$. Now let one of $(\za,f_\za)$ and $(\zb,f_\zb)$ coincide with $(\zg_+,f_{\zg_+})$. We assume $(\zg_+,f_{\zg_+})=(\za,f_\za)$ in the following.  The case  $(\zg_+,f_{\zg_+})=(\zb,f_\zb)$ is proved analogously.

Case 1. For  $\zg$ as in Figure \ref{Figure:Possible positions of a pre-silting arc} (I), suppose that there exists a graded arc $(\zb,f_\zb)$ in $\mu^+_{(\zg,f_\zg)}((\Delta,f))$ and a non-zero map $d$ from $\P_{(\zg_+,f_{\zg_+})}$ to $\P_{(\zb,f_\zb)}[m]$. Then the map $d$ should arise from the endpoints of $\zg_+$, i.e., $q^\prime_1$ or $q^\prime_2$.
Note that
if $(\zb,f_\zb)=(\zg_+,f_{\zg_+})$, then $\zg_+$ is a loop. So $q^\prime_1  =q^\prime_2$, and
$f_{\zg_+}(s^{\zg_+}_{q^\prime_1})=f_{\zg_+}(s^{\zg_+}_{q^\prime_2})-m$
or $f_{\zg_+}(s^{\zg_+}_{q^\prime_1})=f_{\zg_+}(s^{\zg_+}_{q^\prime_2})+m.$
On the other hand, by the proof of the Proposition \ref{proposition:flip and push-out} (2), we have
$$f_{\zg_1}(s^{\zg_1}_{q^\prime_1})=f_{\zg_+}(s^{\zg_+}_{q^\prime_1}) \mbox{ and }
f_{\zg_2}(s^{\zg_2}_{q^\prime_2})=f_{\zg_+}(s^{\zg_+}_{q^\prime_2}).$$
Thus
$f_{\zg_1}(s^{\zg_1}_{q^\prime_1})=f_{\zg_2}(s^{\zg_2}_{q^\prime_2})-m$
or $f_{\zg_1}(s^{\zg_1}_{q^\prime_1})=f_{\zg_2}(s^{\zg_2}_{q^\prime_2})+m$. This implies that  there will be  a non-zero map from $\P_{(\zg_1,f_{\zg_1})}$ to $\P_{(\zg_2,f_{\zg_2})}[m]$ or from $\P_{(\zg_2,f_{\zg_2})}$ to $\P_{(\zg_1,f_{\zg_1})}[m]$ arising from the intersection $q^\prime_1=q^\prime_2$. A contradiction since both $(\zg_1,f_{\zg_1})$ and $(\zg_2,f_{\zg_2})$ belong to the silting dissection $(\zD,f)$. Thus we assume that $\zb\neq \zg_+$ and we consider the following two subcases.

Subcase 1.1. Suppose that the map $d$ arises from the intersection at $q^\prime_1$. Then at $q^\prime_1$ the arcs  $\zg_1$, $\zg_+$, $\zb$ are ordered in the anticlockwise order as $\zg_1$, $\zg_+$, $\zb$.
On the other hand, note that there is a non-zero map $b_1$ from $\P_{(\zg_1,f_{\zg_1})}$ to $\P_{(\zg_+,f_{\zg_+})}$ arising from the intersection $q^\prime_1$. So we have $b_1\circ d$ from $\P_{(\zg_1,f_{\zg_1})}$ to $\P_{(\zb,f_\zb)}[m]$, which is a non-zero map arising from $q^\prime_1$. But then we have a contradiction, noticing that both graded arcs $(\zg_1,f_{\zg_1})$ and $(\zb,f_{\zb})$ belong to the silting dissection $(\zD,f)$.

Subcase 1.2. If the map $d$ arises from $q^\prime_2$, then the proof is similar to the first subcase.

To sum up, for both subcases, we have proved $\Hom_{\cald^b(A)}(\P_{(\zg_+,f_{\zg_+})},\P_{(\zb,f_\zb)}[m])=0$, for any $(\zb,f_\zb)$ in $(\zD,f)$.

Case 2. For  $\zg$ as in Figure \ref{Figure:Possible positions of a pre-silting arc} (II), suppose that there exists a graded arc $(\zb,f_\zb)$ in $\mu^+_{(\zg,f_\zg)}((\Delta,f))$ and a non-zero map $d$ from $\P_{(\zg_+,f_{\zg_+})}$ to $\P_{(\zb,f_\zb)}[m]$. Similar to the first case, $\zb$ is different from $\zg_+$, and we have again two subcases.

Subcase 2.1. If the map $d$ arises from $q^\prime_1$, then the proof is exactly the same as the proof of subcase 1.1.

Subcase 2.2. If the map $d$ arises from an intersection at $q_2$, then at $q_2$ we have in anticlockwise order $\zg_+$, $\zg$, and then  $\zb$. On the other hand, note that there is a non-zero map $c_1$ from $\P_{(\zg_+,f_{\zg_+})}$ to $\P_{(\zg,f_{\zg})}[1]$ arising from the intersection $q_2$.
So we have a map $e$ from $\P_{(\zg,f_{\zg})}$ to $\P_{(\zb,f_{\zb})}[m-1]$ such that $d=c_1\circ e[1]$. If $m=1$, then the existence of $e$ implies $f_{\zb}(s_{q_{2}}^{\zb})$ should equal to $f_{\zg}(s_{q_{2}}^{\zg})$. And since $\zb\neq\zg_+$, so $\zb$ should be drawn in the picture (see the beginning of this subsection for the rules that what kind of arcs should be drawn in the picture). A contradiction to the assumption of $\zg$ in Figure \ref{Figure:Possible positions of a pre-silting arc} (II).
So $m \geq 2$. But then $e$ is a non-zero map from $\P_{(\zg,f_{\zg})}$ to $\P_{(\zb,f_{\zb})}[m-1]$ with $m-1 > 0$, which contradicts to the fact that $(\zg,f_\zg)$ and $(\zb,f_\zb)$ both belong to the silting dissection $(\zD,f)$.

To sum up, we have proved  $\Hom_{\cald^b(A)}(\P_{(\zg_+,f_{\zg_+})},\P_{(\zb,f_\zb)}[m])=0$, for any $(\zb,f_\zb)$ in $(\zD,f)$ for the second case.

Case 3. For the $\zg$ in Figure \ref{Figure:Possible positions of a pre-silting arc} (III), note that $\mu^+_{(\zg,f_\zg)}((\Delta,f))=(\Delta,f)\setminus \{(\zg,f_\zg)\}\sqcup \{(\zg,f_\zg[1])\}$, so $\Hom_{\cald^b(A)}(\P_{(\zg_+,f_{\zg_+})},\P_{(\zb,f_\zb)}[m])=0$ for any $(\zb,f_\zb)\in(\zD,f)$ and $m\geq 2$. By the assumption of $\zg$ in the Figure \ref{Figure:Possible positions of a pre-silting arc} (III), there also exists no non-zero map from $\P_{(\zg,f_{\zg})}$ to $\P_{(\zb,f_{\zb})}$ for any $(\zb,f_\zb)$ in $(\zD,f)$, and thus no non-zero map from $\P_{(\zg_+,f_{\zg_+})}$ ($=\P_{(\zg,f_{\zg})}[1]$) to $\P_{(\zb,f_{\zb})}[1]$.

So $\Hom_{\cald^b(A)}(\P_{(\zg_+,f_{\zg_+})},\P_{(\zb,f_\zb)}[m])=0$ and thus $\Hom_{\cald^b(A)}(\P_{(\za,f_\za)},\P_{(\zb,f_\zb)}[m])=0$ for any $(\za,f_\za)$ and $(\zb,f_\zb)$ in $(\zD,f)$.
\end{proof}

\subsection{The compatibility of mutations}\label{The compatibility of mutations}
Now we have two kinds of mutations, the mutation of silting objects in the derived category $K^b(\proj A)$ and the mutation of silting dissections in the corresponding geometric model $(S,M,\Delta_A)$. The two mutations fit into the following diagram,
\begin{equation}\label{equation:compatibilitymutation}
\xymatrix{(\Delta,f)\ar[dd]_{}\ar[rrr]^{\mu^+_{(\zg,f_\zg)}}&&&
\mu^+_{(\zg,f_\zg)}(\Delta,f)\ar[dd]_{}~\\
&&&\\
\P_{(\Delta,f)}\ar[uu]_{}\ar[rrr]^{\mu^+_{\P_{(\zg,f_\zg)}}}&&&
{\P_{\mu^+_{(\zg,f_\zg)}(\Delta,f)}}\ar[uu]_{}}
\end{equation}
where $(\zD,f)$ is a silting dissection containing a graded arc $(\zg,f_{\zg})$, $\mu^+_{(\zg,f_\zg)}$ is the left mutation at $(\zg, f_\zg)$ of the silting dissection defined in Definition \ref{definition:mutation of silting dissections}, $\mu^+_{\P_{(\zg,f_\zg)}}$ is the left silting mutation of the silting object defined in Definition \ref{Definition:mutation of silting objects}, and the vertical correspondence is established in \cite{APS19,O19} with $\P_{(\zg,f_\zg)}$ the pre-silting object associated to $(\zg,f_\zg)$.

We have a  similar diagram with respect to  the right mutation of silting dissections and the right mutation of the corresponding silting objects in the derived category.

The following theorem shows that these two  mutations are compatible with each other. We only state the compatibility of the left mutations. The statement of the compatibility of the right mutations then is the dual statement of  Theorem~\ref{theorem: compatibilitymutations}.

\begin{theorem}[compatibility of mutations]\label{theorem: compatibilitymutations}
 Let $(\zD,f)$ be a silting dissection, and let $(\zg,f_{\zg})$ be a graded arc in it. Denote by $\P_{(\zD,f)}=\bigoplus_{j=1}^n \P_{(\gamma_j, f_{j})}$ the silting object corresponding to $(\zD,f)$. Then
the left mutation of $(\zD,f)$ is compatible with the left mutation of $\P_{(\zD,f)}$, that is, the diagram \eqref{equation:compatibilitymutation} is commutative. Moreover, the exchange triangle of the left silting mutation is given by

(1) the distinguished triangle arising from the homotopy push-out of the maps $a_1$ and $a_2$ given in Proposition \ref{proposition:flip and push-out}(3) for case (I):
\begin{equation*}
\P_{(\zg,f_\zg)}\s{(a_1,a_2)}\longrightarrow \P_{(\zg_1,f_{\zg_1})}\oplus \P_{(\zg_2,f_{\zg_2})}\s{\left(
                       \begin{smallmatrix}
                         b_1 \\
                         -b_2\\
                       \end{smallmatrix}
                     \right)}\longrightarrow \P_{(\zg_+,f_{\zg_+})}\s{c\circ d}\longrightarrow {\P_{(\zg,f_\zg)}[1]};
                     \end{equation*}

(2) the distinguished triangle arising from the map $a_1$ for case (II):
\begin{equation*}
\P_{(\zg,f_\zg)}\s{a_1}\longrightarrow \P_{(\zg_1,f_{\zg_1})}\s{b_1}\longrightarrow \P_{(\zg_+,f_{\zg_+})}\s{c_1}\longrightarrow {\P_{(\zg,f_\zg)}[1]};
\end{equation*}

(3) the trivial distinguished triangle starting from $\P_{(\zg,f_\zg)}$ for case  (III) :
\begin{equation*}
\P_{(\zg,f_\zg)}\s{0}\longrightarrow 0\s{0}\longrightarrow {\P_{(\zg,f_\zg)}[1]}\s{id}\longrightarrow {\P_{(\zg,f_\zg)}[1]}.
\end{equation*}
\end{theorem}
\begin{proof}
We begin by proving part (1). From Proposition \ref{proposition:flip and push-out} (3), we have the distinguished triangles in the  statement. Since in an admissible $\gpoint$-dissection the only intersections of the arcs are at the endpoints, the maps between indecomposable objects in $\P_{(\zD,f)}$ that arise from these endpoint intersections form a  basis of the corresponding morphism spaces.
Thus for any graded arc $(\zb,f_\zb)\in (\zD,f)$ with a non-zero morphism space $\Hom_{\cald^b(A)}(\P_{(\zg,f_\zg)},\P_{(\zb,f_\zb)})$,  $\zb$ has endpoints  $q_1$ or $q_2$ in Figure~\ref{Figure:Possible positions of a pre-silting arc} with $f_\zb(s^{\zb}_{q_1})=f_\zg(s^{\zg}_{q_1})$ or $f_\zb(s^{\zb}_{q_2})=f_\zg({s^{\zg}_{q_2}})$ respectively. Moreover, when $\zb\neq \zg$,
since $\zg_i$ is the first $\gpoint$-arc in $\zD$ anti-clockwise following $\zg$ based at $q_i$ with $f_{\zg_i}(s^{\zg_i}_{q_i})=f_\zg(s^{\zg}_{q_i})$ for $i=1, 2$, each basis map in
$\Hom_{\cald^b(A)}(\P_{(\zg,f_\zg)},\P_{(\zb,f_\zb)})$ (which arises from $q_i$) factors through $a_i$, and thus any map in
$\Hom_{\cald^b(A)}(\P_{(\zg,f_\zg)},\P_{(\zb,f_\zb)})$ factors through $(a_1,a_2)$.
On the other hand, since $\zg_i\neq \zg$, we have $\P_{(\zg_1,f_{\zg_1})}\oplus\P_{(\zg_2,f_{\zg_2})}\in add(\P_{(\zD,f)}\setminus \P_\zg)$.
Therefore, $(a_1,a_2)$ is the left $add(\P_{(\zD,f)}\setminus \P_\zg)$-approximation of $\P_{(\zg,f_\zg)}$, and the corresponding triangle  is the left exchange triangle of the silting mutation of $\P_{(\zD,f)}$ at $\P_{(\zg,f_\zg)}$.

Part (2) follows from a similar argument as in part (1) combined with the argument in the proof of part (3) below.  In this case, $a_1$ is the left $add(\P_{(\zD,f)}\setminus \P_\zg)$-approximation of $\P_{(\zg,f_\zg)}$, and the triangle in part (2) is the exchange triangle of the left silting mutation.

To prove part (3), we show that the morphism space  $\Hom_{\cald^b(A)}(\P_{(\zg,f_\zg)},\P_{(\zb,f_\zb)})$ is zero for any graded arc $(\zb,f_\zb) \in (\zD,f)$ different from $(\zg,f_\zg)$, and thus the triangle is the exchange triangle of the left silting mutation.

For contradiction, assume that $\Hom_{\cald^b(A)}(\P_{(\zg,f_\zg)},\P_{(\zb,f_\zb)})\neq 0$
for some graded arc $(\zb,f_\zb)\in (\zD,f)$, then by the same argument as in part (1),  $\zb$ has endpoints $q_1$ or $q_2$. Without loss of generality, assume that $\zb$ and $\zg$ intersect at $q_1$ and that this intersection gives rise to a non-zero map in  $\Hom_{\cald^b(A)}(\P_{(\zg,f_\zg)},\P_{(\zb,f_\zb)})$. Then $f_{\zg}(s_{q_1}^{\zg})= f_{\zb}(s_{q_1}^{\zb})$. Note that the existence of $\zb$ implies the existence of $\zg_1$ (see how to define $\zg_1$ at the beginning of subsection \ref{Mutation of silting dissections}). Furthermore at $q_1$ we have in  anti-clockwise order $\zg$ followed by $\zg_1$ and then $\zb$. Since $(\zD,f)$ is a silting dissection, we have $f_\zg(s_{q_1}^{\zg})\geq f_{\zg_1}(s_{q_1}^{\zg_1})$ and
$f_{\zg_1}(s_{q_1}^{\zg_1})\geq f_{\zb}(s_{q_1}^{\zb})$.
On the other hand, by assumption the grading of $\gamma_1$ is such that there is no map from $\P_{(\zg,f_\zg)}$ to $\P_{(\zg_1,f_{\zg_1})}$, that is $f_\zg(s_{q_1}^{\zg})\neq f_{\zg_1}(s_{q_1}^{\zg_1})$. Therefore $f_\zg(s_{q_1}^{\zg}) > f_{\zb}(s_{q_1}^{\zb})$. A contradiction.
\end{proof}

The following corollary can directly be  derived from the above theorem.
\begin{corollary}
 The exchange triangle of a left or a right silting mutation has at most two middle terms.
\end{corollary}

In general, a silting mutation of a tilting object is not necessarily another tilting object and it is difficult to give a characterisation of when it is and when it is not. However, in the case of gentle algebras, using the surface model, we can  characterise precisely when a silting mutation of a tilting object is again a tilting object.  We only state the result for the left silting mutation case, the case for right silting mutation is dual. In the below we call a silting dissection $(\zD, f)$ such that the corresponding object in $K^b(\proj A)$ is tilting, a \emph{tilting dissection}.

\begin{proposition}\label{proposition:tilting}
Let $(\zD, f)$ be a tilting dissection and let $(\zg, f_\zg) \in (\zD, f)$.
Then  the left silting mutation $\mu^+_{(\zg,f_\zg)}(\Delta,f)$ is a tilting dissection if and only if
one of the following two conditions holds
\begin{enumerate}
\item the graded arc $(\zg, f_\zg)$ is as in Case I of Figure \ref{Figure:Possible positions of a pre-silting arc};
\item the graded arc $(\zg, f_\zg)$ is as in Case II of Figure \ref{Figure:Possible positions of a pre-silting arc} and no other arc in $\zD$ starts or ends at $q_2$.
\end{enumerate}

If $\zg$ is as in Case III of Figure \ref{Figure:Possible positions of a pre-silting arc} then $\mu^+_{(\zg,f_\zg)}(\Delta,f)$ is a silting dissection which is not tilting.
\end{proposition}
\begin{proof}
First note that a silting dissection $(\zD,f)$ is a tilting dissection if and only if the gradings of any two graded arcs $(\za,f_\za)$ and $(\zb,f_\zb)$ in $(\zD,f)$ are compatible at any common endpoint $q$, that is if $f_{\za}(s_{q}^{\za})=f_{\zb}(s_{q}^{\zb})$.

If $\zg$ is as in Figure \ref{Figure:Possible positions of a pre-silting arc} (I), then by the proof of Theorem \ref{theorem: compatibilitymutations}, we have the equalities $f_{\zg_+}(s_{q'_1}^{\zg_+})=f_{\zg_1}(s_{q'_1}^{\zg_1})$ and $f_{\zg_+}(s_{q'_2}^{\zg_+})=f_{\zg_2}(s_{q'_2}^{\zg_2})$. On the other hand, the gradings of the remaining arcs in $\zD\setminus \{\zg\}$ are not changed. So in the silting dissection $\mu^+_{(\zg,f_\zg)}(\Delta,f)$, the gradings of any two arcs are compatible at the common endpoints. Therefore it is a tilting dissection.

Let $\zg$ be an arc as in Figure \ref{Figure:Possible positions of a pre-silting arc} (II). If, other than $\zg$,  there exists no  arc in $\zD$ starting or ending at  $q_2$  then $\zg_+$ is the only arc in $\mu^+_{(\zg,f_\zg)}(\Delta,f)$ starting or ending at  $q_2$. On the other hand, we have the equality $f_{\zg_+}(s_{q'_1}^{\zg_+})=f_{\zg_1}(s_{q'_1}^{\zg_1})$. Thus the gradings of the arcs are compatible in $\mu^+_{(\zg,f_\zg)}(\Delta,f)$. So  it is a tilting dissection.  If on the other hand, there exists another arc $\zb$ in $\zD$ starting or ending at $q_2$ then we have $f_{\zg_+}(s_{q_2}^{\zg_+})=f_{\zg}(s_{q_2}^{\zg})-1
=f_{\zb}(s_{q_2}^{\zb})-1$ in $\mu^+_{(\zg,f_\zg)}(\Delta,f)$. That is, the gradings of $\zg_+$ and $\zb$ in $\mu^+_{(\zg,f_\zg)}(\Delta,f)$ are not compatible at $q_2$. So the new  silting dissection is not a tilting dissection.

If $\zg$ is as in Figure \ref{Figure:Possible positions of a pre-silting arc} (III), then $(\zg,f_\zg[1])$ belongs to the new dissection $\mu^+_{(\zg,f_\zg)}(\Delta,f)$ and it is clear that  this is not a tilting dissection.
\end{proof}

We now translate the above geometric statement into a purely algebraic characterisation of when the left silting mutation of a tilting object is again  tilting. The case for the right silting mutation of a tilting object can be dually stated.

\begin{corollary}\label{corollary:tilting}
Let $A$ be a gentle algebra. Let
$$\P_{(\zD,f)}=\bigoplus_{(\zg, f_\zg) \in (\zD, f)} \P_{(\zg, f_{\zg})}$$
be a basic tilting object in $K^b(\proj A)$.
Then the left silting mutation  $$\P_{\mu^+_{(\zg,f_\zg)}(\Delta,f)}= \bigoplus_{(\zg', f_{\zg'}) \in (\zD, f), \zg' \neq \zg} \P_{(\zg', f_{\zg'})} \oplus \P_{(\zg_+, f_{\zg_+})}  $$ of  $\P_{(\zD,f)}$ at $\P_{(\zg,f_\zg)}$ is a tilting object if and only if one of the following two conditions  hold
\begin{enumerate}
\item the exchange triangle giving rise to $\P_{\mu^+_{(\zg,f_\zg)}(\Delta,f)}$ has two middle terms. That is, it is of the form
\begin{equation*}
\P_{(\zg,f_\zg)}\longrightarrow \P_{(\zg_1,f_{\zg_1})}\oplus \P_{(\zg_2,f_{\zg_2})}\longrightarrow \P_{(\zg_+,f_{\zg_+})}\longrightarrow {\P_{(\zg,f_\zg)}[1]};
                     \end{equation*}
 for  $\P_{(\zg_1, f_{\zg_1})}$ and $\P_{(\zg_2, f_{\zg_2})}$ indecomposable summands of $\P_{(\zD,f)}$.
\item the exchange triangle giving rise to $\P_{\mu^+_{(\zg,f_\zg)}(\Delta,f)}$ has one middle term, that is,  it is of the form
\begin{equation*}
\P_{(\zg,f_\zg)}\s{a}\longrightarrow \P_{(\zg_1,f_{\zg_1})}\longrightarrow \P_{(\zg_+,f_{\zg_+})}\longrightarrow {\P_{(\zg,f_\zg)}[1]},
\end{equation*}  for some indecomposable summand $ \P_{(\zg_1, f_{\zg_1})}$ of $\P_{(\zD,f)}$ and there exists no non-zero morphism from any indecomposable summand in $\P_{(\zD,f)}\setminus \P_{(\zg,f_\zg)}$ to $\P_{(\zg,f_\zg)}$, or if such map $b$ exists, then $ba\neq 0$.
\end{enumerate}

If the exchange triangle giving rise to $\P_{\mu^+_{(\zg,f_\zg)}(\Delta,f)}$ has zero middle term, that is, if it is of the form
\begin{equation*}
\P_{(\zg,f_\zg)}\s{0}\longrightarrow 0\s{0}\longrightarrow {\P_{(\zg,f_\zg)}[1]}\s{id}\longrightarrow {\P_{(\zg,f_\zg)}[1]},
\end{equation*}
then $\P_{\mu^+_{(\zg,f_\zg)}(\Delta,f)}$ is not a tilting object.
\end{corollary}

\begin{proof}
The corollary is just a restatement of  Proposition~\ref{proposition:tilting} using Theorem~\ref{theorem: compatibilitymutations}. Note that in the second case, the condition in the corollary coincides with the condition that there exists no arc in $\zD$ with $q_2$ as an endpoint except for $\zg$ and that this is independent of whether  $\zg$ is a loop or not.
\end{proof}

\section{Silting reduction}\label{section:Silting reductions}

In this section, we give a geometric interpretation of the silting reduction at a pre-silting subcategory arising from an indecomposable pre-silting object in the subcategory of perfect objects of the bounded derived category of a gentle algebra. We do this in terms of cutting the surface associated to the gentle algebra along the arc corresponding to this indecomposable pre-silting object.

For this we fix the following set-up. Let $A$ be a gentle algebra with associated   graded marked surface $(S,M,\zD_A)$. Denote by $\zD^*_A$ the dual admissible $\rpoint$-dissection of $\zD_A$, and by $\calk$ the homotopy category $K^b(\proj A)$. We assume throughout this section that $\zg$  is a $\gpoint$-arc without self-intersections, that is, $\zg$ does not have any self-intersections neither in the  interior of $S$ nor at its endpoints.  Then it follows from  the description of the maps in $\calk$, which we recall in subsection \ref{subsection:derived categories of gentle algebras}, that $\P_{(\zg,f_\zg)}$ is an indecomposable pre-silting object. We set $\calp=add(\P_{(\zg,f_\zg)})$ to be the pre-silting subcategory of $\calk$ generated by $\P_{(\zg,f_\zg)}$.

\subsection{Cutting the surface}\label{subsection:cutting the surface}

\begin{definition}[the cut marked surface]
\label{definition:cut surface}
Let $(S,M)$ be a marked surface. Let $\zg$ be an $\gpoint$-arc in $S$ without self-intersections.

(1) The \emph{cut surface} $S_{\zg}$  is obtained from $S$ by cutting along $\zg$ and contracting the new boundary segments along $\zg$.

(2) We define the \emph{cut marked surface}  $(S_\zg,M_\zg)$ as  follows: $P^{\rpoint}_\zg=P^{\rpoint}$,  $M_{\zg}^{{\rpoint}}
=M^{\rpoint}$, and $M_\zg^{{\gpoint}}$ is obtained from $M^{\gpoint}$ in the following way. Set
$M_\zg^{{\gpoint}}=M^{\gpoint}\setminus \{p,q\}\cup \{pq,p'q'\}$, where $p=p', q=q'$, and $pq$ and $p'q'$  are new $\gpoint$-marked points obtained by cutting along $\zg$ and then contracting, as described in Figures \ref{figure:cut 1} and
\ref{figure:cut 2}. Figure \ref{figure:cut 1}
shows the case when the endpoints of $\zg$ lie on
two different boundary components and Figure \ref{figure:cut 2} shows the case when the endpoints of $\zg$ lie on the same boundary component.
\end{definition}

\begin{figure}[H]
\begin{center}
\begin{tikzpicture}[scale=0.3]
\begin{scope}
	\draw (0,0) circle (2cm);
	\clip[draw] (0,0) circle (2cm);
	\foreach \x in {-2.5,-2,-1.5,-1,-0.5,0,0.5,1,1.5,2,2.5}	\draw[xshift=\x cm]  (-3,3)--(3,-3);
\end{scope}
\begin{scope}
    \draw (0,-8) circle (2cm);
    \clip[draw] (0,-8) circle (2cm);
	\foreach \x in {-2.5,-2,-1.5,-1,-0.5,0,0.5,1,1.5,2,2.5}
	\draw[xshift=\x cm]  (-3,-5)--(3,-11);
\end{scope}
\begin{scope}
    \draw (12,-4) circle (3cm);
    \clip[draw] (12,-4) circle (3cm);
	\foreach \x in {-4,-3.5,-3,-2.5,-2,-1.5,-1,-0.5,0,0.5,1,1.5,2,2.5,3,3.5}
	\draw[xshift=\x cm]  (16,-9)--(-2,16);
\end{scope}
    \draw (0,-2) -- (0,-6);
    \draw (0,-2) node {$\gpoint$};
    \draw (0,-6) node {$\gpoint$};
    \draw (-1,-4) node {$\zg$};
    \draw (-0.8,-2.7) node {$p$};
    \draw (0.8,-2.5) node {$p'$};
    \draw (-0.8,-5.5) node {$q$};
    \draw (0.8,-5.3) node {$q'$};
    \draw (0,3) node {$b_1$};
    \draw (0,-11) node {$b_2$};
    \draw (12,0) node {$b_1$};
    \draw (12,-8) node {$b_2$};
    \draw (9,-4) node {$\gpoint$};
    \draw (15,-4) node {$\gpoint$};
    \draw (7.5,-4) node {$pq$};
    \draw (16.5,-4) node {$p'q'$};
\end{tikzpicture}
\end{center}
\begin{center}
\caption{By cutting at  $\zg$ in $S$ as in the left picture, in the  cut marked surface the two boundary components in $S$ are replaced by  one boundary component with two new marked points.}\label{figure:cut 1}
\end{center}
\end{figure}
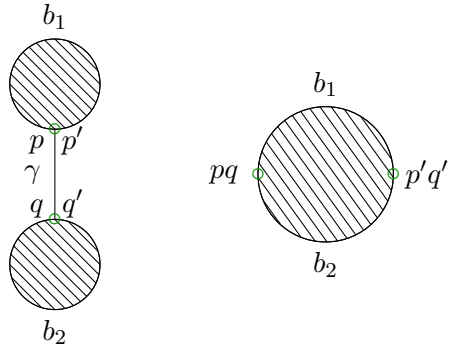

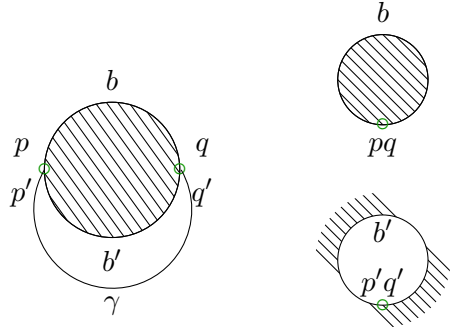
\begin{figure}[H]
\begin{center}
\begin{tikzpicture}[scale=0.3]
\begin{scope}
	\draw (0,0) circle (2cm);
	\clip[draw] (0,0) circle (2cm);
	\foreach \x in {-2.5,-2,-1.5,-1,-0.5,0,0.5,1,1.5,2,2.5}	\draw[xshift=\x cm]  (-3,3)--(3,-3);
\end{scope}
\begin{scope}[even odd rule]
    \draw (0,-8) circle (2cm);
    \clip[clip] (0,-8) circle (2cm) circle (3cm);
	\foreach \x in {-2.5,-2,-1.5,-1,-0.5,0,0.5,1,1.5,2,2.5}
	\draw[xshift=\x cm]  (-3,-5)--(3,-11);
\end{scope}
\begin{scope}
    \draw (-12,-4) circle (3cm);
    \clip[draw] (-12,-4) circle (3cm);
	\foreach \x in {-4,-3.5,-3,-2.5,-2,-1.5,-1,-0.5,0,0.5,1,1.5,2,2.5,3,3.5}
	\draw[xshift=\x cm]  (-8,-9)--(-26,16);
\end{scope}
    \draw (0,-2) node {$\gpoint$};
    \draw (0,-10) node {$\gpoint$};
    \draw (0,-3) node {$pq$};
    \draw (0,-6.7) node {$b'$};
    \draw (0,3) node {$b$};
    \draw (0,-9) node {$p'q'$};
    \draw (-12,0) node {$b$};
    \draw (-12,-8) node {$b'$};
    \draw (-15,-4) node {$\gpoint$};
    \draw (-9,-4) node {$\gpoint$};
    \draw (-16,-3) node {$p$};
    \draw (-16,-5) node {$p'$};
    \draw (-8,-3) node {$q$};
    \draw (-8,-5) node {$q'$};

\draw (-15,-4) arc (150:390:3.5);
    \draw (-12,-10) node {$\zg$};
\end{tikzpicture}
\end{center}
\begin{center}
\caption{If the arc  $\zg$ in $S$ is as in the picture on the left, in the cut marked surface, we obtain two boundary components, each with one new marked point.}\label{figure:cut 2}
\end{center}
\end{figure}

Then similarly to the discussion in  \cite[Proposition 1.11]{APS19}, we have  several cases (note that in our set-up there are no green punctures   in  the initial  dissection of the marked surface, so some of the cases discussed in \cite{APS19} do not appear here). As shown in \cite[Proposition 1.11]{APS19}, for all cases, there are exactly  \( |M^{\gpoint}|+|P|+b+2g-3 \) arcs in an admissible~$\gpoint$-dissection of $(S_\zg,M_\zg)$, where $g$ is the genus of $S$ and $b$ is the number of connected components of $\partial S$.

 \smallskip

\noindent \emph{Case 1:~$\gamma$ is a non-separating curve starting and ending on two different boundary components.}
 In this case,~$(S_\zg,M_\zg)$ is a marked surface with~$|M^{\gpoint}|$ marked points~$\gpoint$,~$|P|$ punctures,
 ~$b-1$ boundary components, and genus~$g$.
 \smallskip

 \noindent \emph{Case 2:~$\gamma$ is a non-separating curve starting and ending on the same boundary component.}
 In this case,~$(S_\zg,M_\zg)$ is a marked surface with~$|M^{\gpoint}|$ ~$\gpoint$-marked points,~$|P|$ punctures,
 ~$b+1$ boundary components, and genus~$g-1$.

 \smallskip

 \noindent \emph{Case 3:~$\gamma$ is a separating curve starting and ending on the same boundary component.}
 In this case,~$(S_\zg,M_\zg)$ is a disjoint union of two marked surfaces, with a total of~$|M^{\gpoint}|$ ~$\gpoint$-marked points,~$|P|$ punctures,
 ~$b+1$ boundary components, and such that the sum of their genuses is equal to $g$.

 \smallskip

Note that the assumption that $\zg$ is not a loop ensures that the cut marked surface $(S_\zg,M_\zg)$  has  no $\gpoint$-punctures.

\subsection{Identifying arcs in the cut surface}\label{subsection:identify arcs by cutting}
Recall that we always assume the curves are in minimal position with regards to the initial $\rpoint$-dissection. All curves which intersect $\zg$ in the interior of $S$ will disappear after cutting. In particular, this holds for any closed curve $\alpha$.  Namely, either $\alpha$ intersects $\gamma$ and then there is no corresponding curve in the cut surface or $\alpha$ does not intersect $\gamma$ and then there is a unique corresponding closed curve in $(S_\zg,M_\zg)$ (up to homotopy). The case of admissible arcs is more complicated, since distinct admissible arcs in $(S,M)$ might be identified in $(S_\zg,M_\zg)$. In order to explicitly describe this, we define an equivalence relation on admissible arcs as follows.

\begin{definition}
Let $\za$ and $\zb$ be two admissible arcs intersecting $\gamma$ at most in their endpoints. We write $\za\mathop{\sim}\limits_{q}^{\zg}\zb$, if $\za$ and $\zg$ intersect at a $\gpoint$-point $q$ and $\zb$ is obtained from $\za$ by  smoothing the crossing with $\zg$ at $q$. We write  $\za\s{\zg}\sim \zb$ if $\za$ and $\zb$ are \emph{$\zg$-smoothing equivalent}, that is, $\zb$ can be obtained from $\za$ by iterated smoothings with $\zg$ at either endpoint of $\gamma$.
\end{definition}

\begin{remark}
Assume $p$ and $q$ are the endpoints of $\zg$. Note that if $\za\mathop{\sim}\limits_{q}^{\zg}\zb$, where $\za$ and $\zg$ intersect  at $q$. Then $\zb$ intersects $\gamma$ at $p$ and $\zb\mathop{\sim}\limits_{p}^{\zg}\za$.
It is then straightforward to verify that
the $\zg$-smoothing equivalence is an equivalence relation on  the  set of admissible arcs on $(S,M)$ that intersect $\gamma$ at most in their endpoints.
\end{remark}

Let $\za$ be an admissible arc, we denote the  $\zg$-smoothing equivalence class of $\za$ by $[\za]_\zg$.
Note that $[\za]_\zg$ has only one element if and only if $\za$ does not intersect with $\zg$ at the boundary of $S$.
The following lemma is a direct consequence of the definitions.

\begin{lemma}\label{lemma:smoothing equivalent-identify}
Two admissible arcs $\za$ and $\zb$ in $(S,M)$,  which intersect $\gamma$ at most in their endpoints, are identified in $(S_\zg,M_\zg)$ if and only if they are in the same $\zg$-smoothing equivalence class.
\end{lemma}

\begin{lemma}\label{lemma:position}
Let $\za$ be an admissible arc in $S$ which intersects  $\zg$ only in either one or both  of its endpoints. Then the following hold.

(1) We have three  cases of relative positions of $\za$ and $\zg$ with respect to their intersection. They are depicted in Figure \ref{figure:possible positions between a and g}, where $\za$ can be equal to any of the $\za_i$, for  $1\leq i\leq4$, where all the triangles are distinguished triangles on the surfaces. In case I, the black bullet $q_3$ corresponds to either a $\gpoint$-marked point or a $\rpoint$-puncture.

(2) Any two admissible arcs intersecting with $\zg$ at their endpoints are $\zg$-smoothing equivalent if and only if they both correspond to one of the arcs $\za_i$ in the same case of Figure \ref{figure:possible positions between a and g}.

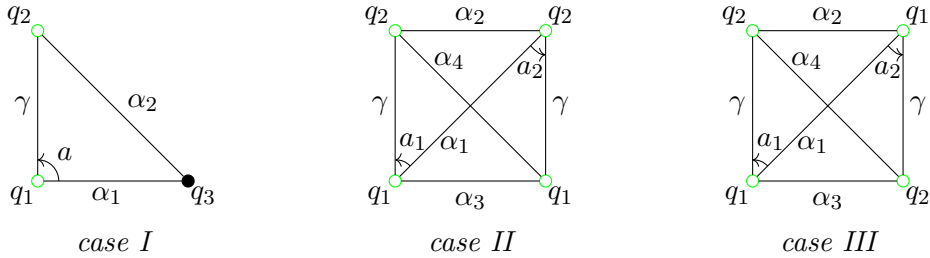
\begin{figure}[H]
\begin{center}
{\begin{tikzpicture}[scale=0.4]
\draw[green] (0,0) circle [radius=0.2];
\draw[green] (0,5) circle [radius=0.2];
\draw[fill] (5,0) circle [radius=0.2];

\draw[-] (0,0.2) -- (0,4.8);

\draw[-] (0.2,0) -- (4.8,0);

\draw[-] (0.15,4.85) -- (4.85,0.15);

\draw[] (3.5,2.6) node {$\za_2$};

\draw (2.3,-.5) node {$\za_1$};

\draw (-.5,2.5) node {$\zg$};

\draw (-.6,5.5) node {$q_2$};
\draw (-.5,-.5) node {$q_1$};
\draw (5.5,-.5) node {$q_3$};

\draw (0.9,0.9) node {$a$};
\draw[->] (0.7,0) to [out=90,in=0] (0,0.7);

\draw (2.5,-2) node {case I};
\draw (2.5,-3.5) node {};
\end{tikzpicture}}
\qquad\qquad
{\begin{tikzpicture}[scale=0.4]
\draw[green] (0,0) circle [radius=0.2];
\draw[green] (0,5) circle [radius=0.2];
\draw[green] (5,5) circle [radius=0.2];
\draw[green] (5,0) circle [radius=0.2];

\draw[-] (0,0.2) -- (0,4.8);
\draw[-] (5,0.2) -- (5,4.8);

\draw[-] (0.2,0) -- (4.8,0);
\draw[-] (0.2,5) -- (4.8,5);

\draw[-] (0.15,0.15) -- (4.85,4.85);
\draw[-] (0.15,4.85) -- (4.85,0.15);

\draw (2,1.2) node {$\za_1$};
\draw[] (1.8,4) node {$\za_4$};
\draw[] (2.5,5.5) node {$\za_2$};
\draw[] (2.5,-.65) node {$\za_3$};

\draw (5.5,2.5) node {$\zg$};
\draw (-.5,2.5) node {$\zg$};

\draw (-.6,5.5) node {$q_2$};
\draw (-.5,-.5) node {$q_1$};
\draw (5.5,-.5) node {$q_1$};
\draw (5.5,5.5) node {$q_2$};

\draw (0.6,1.3) node {$a_1$};
\draw (4.5,3.7) node {$a_2$};
\draw[->] (0.5,0.5) to [out=135,in=0] (0,0.7);
\draw[->] (4.5,4.5) to [out=315,in=180] (5,4.2);

\draw (2.5,-2) node {case II};
\draw (2.5,-3.5) node {};
\end{tikzpicture}}
\qquad\qquad
{\begin{tikzpicture}[scale=0.4]
\draw[green] (0,0) circle [radius=0.2];
\draw[green] (0,5) circle [radius=0.2];
\draw[green] (5,5) circle [radius=0.2];
\draw[green] (5,0) circle [radius=0.2];

\draw[-] (0,0.2) -- (0,4.8);
\draw[-] (5,0.2) -- (5,4.8);

\draw[-] (0.2,0) -- (4.8,0);
\draw[-] (0.2,5) -- (4.8,5);

\draw[-] (0.15,0.15) -- (4.85,4.85);
\draw[-] (0.15,4.85) -- (4.85,0.15);

\draw (2,1.2) node {$\za_1$};
\draw[] (1.8,4) node {$\za_4$};
\draw[] (2.5,5.5) node {$\za_2$};
\draw[] (2.5,-.65) node {$\za_3$};

\draw (5.5,2.5) node {$\zg$};
\draw (-.5,2.5) node {$\zg$};

\draw (-.6,5.5) node {$q_2$};
\draw (-.5,-.5) node {$q_1$};
\draw (5.5,-.5) node {$q_2$};
\draw (5.5,5.5) node {$q_1$};

\draw (0.6,1.3) node {$a_1$};
\draw (4.5,3.7) node {$a_2$};
\draw[->] (0.5,0.5) to [out=135,in=0] (0,0.7);
\draw[->] (4.5,4.5) to [out=315,in=180] (5,4.2);

\draw (2.5,-2) node {case III};
\draw (2.5,-3.5) node {};
\end{tikzpicture}}
\end{center}
\begin{center}
\caption{Possible positions of an arc $\za_i$ intersecting $\zg$ at its endpoints.}\label{figure:possible positions between a and g}
\label{figure:possible intersections}
\end{center}
\end{figure}
\end{lemma}
\begin{proof}
(1) For an arc $\zb$, denote by nep($\zb$) the number of distinct endpoints of $\zb$. That is, nep($\zb$)$=1$ if $\zb$ starts and ends at the same point and nep($\zb$)$=2$ if $\zb$ has distinct endpoints. Note that we have nep($\zg)=2$, since in this section we assume that  $\gamma$ is not a loop.
We list in the following all the possibilities of the relative positions of $\za$ and $\zg$ case by case. In Figure (1.2), the fat black point is used to illustrate that $\za$ and $\zg$ are not homotopic, while the  the fat black points in Figures (2.2) and (2.2)$^\prime$ are used to illustrate that $\za$ is not contractible.

1. nep($\za$)=2.
\begin{center}
{\begin{tikzpicture}[scale=0.5]
	\draw[green] (0,0) circle [radius=0.1];
	\draw[green] (5,0) circle [radius=0.1];
	\draw[fill] (5,-3) circle [radius=0.1];
	\draw[](0.1,0)to(4.9,0);
	\draw (2.5,.5) node {$\zg$};	
	\draw[](0.1,-0.05)to(4.93,-2.93);
	\draw[blue](5,-0.05)to(5,-2.93);
	\draw[blue] (5.5,-1.5) node {$\za_1$};	
	\draw (2.5,-2.2) node {$\za$};	
	
	\draw[bend right](-.2,-.6)to(-.2,.6);
	\draw[-]  (-0.1,-.33)to(-.4,-.6);
	\draw[-]  (-.05,-.1)to(-.4,-.4);
	\draw[-]  (-.05,.1)to(-.4,-.2);
	\draw[-]  (-.05,.3)to(-.4,0);
	\draw[-]  (-.1,.45)to(-.4,0.2);

	\draw[bend left](5.2,-.6)to(5.2,.6);
	\draw[-]  (5.35,-.35)to(5.15,-.5);
	\draw[-]  (5.35,-.15)to(5.1,-.35);
	\draw[-]  (5.35,.05)to(5.05,-.2);
	\draw[-]  (5.35,.25)to(5.05,0);
	\draw[-]  (5.35,.45)to(5.05,0.2);
\draw (6,-3) node {$q_3$};
\draw (-1,0) node {$q_1$};
\draw (6,0) node {$q_2$};

	\draw (2.5,-5) node {(1.1)};
	\end{tikzpicture}}
\qquad\qquad
{\begin{tikzpicture}[scale=0.5]
	\draw[green] (0,0) circle [radius=0.1];
	\draw[green] (5,0) circle [radius=0.1];
	\draw[fill] (5,-3) circle [radius=0.1];
	\draw[](0.1,0)to(4.9,0);
	\draw (2.5,.5) node {$\zg$};	
	\draw[blue](0.1,-0.05)to(4.93,-2.93);
	\draw[blue] (2.5,-2.2) node {$\za_1$};	
		\draw[](5,-0.05)to(5,-2.93);
	\draw[] (5.5,-1.5) node {$\za$};	

	\draw[bend right](-.2,-.6)to(-.2,.6);
	\draw[-]  (-0.1,-.33)to(-.4,-.6);
	\draw[-]  (-.05,-.1)to(-.4,-.4);
	\draw[-]  (-.05,.1)to(-.4,-.2);
	\draw[-]  (-.05,.3)to(-.4,0);
	\draw[-]  (-.1,.45)to(-.4,0.2);

	\draw[bend left](5.2,-.6)to(5.2,.6);
	\draw[-]  (5.35,-.35)to(5.15,-.5);
	\draw[-]  (5.35,-.15)to(5.1,-.35);
	\draw[-]  (5.35,.05)to(5.05,-.2);
	\draw[-]  (5.35,.25)to(5.05,0);
	\draw[-]  (5.35,.45)to(5.05,0.2);
\draw (6,-3) node {$q_3$};
\draw (-1,0) node {$q_1$};
\draw (6,0) node {$q_2$};

	\draw (2.5,-5) node {(1.1)$^\prime$};
	\end{tikzpicture}}
\qquad\qquad
{\begin{tikzpicture}[scale=0.8]
	\draw[green] (0,0) circle [radius=0.1];
	\draw[green] (5,0) circle [radius=0.1];
	\draw[fill] (2.5,0) circle [radius=0.15];
	\draw[](0,0)..controls (1.5,2.5) and (3.5,2.5)..(5,0);
	\draw (2.5,2.2) node {$\zg$};	
	\draw[](0,0)..controls (1.5,-2.5) and (3.5,-2.5)..(5,0);
	\draw[blue] (1.3,.9) node {$\za_1$};
	\draw[blue] (3.7,-.9) node {$\za_2$};
	\draw[blue] (3.8,.15) node {$\za_3$};
	\draw (2.5,-2.2) node {$\za$};

	\draw[bend right](-.2,-.6)to(-.2,.6);
	\draw[-]  (-0.1,-.33)to(-.4,-.6);
	\draw[-]  (-.05,-.1)to(-.4,-.4);
	\draw[-]  (-.05,.1)to(-.4,-.2);
	\draw[-]  (-.05,.3)to(-.4,0);
	\draw[-]  (-.1,.45)to(-.4,0.2);

	\draw[bend left](5.2,-.6)to(5.2,.6);
	\draw[-]  (5.35,-.35)to(5.15,-.5);
	\draw[-]  (5.35,-.15)to(5.1,-.35);
	\draw[-]  (5.35,.05)to(5.05,-.2);
	\draw[-]  (5.35,.25)to(5.05,0);
	\draw[-]  (5.35,.45)to(5.05,0.2);

	\draw[blue](0,0)..controls (4.5,2.5) and (4.5,-2.5)..(0,0);
	\draw[blue](5,0)..controls (.5,2.5) and (.5,-2.5)..(5,0);
	\draw[blue]plot [smooth,tension=1] coordinates {(0,0) (1.5,-.3) (2.5,-.4) (3,0) (2.5,.3) (2.1,0) (2.5,-.4) (3.5, -.3) (5,0)};

\draw (-1,0) node {$q_1$};
\draw (6,0) node {$q_2$};

	\draw (2.5,-3.5) node {(1.2)};
	\end{tikzpicture}}
\end{center}

\begin{center}
{\begin{tikzpicture}[scale=0.35]
			\draw[green] (0,0) circle [radius=0.1];
			\draw[green] (8,0) circle [radius=0.1];

			\draw[bend right](-.3,-1)to(-.3,1);
			\draw[-]  (-.1,-.5)to(-.6,-1);
			\draw[-]  (0,-.2)to(-.6,-.8);
			\draw[-]  (-.05,-.1)to(-.6,-.6);
			\draw[-]  (-.05,.1)to(-.6,-.4);
			\draw[-]  (-.05,.3)to(-.6,-.2);
			\draw[-]  (-.05,.5)to(-.6,0);
			\draw[-]  (-.15,.6)to(-.6,0.2);
			\draw[-]  (-.15,.8)to(-.6,0.4);
			
			\draw[](7.4,-1)..controls (8.2,-.2) and (8.2,0.2)..(7.4,1);
			\draw[-]  (0+8,-.2)to(-.6+8,-.8);
			\draw[-]  (-.05+8,-.1)to(-.6+8,-.6);
			\draw[-]  (-.05+8,.1)to(-.6+8,-.4);
			\draw[-]  (-.05+8,.3)to(-.6+8,-.2);
			\draw[-]  (-.05+8-.1,.5-.1)to(-.6+8,0);
			\draw[-]  (-.15+8-.08,.6-.08)to(-.6+8,0.2);
			\draw[-]  (-.15+8-.15,.8-.15)to(-.6+8,0.4);
			
			\draw[blue] (6,0) circle [radius=6];
			\draw[blue] (13,0) node {$\za_2$};
			
			\draw[](0,0)..controls (6,4) and (10,3)..(8,0);
			\draw (5,3) node {$\zg$};
			
			\draw[](0,0)..controls (6,-4) and (10,-3)..(8,0);
			\draw (5,-3) node {$\za$};
			
			\draw[blue]plot [smooth,tension=1] coordinates {(8,0) (7.4,1.3) (5.7,0)  (7.4,-1.3) (8,0)};
			\draw[blue] (6.8,0) node {$\za_1$};
			
			\draw[blue]plot [smooth,tension=1] coordinates {(8,0) (8,1.4) (5.7,1.2) (5.7,-1.2) (9,-1) (9,3) (4,4) (0,0)};
			\draw[blue] (7,5) node {$\za_3$};

\draw (-1.2,0) node {$q_1$};
\draw (8.8,0) node {$q_2$};

	        \draw (6,-9) node {(1.3)};
			
			
			\end{tikzpicture}}
\qquad
{\begin{tikzpicture}[scale=0.35]
			\draw[green] (0,0) circle [radius=0.1];
			\draw[green] (8,0) circle [radius=0.1];

			\draw[bend right](-.3,-1)to(-.3,1);
			\draw[-]  (-.1,-.5)to(-.6,-1);
			\draw[-]  (0,-.2)to(-.6,-.8);
			\draw[-]  (-.05,-.1)to(-.6,-.6);
			\draw[-]  (-.05,.1)to(-.6,-.4);
			\draw[-]  (-.05,.3)to(-.6,-.2);
			\draw[-]  (-.05,.5)to(-.6,0);
			\draw[-]  (-.15,.6)to(-.6,0.2);
			\draw[-]  (-.15,.8)to(-.6,0.4);
			
			\draw[](7.4,-1)..controls (8.2,-.2) and (8.2,0.2)..(7.4,1);
			\draw[-]  (0+8,-.2)to(-.6+8,-.8);
			\draw[-]  (-.05+8,-.1)to(-.6+8,-.6);
			\draw[-]  (-.05+8,.1)to(-.6+8,-.4);
			\draw[-]  (-.05+8,.3)to(-.6+8,-.2);
			\draw[-]  (-.05+8-.1,.5-.1)to(-.6+8,0);
			\draw[-]  (-.15+8-.08,.6-.08)to(-.6+8,0.2);
			\draw[-]  (-.15+8-.15,.8-.15)to(-.6+8,0.4);
			
			\draw[] (6,0) circle [radius=6];
			\draw[blue] (13,0) node {$\za_2$};
			
			\draw[](0,0)..controls (6,4) and (10,3)..(8,0);
			\draw (5,3) node {$\za$};
			
			\draw[](0,0)..controls (6,-4) and (10,-3)..(8,0);
			\draw (5,-3) node {$\zg$};
			
			\draw[blue]plot [smooth,tension=1] coordinates {(8,0) (7.4,1.3) (5.7,0)  (7.4,-1.3) (8,0)};
			\draw[blue] (6.8,0) node {$\za_1$};
			
			
			\draw[blue]plot [smooth,tension=1] coordinates {(8,0) (7.8,-1.6) (5.5,-1.2) (5.5,1.4) (9,1.5) (9,-3) (4,-4) (0,0)};
			\draw[blue] (7,-5) node {$\za_3$};

\draw (-1.2,0) node {$q_1$};
\draw (8.8,0) node {$q_2$};

			\draw (6,-9) node {(1.3)$^\prime$};
			
			\end{tikzpicture}}
\end{center}

2. nep($\za$)=1.
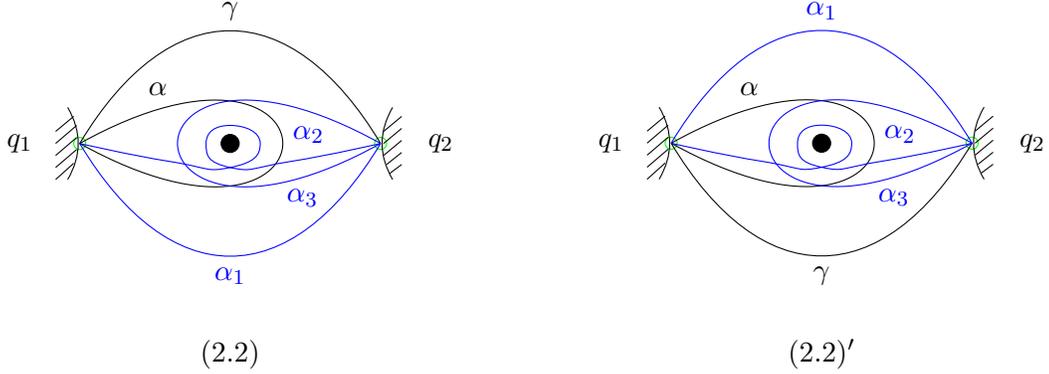
\begin{center}
{\begin{tikzpicture}[scale=0.35]
			\draw[green] (0,0) circle [radius=0.1];
			\draw[green] (8,0) circle [radius=0.1];

			\draw[bend right](-.3,-1)to(-.3,1);
			\draw[-]  (-.1,-.5)to(-.6,-1);
			\draw[-]  (0,-.2)to(-.6,-.8);
			\draw[-]  (-.05,-.1)to(-.6,-.6);
			\draw[-]  (-.05,.1)to(-.6,-.4);
			\draw[-]  (-.05,.3)to(-.6,-.2);
			\draw[-]  (-.05,.5)to(-.6,0);
			\draw[-]  (-.15,.6)to(-.6,0.2);
			\draw[-]  (-.15,.8)to(-.6,0.4);
			
			\draw[](7.4,-1)..controls (8.2,-.2) and (8.2,0.2)..(7.4,1);
			\draw[-]  (0+8,-.2)to(-.6+8,-.8);
			\draw[-]  (-.05+8,-.1)to(-.6+8,-.6);
			\draw[-]  (-.05+8,.1)to(-.6+8,-.4);
			\draw[-]  (-.05+8,.3)to(-.6+8,-.2);
			\draw[-]  (-.05+8-.1,.5-.1)to(-.6+8,0);
			\draw[-]  (-.15+8-.08,.6-.08)to(-.6+8,0.2);
			\draw[-]  (-.15+8-.15,.8-.15)to(-.6+8,0.4);
			
			\draw (6,0) circle [radius=6];
			\draw (13,0) node {$\za$};
			
			\draw[](0,0)..controls (6,4) and (10,3)..(8,0);
			\draw (5,3) node {$\zg$};
			
			\draw[blue](0,0)..controls (6,-4) and (10,-3)..(8,0);
			\draw[blue] (5,-3) node {$\za_1$};
			
			\draw[blue]plot [smooth,tension=1] coordinates {(8,0) (7.4,1.3) (5.7,0)  (7.4,-1.3) (8,0)};
			\draw[blue] (6.8,0) node {$\za_3$};
			
			\draw[blue]plot [smooth,tension=1] coordinates {(8,0) (8,1.4) (5.7,1.2) (5.7,-1.2) (9,-1) (9,3) (4,4) (0,0)};
			\draw[blue] (7,5) node {$\za_2$};

\draw (-1.2,0) node {$q_1$};
\draw (8.8,0) node {$q_2$};

	        \draw (6,-9) node {(2.1)};
			\end{tikzpicture}}
\end{center}
\begin{center}
{\begin{tikzpicture}[scale=0.8]
	\draw[green] (0,0) circle [radius=0.1];
	\draw[green] (5,0) circle [radius=0.1];
	\draw[fill] (2.5,0) circle [radius=0.15];
	\draw(0,0)..controls (1.5,2.5) and (3.5,2.5)..(5,0);
	\draw (2.5,2.2) node {$\zg$};	
	\draw[blue](0,0)..controls (1.5,-2.5) and (3.5,-2.5)..(5,0);
	\draw (1.3,.9) node {$\za$};
	\draw[blue] (3.7,-.9) node {$\za_3$};
	\draw[blue] (3.8,.15) node {$\za_2$};
	\draw[blue] (2.5,-2.2) node {$\za_1$};

	\draw[bend right](-.2,-.6)to(-.2,.6);
	\draw[-]  (-0.1,-.33)to(-.4,-.6);
	\draw[-]  (-.05,-.1)to(-.4,-.4);
	\draw[-]  (-.05,.1)to(-.4,-.2);
	\draw[-]  (-.05,.3)to(-.4,0);
	\draw[-]  (-.1,.45)to(-.4,0.2);

	\draw[bend left](5.2,-.6)to(5.2,.6);
	\draw[-]  (5.35,-.35)to(5.15,-.5);
	\draw[-]  (5.35,-.15)to(5.1,-.35);
	\draw[-]  (5.35,.05)to(5.05,-.2);
	\draw[-]  (5.35,.25)to(5.05,0);
	\draw[-]  (5.35,.45)to(5.05,0.2);

	\draw(0,0)..controls (4.5,2.5) and (4.5,-2.5)..(0,0);
	\draw[blue](5,0)..controls (.5,2.5) and (.5,-2.5)..(5,0);
	\draw[blue]plot [smooth,tension=1] coordinates {(0,0) (1.5,-.3) (2.5,-.4) (3,0) (2.5,.3) (2.1,0) (2.5,-.4) (3.5, -.3) (5,0)};

\draw (-1,0) node {$q_1$};
\draw (6,0) node {$q_2$};

	\draw (2.5,-3.5) node {(2.2)};
	\end{tikzpicture}}
\qquad\qquad
{\begin{tikzpicture}[scale=0.8]
	\draw[green] (0,0) circle [radius=0.1];
	\draw[green] (5,0) circle [radius=0.1];
	\draw[fill] (2.5,0) circle [radius=0.15];
	\draw[blue](0,0)..controls (1.5,2.5) and (3.5,2.5)..(5,0);
	\draw[blue] (2.5,2.2) node {$\za_1$};	
	\draw(0,0)..controls (1.5,-2.5) and (3.5,-2.5)..(5,0);
	\draw (1.3,.9) node {$\za$};
	\draw[blue] (3.7,-.9) node {$\za_3$};
	\draw[blue] (3.8,.15) node {$\za_2$};
	\draw (2.5,-2.2) node {$\zg$};

	\draw[bend right](-.2,-.6)to(-.2,.6);
	\draw[-]  (-0.1,-.33)to(-.4,-.6);
	\draw[-]  (-.05,-.1)to(-.4,-.4);
	\draw[-]  (-.05,.1)to(-.4,-.2);
	\draw[-]  (-.05,.3)to(-.4,0);
	\draw[-]  (-.1,.45)to(-.4,0.2);

	\draw[bend left](5.2,-.6)to(5.2,.6);
	\draw[-]  (5.35,-.35)to(5.15,-.5);
	\draw[-]  (5.35,-.15)to(5.1,-.35);
	\draw[-]  (5.35,.05)to(5.05,-.2);
	\draw[-]  (5.35,.25)to(5.05,0);
	\draw[-]  (5.35,.45)to(5.05,0.2);

	\draw(0,0)..controls (4.5,2.5) and (4.5,-2.5)..(0,0);
	\draw[blue](5,0)..controls (.5,2.5) and (.5,-2.5)..(5,0);
	\draw[blue]plot [smooth,tension=1] coordinates {(0,0) (1.5,-.3) (2.5,-.4) (3,0) (2.5,.3) (2.1,0) (2.5,-.4) (3.5, -.3) (5,0)};

\draw (-1,0) node {$q_1$};
\draw (6,0) node {$q_2$};

	\draw (2.5,-3.5) node {(2.2)$^\prime$};
	\end{tikzpicture}}
\end{center}
Then (1.1), (1.1)$^\prime$ correspond to  case I in Figure~\ref{figure:possible intersections}, (1.3), (1.3)$^\prime$ and (2.1) correspond to  case II in Figure~\ref{figure:possible intersections}, and (1.2), (2.2) and (2.2)$^\prime$ correspond to case III in Figure~\ref{figure:possible intersections}.

(2) We consider the smoothing equivalence class $[\za_1]_\zg$ for each case in Figure \ref{figure:possible positions between a and g}. In case I,  $\zg$ and $\za_1$ only  intersect at $q_1$.
The arc resulting from smoothing the crossing of $\za_1$ and $\zg$ at $q_1$ is the arc $\za_2$. Then $\za_2$  has exactly one intersection  with $\zg$ at  $q_2$ and the smoothing of  $\za_2$ and $\zg$ at $q_2$ is $\za_1$. So $[\za_1]_\zg=\{\za_1,\za_2\}$ for case I.

In case II, $\zg$ and $\za_1$ intersect at two distinct $\gpoint$-points  $q_1$ and $q_2$.
Now we introduce the auxiliary notation of $q_1^\prime$ and $q_2^\prime$, allowing us to distinguish the endpoints on the boundary on  either side $\zg$. More precisely,  we  denote by $q_1^\prime$ and $q_2^\prime$ a new marked  point for each of  $q_1$ and $q_2$, respectively which lie on the `right' of $\gamma$ when looking from the surface to the boundary (analogously to the notation of $p'$ and $q'$ in Figures~\ref{Figure:Possible positions of a pre-silting arc} and \ref{Figure:Possible positions of a pre-silting arc-left}).
Then the arcs $\za_i, 1 \leq i \leq 4$ are connected by $\gamma$-smoothings as follows:
\[\za_1\mathop{\sim}\limits_{q_1}^{\zg}{\za_2}
\mathop{\sim}\limits_{q^\prime_2}^{\zg}
\za_4\mathop{\sim}\limits_{q_2}^{\zg}\za_3
\mathop{\sim}\limits_{q^\prime_1}^{\zg}\za_1.
\]
Note that these are all the possible smoothings of $\za_i, 1 \leq i \leq 4$
with $\zg$. So we have $[\za_1]_\zg=\{\za_1,\za_2,\za_3,\za_4\}$ for this case.

Case III is similar to  case II, namely, we also have $[\za_1]_\zg=\{\za_1,\za_2,\za_3,\za_4\}$.
The $\gamma$-smoothing sequence here is as follows, where we again denote the points $q_1$ and $q_2$ on the right by $q_1^\prime$ and $q_2^\prime$ respectively:
\[\za_1\mathop{\sim}\limits_{q_1}^{\zg}{\za_2}
\mathop{\sim}\limits_{q^\prime_1}^{\zg}
\za_4\mathop{\sim}\limits_{q_2}^{\zg}\za_3
\mathop{\sim}\limits_{q^\prime_2}^{\zg}\za_1.
\]
Then any two admissible arcs intersecting with $\zg$ at the endpoints are $\zg$-smoothing equivalent if and only if they belong to the same picture in Figure \ref{figure:possible positions between a and g}.

\end{proof}
Combining Lemmas \ref{lemma:smoothing equivalent-identify} and \ref{lemma:position}, we summarise the above discussion in the following proposition.

\begin{proposition}[Identifying arcs in the cut marked surface]\label{proposition:identify arcs}
Let $(S_\zg,M_\zg)$ be the cut marked surface of $(S,M)$ with respect to an $\gpoint$-arc $\zg$ without self-intersections. Let $\za$ be an $\gpoint$-arc in $(S,M)$.

(1) If $\za$ and $\zg$ intersect in the interior of $S$ then  $\za$ does not give rise to a curve in $(S_\zg,M_\zg)$.

(2) If $\za$ does not intersect $\zg$, then $\za$ is an arc in $S_\zg$ and there are no other arcs in $S$ which are identified with $\za$ in $(S_\zg,M_\zg)$.

(3) If $\za$ intersects $\zg$ at an endpoint then $\za=\za_i$, for $\za_i$ as in one of the pictures in Figure \ref{figure:possible positions between a and g}, and the other $\za_j$ in the same picture are exactly the arcs which are identified with $\za_i$ in $(S_\zg,M_\zg)$.
\end{proposition}

\subsection{Identifying objects in the silting reduction}\label{subsection:identify objects by reduction}

In the following we will study the silting reduction $\calk/\thick_\calk \calp$ and its geometric realization. Let $\calz=\lp\cap\rp$. Then by Theorem \ref{theorem:equivalence}, the composition $\calz\subset\calk\xrightarrow{\rho}\calk/\thick_\calk\calp$ of natural functors induces an equivalence of triangulated categories:
\begin{equation}\label{equation:natrual functor}
\bar{\rho}\colon{\calz_\calp}\stackrel{\simeq}{\longrightarrow}\calk/\thick_\calk \calp.
\end{equation}
So from now we will not distinguish the two categories and also refer to  ${\calz_\calp}$ as the silting reduction of $\calk$ by $\calp$.
We also recall that by Theorem \ref{theorem:object and map in derived category} (3), the indecomposable objects in $\calk$ only arise from graded $\gpoint$-arcs or graded closed curves, rather than the infinite arcs, which means the endpoint of a graded arc should not be a $\rpoint$-puncture. Recall that for this subsection we have a fixed pre-silting subcategory $\calp = add(\P_{(\zg,f_\zg)})$ where $\gamma$ is an $\gpoint$-arc which is not a loop.

\begin{lemma}[The objects in $\calz$]\label{lemma:the objects in z}
Let $\P_{(\za,f_\za)}$ be an object in $\calz$ corresponding to a graded curve $(\za,f_\za)$, then $\za$ and $\zg$ do not intersect in  the interior of $S$.
\end{lemma}
\begin{proof}
For contradiction, suppose that $\alpha$ and $\zg$ intersect at some point $p$ in the interior of $S$. Without loss of generality we assume that the local configuration of the intersection is as follows. Let  $q_i, 1\leq i\leq 4,$ be the intersections nearest to $p$ in $\za\cap \zD_A^*$,  respectively in $\zg\cap \zD_A^*$.
	
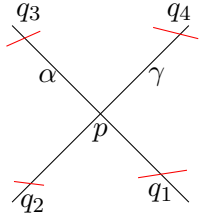
	\begin{center}
		{\begin{tikzpicture}[scale=0.5]
			\draw[-] (0.15,0.15) -- (4.85,4.85);
			\draw[-] (0.15,4.85) -- (4.85,0.15);

			\draw[-,red] (3.5,0.8) -- (4.8,1);
			\draw[-,red] (3.9,4.8) -- (5.1,4.5);
			\draw[-,red] (0,4.3) -- (.9,4.7);
			\draw[-,red] (0.2,.7) -- (1,.6);

			\draw (1.1,3.5) node {$\za$};
			\draw (4,3.5) node {$\zg$};
			
			\draw (4.6,5.2) node {$q_4$};
			\draw (.6,5.2) node {$q_3$};
			\draw (.7,.1) node {$q_2$};
			\draw (4.1,0.4) node {$q_1$};
			\draw (2.5,2) node {$p$};
			\end{tikzpicture}}
	\end{center}
	Then since $\P_{(\za,f_\za)}\in \lp$, we have
$$f_\za(q_1)\geq f_\zg(q_4),~f_\za(q_3)\geq f_\zg(q_2),$$
and since $\P_{(\za,f_\za)}\in \rp$, we have
$$f_\za(q_3)\leq f_\zg(q_4),~f_\za(q_1)\leq f_\zg(q_2).$$
 So we have
$$f_\zg(q_4)\geq f_\za(q_3)\geq f_\zg(q_2)\geq f_\za(q_1)\geq f_\zg(q_4),$$
and thus they are all equal.
A contradiction since we have $f_\za(q_1)=f_\za(q_3)\pm 1$ and $f_\zg(q_4)=f_\zg(q_2)\pm 1$.
\end{proof}
\begin{remark}
We mention that the arcs corresponding to  indecomposable objects in $\calz$ may intersect with $\zg$ at  their endpoints. On the other hand, the inverse statement of the lemma does not hold, that is, even if $\za$ intersects $\zg$ at the boundary of $S$, rather than in the interior of $S$, $\za$ might not give rise to an object in $\calz$ for any grading $f_\za$. Examples of this are  the arcs $\za_2$ and $\za_3$ in Lemma \ref{lemma:reduction2} and Lemma \ref{lemma:reduction3} below.
\end{remark}

The following four lemmas determine when a graded $\gpoint$-arc corresponds to  an object in ${\calz_\calp}$, and describe the $\langle 1 \rangle$-orbits of this object in ${\calz_\calp}$. Each Lemma corresponds to either the case of $\za$ and $\zg$ not intersecting or to one of the three cases in Figure~\ref{figure:possible positions between a and g}.

\begin{lemma}\label{lemma:reduction0}
Let $\za$ be an $\gpoint$-arc which does not intersect $\zg$, then $\P_{(\za,f_\za)}$ is in ${\calz_\calp}$ for any grading $f_\za$. Moreover, the shift functor $\langle 1 \rangle$ in $\calz_\calp$ coincides with $[1]$, and the $\langle 1 \rangle$-orbit of $\P_{(\za,f_\za)}$ in ${\calz_\calp}$ coincides with the $[1]$-orbit of $\P_{(\za,f_\za)}$ in $\calk$.
\end{lemma}
\begin{proof}
By Theorem \ref{theorem:object and map in derived category2} and Remark \ref{remark:object and map in derived category2}, if there are no intersections between $\za$ and $\zg$, then $$\Hom_{\calk}(\P_{(\za,f_\za)},\P_{(\zg,f_\zg)}[i])=0$$
	for any $i\in \mathbb{Z}$ and any grading $f_\za$. So $\P_{(\zg,f_\zg)}\in \calz_\calp$ and $\P_{(\za,f_\za)}\langle i\rangle=\P_{(\za,f_\za)}[i]$ by the definition of $\langle 1\rangle$ in Theorem \ref{theorem:equivalence}. Therefore the $\langle 1 \rangle$-orbit of $\P_{(\za,f_\za)}$ in ${\calz_\calp}$ coincides with the $[1]$-orbit of $\P_{(\za,f_\za)}$ in $\calk$.
\end{proof}

\begin{lemma}\label{lemma:reduction1}
For $\gpoint$-arcs $\za_1$ and $\za_2$ as in case I of Figure \ref{figure:possible positions between a and g}, let $f_{\za_1}$ and $f_{\za_2}$ be the gradings of $\za_1$ and $\za_2$ such that  $f_{\za_1}(s_{q_1}^{\za_1})=f_{\zg}(s_{q_1}^{\zg})$ and $f_{\za_2}(s_{q_2}^{\za_2})=f_{\zg}(s_{q_2}^{\zg})$ respectively.
Then
$$
\P_{(\za_1,f_{\za_1})}[i]\in \calz \Leftrightarrow i\leq 0,$$
$$ {\P_{(\za_2,f_{\za_2})}[i]}\in \calz\Leftrightarrow i\geq 0.
$$
Furthermore, the $\langle 1 \rangle$-orbit of $\P_{(\za_1,f_{\za_1})}$ in $\calz_\calp$ is
\begin{equation*}
	\cdots\s{\langle1\rangle}\longrightarrow\P_{(\za_1,f_{\za_1})}[-2]\s{\langle1\rangle}
	\longrightarrow \P_{(\za_1,f_{\za_1})}[-1]\s{\langle1\rangle}
	\longrightarrow \P_{(\za_1,f_{\za_1})}\s{\langle1\rangle}
	\longrightarrow
\end{equation*}
\begin{equation*}
	{\P_{(\za_2,f_{\za_2})}}\s{\langle1\rangle}
	\longrightarrow \P_{(\za_2,f_{\za_2})}[1]\s{\langle1\rangle}
	\longrightarrow\P_{(\za_2,f_{\za_2})}[2]\s{\langle1\rangle}
	\longrightarrow\cdots.
\end{equation*}
\end{lemma}
\begin{proof}
Let $\za_1$ be an $\gpoint$-arc as in case I of Figure \ref{figure:possible positions between a and g}. Since $f_{\za_1}(s_{q_1}^{\za_1})=f_{\zg}(s_{q_1}^{\zg})$, the intersection of $\za_1$ and $\zg$ at $q_1$ gives rise to a map $a: \P_{(\za_1,f_{\za_1})} \longrightarrow \P_{(\zg,f_{\zg})}$ in $\calk$.
	Since $q_1$ is the unique intersection between $\za_1$ and $\zg$, we have ${\P_{(\za_1,f_{\za_1})}[i]}\in \rp$ for any $i \in \mathbb{Z}$, and ${\P_{(\za_1,f_{\za_1})}[i]}\in \lp$ if and only if $i \leq 0$. Therefore
	$${\P_{(\za_1,f_{\za_1})}[i]}\in \calz\Leftrightarrow i\leq 0. $$
	
	On the one hand, note that the map $a: \P_{(\za_1,f_{\za_1})} \longrightarrow \P_{(\zg,f_{\zg})}$ is a left-$\calp$-approximation of $\P_{(\za_1,f_{\za})}$,  and by Lemma \ref{lemma:triangles as triangles} (3), it determines a distinguished triangle
	\begin{equation*}
		\P_{(\za_1,f_{\za_1})}\s{a}\longrightarrow \P_{(\zg,f_\zg)}\s{}\longrightarrow \P_{(\za_2,f_{\za_2})}\s{}\longrightarrow {\P_{(\za_1,f_{\za_1})}[1]}
	\end{equation*}
	for the grading $f_{\za_2}$ of $\za_2$ such that $f_{\za_2}(s_{q_2}^{\za_2})=f_{\zg}(s_{q_2}^{\zg})$.
	So we have $\P_{(\za_1,f_{\za_1})}\langle1\rangle=\P_{(\za_2,f_{\za_2})}$ in ${\calz_\calp}$.
	
	On the other hand, by a similar argument, we have
	$${\P_{(\za_2,f_{\za_2})}[i]}\in \calz\Leftrightarrow i\geq 0.$$
	Moreover, note that for any $i<0$, we have $\Hom_{\calk}(\P_{(\za_1,f_\za)},\P_{(\zg,f_\zg)}[i])=0$, and for any $i>0$, we have $\Hom_{\calk}(\P_{(\za_2,f_{\za_2})},\P_{(\zg,f_\zg)}[i])=0$.
	Therefore $\P_{(\za_1,f_{\za_1})}[i]\langle1\rangle=\P_{(\za_1,f_{\za_1})}[i+1]$ for any $i<0$, and
	$\P_{(\za_2,f_{\za_2})}[i]\langle1\rangle=\P_{(\za_2,f_{\za_2})}[i+1]$ for any $i>0$.
	
	To sum up, these objects are connected by the shift functor $\langle1\rangle$ as in the statement. Furthermore the orbit uniquely corresponds to the $\zg$-smoothing equivalence class $[\za_1]_\zg=[\za_2]_\zg=\{\za_1,\za_2\}$ which are identified in $S_\zg$ by Proposition \ref{proposition:identify arcs} (3).
\end{proof}
\begin{lemma}\label{lemma:reduction2}
For $\za_1$ as in case II of Figure \ref{figure:possible positions between a and g}, we denote the copy of  $q_i$ on the right by $q^\prime_i$. Let $f_{\za_1}$ be the grading of $\za_1$ such that  $f_{\za_1}(s_{q_1}^{\za_1})=f_{\zg}(s_{q_1}^{\zg})$.
Then $\P_{(\za_1,f_{\za_1})}\in \calz$ and
\begin{itemize}
	\item  if $f_{{\za_1}}(s^{{\za_1}}_{q^\prime_2})=f_{\zg}(s^{\zg}_{q^\prime_2})$, then the $\langle 1 \rangle$-orbit of $\P_{(\za_1,f_{\za_1})}$ is
	\begin{equation*}
		\cdots\s{\langle1\rangle}\longrightarrow\P_{({\za_1},f_{\za_1})}[-2]\s{\langle1\rangle}
		\longrightarrow \P_{({\za_1},f_{\za_1})}[-1]\s{\langle1\rangle}
		\longrightarrow \P_{({\za_1},f_{{\za_1}})}\s{\langle1\rangle}
		\longrightarrow
	\end{equation*}
	\begin{equation*}
		{\P_{(\za_4,f_{\za_4})}}\s{\langle1\rangle}
		\longrightarrow \P_{(\za_4,f_{\za_4})}[1]\s{\langle1\rangle}
		\longrightarrow\P_{(\za_4,f_{\za_4})}[2]\s{\langle1\rangle}
		\longrightarrow\cdots.
	\end{equation*}
	In particular, $\P_{(\za_2,f_{\za_2})}\notin \calz$ and $\P_{(\za_3,f_{\za_3})}\notin \calz$ for any gradings $f_{\za_2}$ and $f_{\za_3}$ of $\za_2$ and $\za_3$ respectively.\\
	\item  if $f_{{\za_1}}(s^{{\za_1}}_{q^\prime_2})- f_{\zg}(s^{\zg}_{q^\prime_2})=m \geq 1$, then the $\langle 1 \rangle$-orbit of $\P_{(\za_1,f_{\za_1})}$ is
	\begin{equation*}
		\cdots\s{\langle1\rangle}\longrightarrow\P_{({\za_1},f_{\za_1})}[-1]\s{\langle1\rangle}
		\longrightarrow \P_{({\za_1},f_{\za_1})}\s{\langle1\rangle}
		\longrightarrow \P_{(\za_2,f_{\za_2})}\s{\langle1\rangle}
		\longrightarrow \P_{(\za_2,f_{\za_2})}[1]\s{\langle1\rangle}
		\longrightarrow\cdots
		\s{\langle1\rangle}\longrightarrow
	\end{equation*}
	\begin{equation*}
		\P_{(\za_2,f_{\za_2})}[m-1]\s{\langle1\rangle}
		\longrightarrow \P_{(\za_4,f_{\za_4})}[m]\s{\langle1\rangle}
		\longrightarrow \P_{(\za_4,f_{\za_4})}[m+1]\s{\langle1\rangle}
		\longrightarrow\cdots.
	\end{equation*}
	In particular, $\P_{(\za_3,f_{\za_3})}\notin \calz$ for any grading $f_{\za_3}$ of $\za_3$.\\
	\item  if $f_{{\za_1}}(s^{{\za_1}}_{q^\prime_2})- f_{\zg}(s^{\zg}_{q^\prime_2})=m \leq -1$, then the $\langle 1 \rangle$-orbit of $\P_{(\za_1,f_{\za_1})}$ is
	\begin{equation*}
		\cdots\s{\langle1\rangle}\longrightarrow\P_{({\za_1},f_{\za_1})}[m-1]\s{\langle1\rangle}
		\longrightarrow \P_{({\za_1},f_{\za_1})}[m]\s{\langle1\rangle}
		\longrightarrow \P_{(\za_3,f_{\za_3})}[m+1]\s{\langle1\rangle}
		\longrightarrow
	\end{equation*}
	\begin{equation*}
		\cdots
		\s{\langle1\rangle}\longrightarrow\P_{(\za_3,f_{\za_3})}
		\s{\langle1\rangle}\longrightarrow \P_{(\za_4,f_{\za_4})}\s{\langle1\rangle}
		\longrightarrow \P_{(\za_4,f_{\za_4})}[1]\s{\langle1\rangle}
		\longrightarrow\cdots.
	\end{equation*}
	In particular, $\P_{(\za_2,f_{\za_2})}\notin \calz$ for any grading $f_{\za_2}$ of $\za_2$.
\end{itemize}
\end{lemma}
\begin{proof}
Let $\za_1$ be an $\gpoint$-arc in case II of Figure \ref{figure:possible positions between a and g}. Since $f_{\za_1}(s_{q_1}^{\za_1})=f_{\zg}(s_{q_1}^{\zg})$, we  have a map $a_1: \P_{({\za_1},f_{{\za_1}})} \longrightarrow \P_{(\zg,f_{\zg})}$, which gives rise to  the distinguished triangle
	\begin{equation*}
		\P_{({\za_1},f_{\za_1})}\s{a_1}\longrightarrow \P_{(\zg,f_\zg)}\s{}\longrightarrow \P_{(\za_2,f_{\za_2})}\s{}\longrightarrow {\P_{({\za_1},f_{\za_1})}[1]},
	\end{equation*}
	where $f_{\za_2}$ is the grading of $\za_2$ such that $f_{\zg}(s^{\zg}_{q_2})=f_{\za_2}(s^{\za_2}_{q_2})$ and $f_{\za_2}(s^{\za_2}_{q^\prime_2})=f_{{\za_1}}(s^{{\za_1}}_{q^\prime_2})-1$.
	Now there are three subcases.
	
	Subcase one: if $f_{{\za_1}}(s^{{\za_1}}_{q^\prime_2})=f_{\zg}(s^{\zg}_{q^\prime_2})$, then there is a triangle
	\begin{equation*}
		\P_{({\za_1},f_{\za_1})}\s{a_2}\longrightarrow \P_{(\zg,f_\zg)}\s{}\longrightarrow \P_{(\za_3,f_{\za_3})}\s{}\longrightarrow {\P_{({\za_1},f_{\za_1})}[1]},
	\end{equation*}
	where $f_{\za_3}$ is the grading of $\za_3$ such that  $f_{\zg}(s^{\zg}_{q^\prime_1})=f_{\za_3}(s^{\za_3}_{q^\prime_1})$ and $f_{\za_3}(s^{\za_3}_{q_1})=f_{{\za_1}}(s^{{\za_1}}_{q_1})-1$.
	Suppose $\P_{(\za_2,g_{\za_2})}\in \calz$ for some grading $g_{\za_2}$. Then we have $g_{\za_2}(s^{\za_2}_{q_2})\leq f_{\zg}(s^{\zg}_{q_2})$ by $\P_{(\za_2,g_{\za_2})}\in \rp$, and $g_{\za_2}(s^{\za_2}_{q^\prime_2})\geq f_{\zg}(s^{\zg}_{q^\prime_2})=f_{\zg}(s^{\zg}_{q_2})$ by $\P_{(\za_2,g_{\za_2})}\in \rp$. So
	$$g_{\za_2}(s^{\za_2}_{q_2})\leq g_{\za_2}(s^{\za_2}_{q^\prime_2}).$$
	On the other hand, we have $$f_{\za_2}(s^{\za_2}_{q_2})=f_{\zg}(s^{\zg}_{q_2})=f_{\zg}(s^{\zg}_{q^\prime_2})=f_{{\za_1}}(s^{{\za_1}}_{q^\prime_2})
	=f_{\za_2}(s^{\za_2}_{q^\prime_2})+1.$$
	So
	$$g_{\za_2}(s^{\za_2}_{q_2})=g_{\za_2}(s^{\za_2}_{q^\prime_2})+1,$$
	a contradiction. Therefore  $\P_{(\za_2,g_{\za_2})}\notin \calz$ for any grading $g_{\za_2}$, and similarly $\P_{(\za_3,g_{\za_3})}\notin \calz$ for any grading $g_{\za_3}$.
	
	Noticing that by Proposition \ref{proposition:flip and push-out} (3),
	\begin{equation*}
		\P_{({\za_1},f_{\za_1})}\s{(a_1,a_2)}\longrightarrow \P_{(\zg,f_{\zg})}\oplus \P_{(\zg,f_{\zg})}
		\longrightarrow \P_{(\za_4,f_{\za_4})}
		\longrightarrow {\P_{({\za_1},f_{\za_1})}[1]}
	\end{equation*}
	is a distinguished triangle for some proper grading $f_{\za_4}$ of $\za_4$, where $(a_1,a_2)$ is a left $\P_{(\zg,f_{\zg})}$-approximation. So $\P_{({\za_1},f_{\za_1})}\langle1\rangle=\P_{(\za_4,f_{\za_4})}$.
	On the other hand, $\P_{(\za_4,f_{\za_4})}[i]\in \calz$ if and only if $i\geq 0$, and $\P_{({\za_1},f_{{\za_1}})}[i]\in \calz$ if and only if $i\leq 0$. Finally we have the wanted $\langle1\rangle$-orbit, which uniquely corresponds with the $\gpoint$-arc $\za_1=\za_4$ in $S_\zg$.

	Subcase two: if $f_{{\za_1}}(s^{{\za_1}}_{q^\prime_2})\geq f_{\zg}(s^{\zg}_{q^\prime_2})+1$, assume that $f_{{\za_1}}(s^{{\za_1}}_{q^\prime_2})- f_{\zg}(s^{\zg}_{q^\prime_2})=m \geq 1$.
	On the one hand, for $i\in \mathbb{Z}$,
	$$\P_{(\za_2,f_{\za_2})}[i]\in \rp \Longleftrightarrow f_{\za_2}(s^{\za_2}_{q_2})\leq f_{\zg}(s^{\zg}_{q_2})+i,$$
	$$\P_{(\za_2,f_{\za_2})}[i]\in \lp \Longleftrightarrow f_{\za_2}(s^{\za_2}_{q^\prime_2})\geq f_{\zg}(s^{\zg}_{q^\prime_2})+i.$$
	On the other hand, we have
	$$f_{\za_2}(s^{\za_2}_{q_2})=f_{\zg}(s^{\zg}_{q_2}),$$
	$$f_{\za_2}(s^{\za_2}_{q^\prime_2})= f_{{\za_1}}(s^{{\za_1}}_{q^\prime_2})-1=f_{\zg}(s^{\zg}_{q^\prime_2})+m-1.$$
	So $\P_{(\za_2,f_{\za_2})}[i]\in \calz$ if and only if $0\leq i \leq m-1$.
	
	Note that there is a triangle
	\begin{equation*}
		\P_{({\za_1},f_{\za_1})}\s{a_2}\longrightarrow \P_{(\zg,f_\zg)}[-m]\s{}\longrightarrow \P_{(\za_3,f_{\za_3})}\s{}\longrightarrow {\P_{({\za_1},f_{\za_1})}[1]},
	\end{equation*}
	where $f_{\za_3}$ is a grading such that $f_{\za_3}(s^{\za_3}_{q^\prime_1})=f_{\zg}(s^{\zg}_{q^\prime_1})+m$ and $f_{\za_3}(s^{\za_3}_{q_1})=f_{{\za_1}}(s^{{\za_1}}_{q_1})-1$.
	Then on the one hand, for $i\in \mathbb{Z}$,
	$$\P_{(\za_3,f_{\za_3})}[i]\in \rp \Longleftrightarrow f_{\za_3}(s^{\za_3}_{q^\prime_1})\leq f_{\zg}(s^{\zg}_{q^\prime_1})+i,$$
	$$\P_{(\za_3,f_{\za_3})}[i]\in \lp \Longleftrightarrow f_{\za_3}(s^{\za_3}_{q_1})\geq f_{\zg}(s^{\zg}_{q_1})+i.$$
	So if $\P_{(\za_3,f_{\za_3})}[i]\in \calz$, then
	$$f_{\za_3}(s^{\za_3}_{q^\prime_1})\leq f_{\za_3}(s^{\za_3}_{q_1}).$$
	On the other hand,
	$$f_{\za_3}(s^{\za_3}_{q_1})=f_{{\za_1}}(s^{{\za_1}}_{q_1})-1=f_{\zg}(s^{\zg}_{q_1})-1=f_{\za_3}(s^{\za_3}_{q^\prime_1})-m-1.$$
	So
	$$f_{\za_3}(s^{\za_3}_{q^\prime_1})> f_{\za_3}(s^{\za_3}_{q_1}),$$
	since $m\geq 1.$ Therefore we obtain a contradiction, and $\P_{(\za_3,g_{\za_3})}\notin \calz$ for any grading $g_{\za_3}$.
	Similarly, we may prove that
	$$\P_{({\za_1},f_{{\za_1}})}[i]\in \calz \Longleftrightarrow i\leq 0,$$
	$$\P_{(\za_4,f_{\za_4})}[i]\in \calz \Longleftrightarrow i\geq m.$$
	
Finally we have the wanted $\langle1\rangle$-orbit, which uniquely corresponds with the $\gpoint$-arc $\za_1=\za_2=\za_4$ over $S_\zg$.
	
	Subcase three: $f_{{\za_1}}(s^{{\za_1}}_{q^\prime_2})\leq f_{\zg}(s^{\zg}_{q^\prime_2})-1$, this is the dual case of Subcase two.
\end{proof}

\begin{lemma}\label{lemma:reduction3}
For $\za_1$ as in case III of Figure \ref{figure:possible positions between a and g}, we denote by the copy of $q_i$ on the right by $q^\prime_i$. Let $f_{\za_1}$ be the grading of $\za_1$ such that  $f_{\za_1}(s_{q_1}^{\za_1})=f_{\zg}(s_{q_1}^{\zg})$.
Then $\P_{(\za_1,f_{\za_1})}\in \calz$ and
\begin{itemize}
	\item  if $f_{{\za_1}}(s^{{\za_1}}_{q^\prime_1})=f_{\zg}(s^{\zg}_{q^\prime_1})$, then the $\langle 1 \rangle$-orbit of $\P_{(\za_1,f_{\za_1})}$ is
	\begin{equation*}
		\cdots\s{\langle1\rangle}\longrightarrow\P_{({\za_1},f_{\za_1})}[-2]\s{\langle1\rangle}
		\longrightarrow \P_{({\za_1},f_{\za_1})}[-1]\s{\langle1\rangle}
		\longrightarrow \P_{({\za_1},f_{{\za_1}})}\s{\langle1\rangle}
		\longrightarrow
	\end{equation*}
	\begin{equation*}
		{\P_{(\za_4,f_{\za_4})}}\s{\langle1\rangle}
		\longrightarrow \P_{(\za_4,f_{\za_4})}[1]\s{\langle1\rangle}
		\longrightarrow\P_{(\za_4,f_{\za_4})}[2]\s{\langle1\rangle}
		\longrightarrow\cdots.
	\end{equation*}
	In particular, $\P_{(\za_2,f_{\za_2})}\notin \calz$ and $\P_{(\za_3,f_{\za_3})}\notin \calz$ for any gradings $f_{\za_2}$ and $f_{\za_3}$ of $\za_2$ and $\za_3$ respectively.\\
	\item  if $f_{{\za_1}}(s^{{\za_1}}_{q^\prime_1})- f_{\zg}(s^{\zg}_{q^\prime_1})=m \geq 1$, then the $\langle 1 \rangle$-orbit of $\P_{(\za_1,f_{\za_1})}$ is
	\begin{equation*}
		\cdots\s{\langle1\rangle}\longrightarrow\P_{({\za_1},f_{\za_1})}[-1]\s{\langle1\rangle}
		\longrightarrow \P_{({\za_1},f_{\za_1})}\s{\langle1\rangle}
		\longrightarrow \P_{(\za_2,f_{\za_2})}\s{\langle1\rangle}
		\longrightarrow \P_{(\za_2,f_{\za_2})}[1]\s{\langle1\rangle}
		\longrightarrow\cdots
		\s{\langle1\rangle}\longrightarrow
	\end{equation*}
	\begin{equation*}
		\P_{(\za_2,f_{\za_2})}[m-1]\s{\langle1\rangle}
		\longrightarrow \P_{(\za_4,f_{\za_4})}[m]\s{\langle1\rangle}
		\longrightarrow \P_{(\za_4,f_{\za_4})}[m+1]\s{\langle1\rangle}
		\longrightarrow\cdots.
	\end{equation*}In particular, $\P_{(\za_3,f_{\za_3})}\notin \calz$ for any grading $f_{\za_3}$ of $\za_3$.\\
	\item  if $f_{{\za_1}}(s^{{\za_1}}_{q^\prime_1})- f_{\zg}(s^{\zg}_{q^\prime_1})=m \leq -1$, then the $\langle 1 \rangle$-orbit of $\P_{(\za_1,f_{\za_1})}$ is
	\begin{equation*}
		\cdots\s{\langle1\rangle}\longrightarrow\P_{({\za_1},f_{\za_1})}[m-1]\s{\langle1\rangle}
		\longrightarrow \P_{({\za_1},f_{\za_1})}[m]\s{\langle1\rangle}
		\longrightarrow \P_{(\za_3,f_{\za_3})}[m+1]\s{\langle1\rangle}
		\longrightarrow
	\end{equation*}
	\begin{equation*}
		\cdots
		\s{\langle1\rangle}\longrightarrow\P_{(\za_3,f_{\za_3})}
		\s{\langle1\rangle}\longrightarrow \P_{(\za_4,f_{\za_4})}\s{\langle1\rangle}
		\longrightarrow \P_{(\za_4,f_{\za_4})}[1]\s{\langle1\rangle}
		\longrightarrow\cdots.
	\end{equation*}
	In particular, $\P_{(\za_2,f_{\za_2})}\notin \calz$ for any grading $f_{\za_2}$ of $\za_2$.\\
\end{itemize}
\end{lemma}

\begin{proof}
The proof is similar to the proof of Lemma \ref{lemma:reduction2}.
\end{proof}

Combining  Lemmas \ref{lemma:reduction0},
 \ref{lemma:reduction1},  \ref{lemma:reduction2},
  \ref{lemma:reduction3} and Proposition \ref{proposition:identify arcs},	we obtain the following.

\begin{theorem}\label{theorem:reduction0}
Let $\P_{(\za,f_\za)}$ and $\P_{(\zb,f_\zb)}$ be two indecomposable objects in $\calz_\calp$ corresponding to  $\gpoint$-arcs $\za$ and $\zb$ in $S$ respectively. Then the following are equivalent
\begin{enumerate}
\item $\P_{(\za,f_\za)}$ and $\P_{(\zb,f_\zb)}$ are in the same $\langle 1 \rangle$-orbit in ${\calz_\calp}$,
\item $\za$ and $\zb$ are in the same $\zg$-smoothing equivalence class,
\item $\za$ and $\zb$ are identified in $(S_\zg,M_\zg)$.
\end{enumerate}
\end{theorem}

\subsection{A geometric model of silting reduction}\label{subsection:the geometric model of silting reduction}

Recall that we denote by  $\overline{\calz_\calp}$ the orbit category  of the subfactor category $\calz_{\calp}$ by the shift functor, see  subsection \ref{subsection:}.
We will show in the following that the cut marked surface $(S_\zg,M_\zg)$ gives a geometric model of $\overline{\calz_\calp}$.
\begin{definition}
We call a non-zero indecomposable object $\overline{X}$ in $\overline{\calz_\calp}$ a string (resp. band) object, if its representative $X$ in $\calz$ is a string (resp. band) object.
\end{definition}

Note that the above definition is well-defined and any indecomposable object in $\overline{\calz_\calp}$ is either a string object or a band object.

\begin{definition}
Define
$$\Phi : \overline{\calz_\calp} \to (S_\zg,M_\zg)$$
$$\overline{X}\mapsto \za_{\overline{X}}$$
 as the composition of the three correspondences in the following diagram, given by first lifting a non-zero object $\overline{X}\in \overline{\calz_\calp}$ to an object $X\in \calz$ , then identifying it with the corresponding  graded curve $(\za_X,f_X)$ on $(S,M,\Delta_A)$, and finally with the associated  curve $\za_X$ in $(S_\zg,M_\zg)$.
\begin{center}
{\begin{tikzpicture}[scale=0.4]
\draw (-4,0) node {$\calz$};
\draw (6,0) node {$\overline{\calz_\calp}$};
\draw (-4,-5) node {$(S,M)$};
\draw (6,-5) node {$(S_\zg,M_\zg)$};

\draw (-4,1.5) node {$X$};
\draw (6,1.5) node {$\overline{X}$};
\draw (1,0.8) node {$\Phi_1$};
\draw (-4,-6.2) node {$(\za_X,f_{\za_X})$};
\draw (6,-6.3) node {$\za_{\overline{X}}=\za_X$};
\draw (1,-5.8) node {$\Phi_3$};

\draw[->] (-4,-0.5) -- (-4,-4.5);
\draw[dashed,->] (6,-0.5) -- (6,-4.5);
\draw (6.8,-2.5) node {$\Phi$};
\draw (-4.8,-2.5) node {$\Phi_2$};
{\sib \draw[dotted, ->] (-3.3,0) -- (4.5,0);}
\draw[->] (-2,-5) -- (3.2,-5);
\end{tikzpicture}}
\end{center}
\end{definition}

For a non-zero object $\overline{X}$ in $\overline{\calz_\calp}$, by Lemma \ref{lemma:the objects in z}, there is no interior intersection between $\za_X$ and $\zg$, and thus $\za_X$ still exists in $S_\zg$. So $\Phi$ establishes a correspondence from
the non-zero objects in $\overline{\calz_\calp}$ to $\gpoint$-arcs or closed curves on $S_\zg$. On the other hand, note that there are many choices for the representative $X$ of $\overline{X}$. However, we now  show, that the composition $\Phi$ is independent of the choice of representative and therefore it is a  well-defined map.

\begin{lemma}\label{lemma:correspondence between objects and arcs}
The map $\Phi$ is a well-defined injection from the set of isomorphism classes of indecomposable objects in $\overline{\calz_\calp}$ to the set of $\gpoint$-arcs and closed curves in $(S_\zg,M_\zg)$ restricting to a bijection from the set of isomorphism classes of indecomposable string objects in $\overline{\calz_\calp}$ to the set of  $\gpoint$-arcs  in $(S_\zg,M_\zg)$.
\end{lemma}

\begin{proof}
For a non-zero indecomposable band object $\overline{X}$ in $\overline{\calz_\calp}$, any representative $X$ is a non-zero object in ${\calz_\calp}$. So by Lemma \ref{lemma:the objects in z}, there is no interior intersection between $\za_X$ and $\zg$, and thus no intersection at all, since $\za_X$ is a closed curve. Thus any two distinct lifts of $\overline{X}$ are associated to the same closed curve $\za_X$ in $S_\zg$. Therefore $\Phi$ is well-defined and injective on the band objects.

On the other hand, Theorem \ref{theorem:reduction0} implies that $\Phi$ is well-defined and injective on the indecomposable string objects.
Furthermore, note that any $\gpoint$-arc in $S_\zg$ is the image of some $\gpoint$-arc in $S$, which is always gradable and gives us an indecomposable string object in $\overline{\calz_\calp}$. So the correspondence $\Phi$ is surjective, and thus bijective, on the indecomposable string objects.
\end{proof}

It follows from the proof above that any closed curve in $(S_\zg,M_\zg)$ can be lifted to a unique closed curve in $(S,M)$.  We define the \emph{ winding number} of a closed curve in $(S_\zg,M_\zg)$ to be zero if the corresponding closed curve in $(S,M,\zD_A)$ has winding number zero.

\begin{theorem}\label{theorem:cutsurface}
Let $A$ be a gentle algebra  with associated surface model $(S,M,\zD_A)$. Set $\calp=add(\P_{(\zg,f_\zg)})$ to be the pre-silting subcategory in $K^b(\proj A)$ arising from a graded $\gpoint$-arc $(\zg,f_\zg)$, where $\zg$ is an $\gpoint$-arc without self-intersections. Let ${\calz_\calp}$ be the silting reduction of $K^b(\proj A)$ by $\calp$, and denote by $\overline{\calz_\calp}$ the orbit category of $\calz_\calp$ by the shift functor. Then
\begin{enumerate}
\item
there is a bijection from the set of isomorphism classes of indecomposable objects in $\overline{\calz_\calp}$ to the set of $\gpoint$-arcs and pairs consisting of a closed curve of winding number zero in $(S_\zg,M_\zg)$ and  a non-zero element in the base field $k$;
\item
under the correspondence in (1), the dimensions of the morphism spaces in $\overline{\calz_\calp}$ coincide with the number of oriented intersections of the corresponding $\gpoint$-arcs and  closed curves  in $(S_\zg,M_\zg)$.  That is, for any indecomposable objects $\overline{X},\overline{Y}\in{{\calz_\calp}}$,
$$\dim\Hom_{\overline{\calz_\calp}}
(\overline{X},\overline{Y})=|\Int_{S_\zg}(\za_X,\za_Y)|,$$
where $\alpha_X$ and $\alpha_Y$ are $\gpoint$-arcs or closed curves corresponding to $\overline{X}$ and $\overline{Y}$ respectively. \end{enumerate}
\end{theorem}

\begin{proof}
 The proof of the first part for $\gpoint$-arcs directly follows from Lemma \ref{lemma:correspondence between objects and arcs}. For the case of  closed curves, it is enough to remark that  any band object in $\overline{\calz_\calp}$ uniquely corresponds to a pair consisting of a gradable  primitive closed curve and an indecomposable $k[x,x^{-1}]$-module, which in turn by Theorem \ref{theorem:object and map in derived category} corresponds to a pair consisting of a closed curve of winding number zero and a non-zero element in $k$. We now prove the second part.
Let $\overline{X}$ and $\overline{Y}$ be two indecomposable objects in $\overline{\calz_\calp}$.

Assume first that  the associated curves $\za_X$ and $\za_Y$ do not intersect $\zg$.  Then there are no morphisms between $\P_{(\zg,f_\zg)}$ and $X[i]$, and between $\P_{(\zg,f_\zg)}$ and $Y[i]$, for any $i \in \mathbb{Z}$.
Moreover, by Lemma \ref{lemma:reduction0}, on the objects arising from $\za_X$ and $\za_Y$ the shift functor $\langle 1 \rangle$ in $\calz_\calp$ coincides with $[1]$. Thus
$$\Hom_{{\calz_\calp}}(X,Y\langle i \rangle)=\Hom_\calk(X,Y[i])$$
for any $i \in \mathbb{Z}.$ So
$$\Hom_{\overline{\calz_\calp}}(\overline{X},\overline{Y})
=\Hom_{\calk}^{\bullet}(X,Y)$$
and, in particular,
$$\dim\Hom_{\overline{\calz_\calp}}(\overline{X},\overline{Y})
=\dim\Hom_{\calk}^{\bullet}(X,Y).$$
Furthermore, the intersections between $\za_X$ and $\za_Y$ remain unchanged after cutting the surface, so
$$|\Int_{S}(\za_X,\za_Y)|
=|\Int_{S_\zg}(\za_X,\za_Y)|.$$
On the other hand, by Theorem \ref{theorem:object and map in derived category2} (2),
$\dim\Hom_{\calk}^{\bullet}(X,Y)=|\Int_{S}(\za_X,\za_Y)|$.
Therefore we have
$$\dim\Hom_{\overline{\calz_\calp}}(\overline{X},\overline{Y})
=|\Int_{S_\zg}(\za_X,\za_Y)|.$$

 Now assume that at least one of $\za_X$ and $\za_Y$ intersects $\zg$ but that their intersection with  $\zg$ is only at their endpoints. Then  $\za_X$ and $\za_Y$ are both $\gpoint$-arcs, and there are several possibilities depending on the relative positions of $\za_X$, $\za_Y$ and $\zg$, see Figure \ref{figure:possible positions between a and g}. In the following we give a detailed proof of the case when both $\za_X$ and $\za_Y$ are  arcs such as in case I of Figure \ref{figure:possible positions between a and g}, that is,  both of them intersect  $\zg$. The other cases are proved in a similar way.

 Case I of Figure \ref{figure:possible positions between a and g} splits into two subcases as described in Figure \ref{figure:morphisms}.
  For convenience, we will denote $\za_X$ by $\za$ and denote $\za_Y$ by $\za^\prime$. We also assume that $X=\P_{(\za,f_\za)}$ and $Y=\P_{(\za^\prime,f_{\za^\prime})}$, where $f_\za$ and $f_{\za^\prime}$ are gradings of $\za$ and $\za'$ respectively  such that we have the following maps $a:\P_{(\za,f_\za)}\rightarrow \P_{(\za^\prime,f_{\za^\prime})}$ and $b:\P_{(\za^\prime,f_{\za^\prime})}\rightarrow \P_{(\zg,f_\zg)}$ in subcase I, and maps $a:\P_{(\za,f_\za)}\rightarrow \P_{(\zg,f_\zg)}$ and $b:\P_{(\za^\prime,f_{\za^\prime})}\rightarrow \P_{(\zg,f_\zg)}$ in subcase II.

\begin{figure}[H]
\begin{center}
{\begin{tikzpicture}[scale=0.4]
\draw[green] (0,0) circle [radius=0.2];
\draw[green] (0,5) circle [radius=0.2];
\draw[green] (5,0) circle [radius=0.2];
\draw[green] (5,5) circle [radius=0.2];

\draw[-] (0,0.2) -- (0,4.8);
\draw[-,blue] (0.2,5) -- (4.8,5);
\draw[-] (0.2,0) -- (4.8,0);

\draw[-,blue] (0.15,4.85) -- (4.85,0.15);
\draw[-] (0.15,0.15) -- (4.85,4.85);

\draw[blue] (4.5,1.6) node {$\zb$};
\draw (5.7,3.6) node {$\za^\prime {\sib (= \za_Y})$};
\draw (2.3,-.5) node {$\za {\sib (= \za_X)}$};
\draw[blue] (2.3,5.6) node {$\zb^\prime$};
\draw (-.5,2.5) node {$\zg$};

\draw (-.6,5.5) node {$q_2$};
\draw (-.5,-.5) node {$q_1$};
\draw (5.5,-.5) node {$q_3$};
\draw (5.5,5.5) node {$q_4$};

\draw (0.5,1.1) node {$b$};
\draw (1.1,0.5) node {$a$};

\draw[->] (0.5,0.5) to [out=135,in=0] (0,0.7);
\draw[->](.7,0) to [out=90,in=-45] (.5,.5);

\draw (2.5,-2) node {subcase I};
\end{tikzpicture}}
\quad\quad
{\begin{tikzpicture}[scale=0.4]
\draw[green] (0,0) circle [radius=0.2];
\draw[green] (0,5) circle [radius=0.2];
\draw[green] (5,0) circle [radius=0.2];
\draw[green] (5,5) circle [radius=0.2];

\draw[-,blue] (0,0.2) -- (0,4.8);
\draw[-,blue] (5,0.2) -- (5,4.8);
\draw[-] (0.2,5) -- (4.8,5);
\draw[-] (0.2,0) -- (4.8,0);
\draw[-] (0.15,0.15) -- (4.85,4.85);

\draw (2.3,-.5) node {$\za {\sib (= \za_X)}$};
\draw (2.3,5.6) node {$\za^\prime {\sib( = \za_Y)}$};
\draw (2,2.5) node {$\zg$};
\draw[blue] (-.5,2.5) node {$\zb^\prime$};
\draw[blue] (5.5,2.5) node {$\zb$};
\draw (-.6,5.5) node {$q_4$};
\draw (-.5,-.5) node {$q_1$};
\draw (5.5,-.5) node {$q_3$};
\draw (5.5,5.5) node {$q_2$};

\draw (1.1,0.5) node {$a$};
\draw (3.9,4.5) node {$b$};
\draw (0.5,1.1) node {$c$};

\draw[->] (0.5,0.5) to [out=135,in=0] (0,0.7);
\draw[->](.7,0) to [out=90,in=-45] (.5,.5);
\draw[->](4.3,5) to [out=270,in=135] (4.5,4.5);
\draw (2.5,-2) node {subcase II};
\end{tikzpicture}}
\end{center}
\begin{center}
\caption{$\gpoint$-arcs $\za$ and $\za^\prime$ which intersect with $\zg$ at endpoints. The points $q_3$ and $q_4$ may coincide.}\label{figure:morphisms}
\end{center}
\end{figure}
Recall from the Lemma \ref{lemma:reduction1}, that the $\langle1\rangle$-orbits of $\P_{(\za,f_{\za})}$ and $\P_{(\za',f_{\za'})}$ are respectively
\begin{equation*}
\cdots\s{\langle1\rangle}
\longrightarrow \P_{(\za,f_\za)}[-1]\s{\langle1\rangle}
\longrightarrow \P_{(\za,f_{\za})}\s{\langle1\rangle}
\longrightarrow {\P_{(\zb,f_{\zb})}}\s{\langle1\rangle}
\longrightarrow \P_{(\zb,f_{\zb})}[1]\s{\langle1\rangle}
\longrightarrow\cdots
\end{equation*}
and
\begin{equation*}
\cdots\s{\langle1\rangle}
\longrightarrow \P_{(\za^\prime,f_{\za^\prime})}[-1]\s{\langle1\rangle}
\longrightarrow \P_{(\za^\prime,f_{\za^\prime})}\s{\langle1\rangle}
\longrightarrow {\P_{(\zb^\prime,f_{\zb^\prime})}}\s{\langle1\rangle}
\longrightarrow \P_{(\zb^\prime,f_{\zb^\prime})}[1]\s{\langle1\rangle}
\longrightarrow\cdots.
\end{equation*}

Subcase I. We have the following
$$\begin{array}{rcl}
 & &\Hom_{\overline{\calz_\calp}}(\overline{\P_{(\za,f_{\za})}},\overline{\P_{(\za^\prime,f_{\za^\prime})}})\\
 & = &
 \underset{i=-\infty}{\overset{\infty}{\bigoplus}}\Hom_{{\calz_\calp}}(\P_{(\za,f_{\za})},\P_{(\za^\prime,f_{\za^\prime})}\langle i \rangle) \\
& = & \underset{i=-\infty}{\overset{0}{\bigoplus}}\Hom_{{\calz_\calp}}(\P_{(\za,f_{\za})},\P_{(\za^\prime,f_{\za^\prime})}[i])\oplus
\underset{i=0}{\overset{\infty}{\bigoplus}}\Hom_{{\calz_\calp}}(\P_{(\za,f_{\za})},\P_{(\zb^\prime,f_{\zb^\prime})}[i]) \\
& = & \underset{i=-\infty}{\overset{0}{\bigoplus}}\Hom_{\calk}(\P_{(\za,f_{\za})},\P_{(\za^\prime,f_{\za^\prime})}[i])\oplus
\underset{i=0}{\overset{\infty}{\bigoplus}}\Hom_{\calk}(\P_{(\za,f_{\za})},\P_{(\zb^\prime,f_{\zb^\prime})}[i]).
\end{array}$$

The last equality holds since we know from Figure \ref{figure:morphisms} that in this case no morphism in the two  sums in the last line  will  factor through an object in $\calp$.

 We claim that there is a bijection between the space
$$
\underset{i=0}{\overset{\infty}{\bigoplus}}\Hom_{\calk}(\P_{(\za,f_{\za})},
\P_{(\zb^\prime,f_{\zb^\prime})}[i])
$$
and the space
$$
\underset{i=1}{\overset{\infty}{\bigoplus}}\Hom_{\calk}(\P_{(\za,f_{\za})},
\P_{(\za^\prime,f_{\za^\prime})}[i]),
$$
and thus there is a bijection between the last sum of spaces in the above equation and the space
$$\underset{i=-\infty}{\overset{\infty}{\bigoplus}}
\Hom_{\calk}(\P_{(\za,f_{\za})},\P_{(\za^\prime,f_{\za^\prime})}[i]).$$

Thus finally $$\dim\Hom_{\overline{\calz_\calp}}(\overline{\P_{(\za,f_{\za})}},
\overline{\P_{(\za^\prime,f_{\za^\prime})}})
=\dim\Hom_{\calk}^{\bullet}(\P_{(\za,f_{\za})},
\P_{(\za^\prime,f_{\za^\prime})}).$$

To prove this claim it is enough to show that for any $i\geq 0$, there is an one-to-one correspondence between $\Hom_{\calk}(\P_{(\za,f_{\za})},\P_{(\zb^\prime,f_{\zb^\prime})}[i])$ and $\Hom_{\calk}(\P_{(\za,f_{\za})},\P_{(\za^\prime,f_{\za^\prime})}[i+1])$. However, this directly follows from the following two observations:

\medskip

(1) By assumption $\zb^\prime$ is the smoothing of $\za^\prime$ and $\zg$. Thus if  $\za$ intersects  $\za^\prime$ at some point $_\za p_{\za^\prime}$  different from $q_1$
then $\za$ must also intersect $\zb^\prime$  at some point $_\za p_{\zb^\prime}$. Conversely, every intersection of $\za$ with $\zb'$  implies that there is an intersection of $\za$ with $\za'$ distinct from $q_1$. Thus we have an one-to-one correspondence between the following sets, which maps $_\za p_{\za^\prime}$ to $_\za p_{\zb^\prime}$:
\begin{center}
\{oriented intersections  from $\za$ to $\za^\prime$  excepting the intersection at $q_1$\}

 $\stackrel{1-1}{\longleftrightarrow}$ \{oriented intersections from $\za$ to $\zb^\prime$\}.
\end{center}

(2)  Note that, if we denote by   $s^{\zb^\prime}_{q_4}$, respectively $s^{\za^\prime}_{q_4}$,  the first intersection of $\zb$, respectively of $\za$, with $\Delta^*$, then $f_{\zb^\prime}(s^{\zb^\prime}_{q_4})=f_{\za^\prime}(s^{\za^\prime}_{q_4})-1$. This follows from the fact that   $\triangle(\zg,\zb',\za')$ is a distinguished triangle and from the existence of the map $b$ in the left picture of Figure \ref{figure:morphisms}. So we have $f_{\zb^\prime}(_\za p_{\zb^\prime})=f_{\za^\prime}(_\za p_{\za^\prime})-1$, and $_\za p_{\zb^\prime}$ induces a map from $\P_{(\za,f_{\za})}$ to $\P_{(\zb^\prime,f_{\zb^\prime})}[i]$ if and only if
$_\za p_{\za^\prime}$ induces a map from $\P_{(\za,f_{\za})}$ to $\P_{(\za,f_{\za^\prime})}[i+1].$
On the other hand, the intersection of $\za$ and $\za'$ at $q_1$, gives us a map in $\Hom_{\calk}(\P_{(\za,f_{\za})},\P_{(\za^\prime,f_{\za^\prime})})$. However, note that we only consider the space $\Hom_{\calk}(\P_{(\za,f_{\za})},\P_{(\za^\prime,f_{\za^\prime})}[i+1])$ for $i\geq 0$ in the claim.

To sum up, we have
$$\dim\Hom_{\overline{\calz_\calp}}(\overline{X},\overline{Y})
=\dim\Hom_{\calk}^{\bullet}(X,Y)=|\Int_{S}(\za,\za^\prime)|.$$
On the other hand, note that in this case we have
$$|\Int_{S}(\za,\za^\prime)|=|\Int_{S_{\zg}}(\za,\za^\prime)|.$$ Therefore
$$\dim\Hom_{\overline{\calz_\calp}}(\overline{X},\overline{Y})=|\Int_{S_{\zg}}(\za,\za^\prime)|.$$

Subcase II. Similar to the subcase I, we have equalities
$$\begin{array}{rcl}
& &\Hom_{\overline{\calz_\calp}}(\overline{\P_{(\za,f_{\za})}},\overline{\P_{(\za^\prime,f_{\za^\prime})}})\\
& = &
\underset{i=-\infty}{\overset{\infty}{\bigoplus}}\Hom_{{\calz_\calp}}(\P_{(\za,f_{\za})},\P_{(\za^\prime,f_{\za^\prime})}\langle i \rangle) \\
& = & \underset{i=-\infty}{\overset{0}{\bigoplus}}\Hom_{{\calz_\calp}}(\P_{(\za,f_{\za})},\P_{(\za^\prime,f_{\za^\prime})}[i])\oplus
\Hom_{{\calz_\calp}}(\P_{(\za,f_{\za})},\P_{(\zb^\prime,f_{\zb^\prime})})\oplus \\
&   &
~~\underset{i=1}{\overset{\infty}{\bigoplus}}\Hom_{{\calz_\calp}}(\P_{(\za,f_{\za})},\P_{(\zb^\prime,f_{\zb^\prime})}[i])\\
& = & \underset{i=-\infty}{\overset{0}{\bigoplus}}\Hom_{\calk}(\P_{(\za,f_{\za})},\P_{(\za^\prime,f_{\za^\prime})}[i])\oplus
\Hom_{\calk}(\P_{(\za,f_{\za})},\P_{(\zb^\prime,f_{\zb^\prime})})/\langle ac \rangle\oplus \\
&   &
~~\underset{i=1}{\overset{\infty}{\bigoplus}}\Hom_{\calk}(\P_{(\za,f_{\za})},\P_{(\zb^\prime,f_{\zb^\prime})}[i]).
\end{array}$$
To show the validity of the last equality, one notices that the only intersections between $\za$, $\za'$, $\zb$ and $\zb'$ with $\zg$ are the ending points, so the linear extensions of the composition $ac$ are the only maps in the above Hom-space  which factor through an object in $\calp$.

Then similar to the above statement for case I, for any $i\leq 0$, there is an one-to-one correspondence between $\Hom_{\calk}(\P_{(\za,f_{\za})},\P_{(\zb^\prime,f_{\zb^\prime})}[i])$ and $\Hom_{\calk}(\P_{(\za,f_{\za})},\P_{(\za^\prime,f_{\za^\prime})}[i+1])$.
Therefore we have an one-to-one correspondence between
$\Hom_{\overline{\calz_\calp}}(\overline{\P_{(\za,f_{\za})}},\overline{\P_{(\za^\prime,f_{\za^\prime})}})$
and
$\Hom_{\calk}^{\bullet}(\P_{(\za,f_{\za})},\P_{(\zb^\prime,f_{\zb^\prime})})/\langle ac \rangle.$
Thus
$$
\dim\Hom_{\overline{\calz_\calp}}(\overline{X},\overline{Y})=
|\Int_{S}(\za,\zb^\prime)|-1=|\Int_{S}(\za,\za^\prime)|$$
On the other hand, note that
$|\Int_{S}(\za,\za^\prime)|=|\Int_{S_{\zg}}(\za,\za^\prime)|.$ Therefore
$$\dim\Hom_{\overline{\calz_\calp}}(\overline{X},\overline{Y})=|\Int_{S_{\zg}}(\za,\za^\prime)|.$$

\end{proof}

\end{document}